\documentclass[12pt]{amsart}
\usepackage{pifont}
\usepackage{mathrsfs}
\usepackage{anysize}
\marginsize{2.7cm}{2.44cm}{2.0cm}{3.0cm}
\usepackage[small,it]{caption}
\usepackage[OT2,T1]{fontenc}
\DeclareSymbolFont{cyrletters}{OT2}{wncyr}{m}{n}
\DeclareMathSymbol{\Sha}{\mathalpha}{cyrletters}{"58}
 \usepackage[latin1]{inputenc}
 \usepackage[dvips]{graphicx}
 \usepackage{wrapfig}
 \usepackage{amsmath}
 \usepackage{amsthm}
 \usepackage{amsfonts}
 \usepackage{amssymb}
 \usepackage{amscd}
 \usepackage{layout}
\usepackage{verbatim}
 \usepackage{alltt}
\usepackage{yfonts}
\usepackage{color}

\usepackage[all]{xy}
\usepackage{setspace}
\newtheorem*{ithm}{Theorem 1}
\newtheorem*{jthm}{Theorem 2}
\newtheorem*{hthm}{Theorem 3}

\newtheorem*{claim}{Claim}

\newcommand{\X}{\mathscr{X}}

\newcommand{\OO}{\mathscr{A}}

\newcommand{\K}{{K}} 

\newcommand{\T}{\mathscr{T}}

\newcommand{\HH}{\mathcal{H}}
\newcommand{\hh}{\mathcal{H}^+}

\newcommand{\Q}{\mathbb{Q}_p}

\newcommand{\Z}{\mathbb{Z}_p}

\numberwithin{equation}{section}
\newtheorem{thm}{Theorem}[section]
\newtheorem{lemma}[thm]{Lemma}
\newtheorem{defn}[thm]{Definition}
\newtheorem{prop}[thm]{Proposition}

\newtheorem{conj}[thm]{Conjecture}
\newenvironment{rem}{\par\medskip\noindent\refstepcounter{thm}
\bgroup{\hspace*{-0.15 cm}\bf{Remark} \thethm.}\bgroup}{\egroup
\egroup\par\medskip} \parskip 2pt

\newcounter{Athm}[section]

\renewcommand{\theAthm} {\arabic{Athm}}

\long\def\symbolfootnote[#1]#2{\begingroup%
\def\thefootnote{\fnsymbol{footnote}}\footnote[#1]{#2}\endgroup}

\begin{document}
\title{Iwasawa Main Conjecture for $p$-adic families of elliptic modular cuspforms}

\author{Tadashi Ochiai}


\begin{abstract}
In this article, we discuss an Iwasawa Main Conjecture for $p$-adic families of elliptic modular cuspforms. After an overview in the situation of the ordinary case of Hida families, we recall the Coleman map 
for the non-ordinary case (Coleman families) which was obtained as a joint work with Filippo Nuccio and we give some results on Iwasawa Theory for Coleman families 
as applications of this Coleman map. 
\end{abstract}

\maketitle

\tableofcontents

\section{Introduction}
\label{sec:Intro_motif}

In this paper, we discuss Iwasawa Main Conjectures for various $p$-adic families of elliptic cuspforms. 
The goal is a result in the case of non-ordinary $p$-adic families obtained as an application of 
the construction of Coleman map (see Theorem 3 in this section or Section \ref{section:last} for more precise statements). 
However, we will start from the cyclotomic deformation of an ordinary cuspform. Though this case is rather classical, we 
fill some of missing details and the gap in the standard references. A review of this case is also useful in order to understand better 
the result for non-ordinary $p$-adic families which will be explained later.
\par 
The following result holds true due to contributions by a lot of people (the notation in the statement will be explained in \S \ref{sec:Intro_motif} and 
some ambiguous assumptions below will be precise later in the statement of Theorem C). 
\begin{ithm}[Theorem A, Theorem B and Theorem C] 
Let $f$ be an ordinary elliptic eigen cuspform. 
\begin{enumerate}
\item 
For each choice of $\pm$-complex periods $\{ \Omega^{\pm}_{f,\infty} \}$, 
there exists a $p$-adic $L$-function $L_p (\{ \Omega^{\pm}_{f,\infty} \})$ in the semi-local Iwasawa algebra $\Lambda_{\mathrm{cyc}}\otimes_{\mathbb{Z}_p} K $ which interpolates all critical special values of the Hecke $L$-function $L(f,s)$. Here $K$ denotes a finite extension of $\mathbb{Q}_p$ generated by the Fourier coefficients of $f$. 
\item For each lattice $T$ of the $p$-adic Galois representation $V_f$ associated to $f$, the Pontrjagin dual 
$\mathrm{Sel}_A (\mathbb{Q}(\mu_{p^\infty}))^\vee$ of the Selmer group is a finitely generated torsion $\Lambda_{\mathrm{cyc}}$-module. Here $A$ denotes 
the discrete Galois representation $V_f /T$. 
\item 
Under certain technical assumptions, the equality of the cyclotomic one-variable Iwasawa Main Conjecture modulo $\mu$-invariant for $f$ holds true. That is, we have the following equality between 
principal ideals of $\Lambda_{\mathrm{cyc}}\otimes_{\mathbb{Z}_p} K$
\begin{equation}\label{equation:IMJ1intro}
(L_p (\{ \Omega^{\pm}_{f,\infty} \})) = 
\mathrm{char}_{\Lambda_{\mathrm{cyc}} \otimes_{\mathbb{Z}_p}  {{K}}
} \,  (\mathrm{Sel}_A (\mathbb{Q}(\mu_{p^\infty}))^\vee ) \otimes_{\mathcal{O}_K}  {K}. 
\end{equation}
Under some stronger assumptions, we have a more precise equality of Iwasawa Main Conjecture between 
principal ideals of $\Lambda_{\mathrm{cyc}}\otimes_{\mathbb{Z}_p} \mathcal{O}_K$ 
\end{enumerate}
\end{ithm}
As is explained later, the situation is very much different depending on whether the cuspform $f$ has complex multiplication or not. 
When $f$ has complex multiplication by an imaginary quadratic field $F$, the key ingredients are 
\begin{enumerate}
\item[(a)] 
The $p$-adic $L$-function of $f$ is obtained as a one-variable 
specialization of a two-variable $p$-adic $L$-function constructed by Katz \cite{katz} 
\item[(b)] 
The equality of Iwasawa Main Conjecture follows by 
a one-variable specialization of the two-variable equality of Iwasawa Main Conjecture for $F$ proved by Rubin \cite{419efg}. 
\end{enumerate}
In order to perform (a), we need to calculate the interpolation terms of the $p$-adic $L$-function of Katz and compare them with those appearing in the $p$-adic $L$-function for the modular form $f$. In order to perfoem (b), we need to study Selmer groups carefully by paying attentions to 
pseudo-null submodules and the control theorem. We refer to \cite{haraochiai} for the reference on these process with full generality.  
\par 
For elliptic cuspforms without complex multiplication, Kato \cite{320} made an essential progress 
by proving the inclusion $\subset $ of the equality \eqref{equation:IMJ1intro} by the method of Euler system. The opposite inclusion 
$\supset$ was shown by Skinner-Urban \cite{444sui} by the method of Eisenstein ideal. 
In this paper, we shall explain the method of Euler system in a more detailed manner than the method of Eisenstein ideal because 
this method is related to subsequent contents of the paper. By such a detailed explanation, we can also 
fill some arguments which are missing in the paper \cite{320}. For example, the case of ordinary cuspforms whose level is divisible by $p$ is missing in \cite{320}\footnote{Kato's proof relies on the paper \cite{396} for the Coleman map, where it is assumed that Galois representations are crystalline.} if we wish the totality of the result. Also, there are some subtle points in the formulation of the Cyclotomic Iwasawa Main Conjecture 
for ordinary cuspforms such as the dependence of the Selmer group on the choice of lattices of the Galois representation associated to $f$,
the dependence of the $p$-adic $L$-function on the choice of complex periods. 
We provides an overview paying attentions to such subtle points which are often ignored in the reference. 
\par
The main objectives of the latter half of this paper is to discuss the two-variable Iwasawa Main Conjecture for $p$-adic families of 
elliptic cuspforms by Euler system approach. First, we recall the following theorem on the results in the ordinary case (Hida family) 
obtained by the author in his earlier articles. In order to understand our results in the non-ordinary case which will be given below, 
it will be important to contrast these results with the results and their proofs in the ordinary case (the notation in the statement will be explained 
later).   
\begin{jthm}[Theorem A$'$, Threom B$'$ and Theorem  C$'$]
Let $\mathbb{F}$ be a $p$-adic family of ordinary elliptic eigen cuspforms constructed by Hida over a local domain $\mathbb{I}$ finite flat over $\mathbb{Z}_p [[1+p\mathbb{Z}_p]]$. 
\begin{enumerate}
\item 
For each choice of $\mathbb{I}$-basis 
$\Xi^\pm$ of $\mathbb{MS}(\mathbb{I})^\pm$, there exists an analytic $p$-adic $L$-function $L_p (\{ \Xi^\pm \}) 
\in (\mathbb{I} \widehat{\otimes}_{\mathbb{Z}_p} \Lambda_{\mathrm{cyc}} )
\otimes_{\mathbb{Z}_p} \mathbb{Q}_p$ which interpolates all critical special values of the Hecke $L$-function $L(f,s)$ where $f$ varies in the Hida family $\mathbb{F}$. 
\item 
For each $\mathbb{I}$-lattice $\mathbb{T}$ of the $p$-adic Galois representation $\mathbb{V}$ associated to $\mathbb{F}$, the Pontrjagin dual 
$\mathrm{Sel}_{\mathbb{A}} (\mathbb{Q}(\mu_{p^\infty}))^\vee$ of the Selmer group $\mathrm{Sel}_{\mathbb{A}} (\mathbb{Q}(\mu_{p^\infty}))$
is a finitely generated torsion $\mathbb{I} \widehat{\otimes}_{\mathbb{Z}_p} \Lambda_{\mathrm{cyc}}$-module. Here 
$\mathbb{A}$ denotes the discrete Galois representation $\mathbb{T}\otimes_{\mathbb{I}} \mathbb{I}^\vee$.
\item 
Under certain technical assumptions, the two-variable Iwasawa Main Conjecture for $\mathbb{F}$ holds true. That is, we have the following equality between 
principal ideals of $\mathbb{I} \widehat{\otimes}_{\mathbb{Z}_p} \Lambda_{\mathrm{cyc}}$ 
$$
(L_p (\{ \Xi^\pm \}))=\mathrm{char}_{
\mathbb{I} \widehat{\otimes}_{\mathbb{Z}_p} \Lambda_{\mathrm{cyc}}
} \, \mathrm{Sel}_{\mathbb{A}} (\mathbb{Q}(\mu_{p^\infty}))^\vee . 
$$ 
\end{enumerate}
\end{jthm}
Theorem 2 is a family version of Theorem 1 and the global strategy of the proof of Theorem 2 is similar to that of Theorem 1. 
However, working with $p$-adic families, we have some new phenomena and new technical difficulties which did not arise in Theorem 1. 
For example, in addition to complex periods, $p$-adic periods will play an essential role for Theorem 2. Also the classical 
Euler system machinary does not work in this setting and we need some innovations to overcome these problems. We will come back to these problems in later 
sections. 
\par 
The Iwasawa theory in the case of non-ordinary families (Coleman families) is much different from 
the situation of Theorem 1 and Theorem 2. On the analytic side, the denominators in the interpolation formula are unbounded. 
Hence the $p$-adic $L$-function 
is not an element in the deformation ring anymore as it was in the ordinary case. 
On the algebraic side, the Selmer group is not a torsion-module over the deformation ring. Hence its characteristic ideal is trivial. 
For these reasons, an Iwasawa Main Conjecture which relates the values of $L$-functions to the size of the Selmer group as in Theorem 1 and Theorem 2 
does not make sense in the non-ordinary case. 
\par 
In the non-ordinary case, we generalize the Iwasawa Main Conjecture along with Kato's formulation. 
Our result is summarized as follows (the notation in the statement will be explained later). 
\begin{hthm}[Theorem \ref{theorem:2variablep-adicLnon-ordinary}, Theorem \ref{theorem:eulersystem_Coleman2} and Theorem \ref{theorem:IMC_nonord }]
Let $\mathbb{F}$ be a $p$-adic family of non-ordinary elliptic eigen cuspforms of fixed slope $s \in \mathbb{Q}_{\geq 0}$
constructed by Coleman over a local domain $\Lambda_{(k_0 ;r),\mathcal{O}_{{K}}}$ which is the ring of rigid analytic functions 
on the open disc of radius $r$ centered at $k_0$ in the weight space\footnote{The ring is isomorphic to 
the ring of formal power series in one variable with coefficients in $\mathcal{O}_K$.}. Under certain technical assumptions, we have 
the following reesults: 
\begin{enumerate}
\item 
For each choice of $\Lambda_{(k_0 ;r),\mathcal{O}_{{K}}}$-basis 
$\Xi^\pm$ of $\mathbb{MS}(\Lambda_{(k_0 ;r),\mathcal{O}_{{K}}})^\pm$, there exists an analytic $p$-adic $L$-function $L_p (\{ \Xi^\pm \}) 
\in \Lambda_{(k_0;r)}  \widehat{\otimes}_{\mathbb{Z}_p} \hh_{h,\mathrm{cyc}}$  
which interpolates all critical special values of the Hecke $L$-function $L(f,s)$ where $f$ varies in the Coleman family $\mathbb{F}$. 
\item 
For each choice of $\Lambda_{(k_0 ;r),\mathcal{O}_{{K}}}$-basis 
$\Xi^\pm$ of $\mathbb{MS}(\Lambda_{(k_0 ;r),\mathcal{O}_{{K}}})^\pm$, there exists 
an Euler system 
$$
\bigl\{ \mathcal{Z} (r) \in H^1(\mathbb{Q} (\mu_r ) ,
(\mathbb{T}\widehat{\otimes}_{\mathbb{Z}_p} \Lambda^\sharp_{\mathrm{cyc}}
)^\ast (1)) \bigr\} _{r\in \mathcal{S} }
$$ 
such that the image of 
the local $p$-part 
$\mathrm{loc}_p (\mathcal{Z} (1)) \in H^1(\mathbb{Q}_p ,
(\mathbb{T}\widehat{\otimes}_{\mathbb{Z}_p} \Lambda^\sharp_{\mathrm{cyc}}
)^\ast (1))$ of the first layer $\mathcal{Z} (1)$ 
via the Coleman map 
$\mathrm{Col} : 
H^1 \big(\mathbb{Q}_{p} ,(\mathbb{T}\widehat{\otimes}_{\mathbb{Z}_p} \Lambda^\sharp_{\mathrm{cyc}}
)^\ast(1) \big) 
\longrightarrow \Lambda_{(k_0 ;r),\mathcal{O}_{{K}}}
 \widehat{\otimes}_{\mathbb{Z}_p} 
\HH_{h,\mathrm{cyc}}$ is equal to $L_p (\{ \Xi^\pm \}) 
\in \Lambda_{(k_0;r)}  \widehat{\otimes}_{\mathbb{Z}_p} \hh_{h,\mathrm{cyc}}$ \item 
The group $ H^2(\mathbb{Q}_{\Sigma}/\mathbb{Q} ,
(\mathbb{T}\widehat{\otimes}_{\mathbb{Z}_p} \Lambda^\sharp_{\mathrm{cyc}}
)^\ast (1)) $ is a torsion $\Lambda_{(k_0 ;r),\mathcal{O}_{{K}}} \widehat{\otimes}_{\mathbb{Z}_p} \Lambda_{\mathrm{cyc}}$-module 
for an $ \Lambda_{(k_0 ;r),\mathcal{O}_{{K}}} $-lattice $\mathbb{T}$ of the $p$-adic Galois representation $\mathbb{V}$ associated to $\mathbb{F}$.
\item 
One of the inclusions predicted by the two-variable Iwasawa Main conjecture holds true. 
That is, we have the following inclusion:
\begin{multline*}
\mathrm{char}_{\Lambda_{(k_0 ;r),\mathcal{O}_{{K}}} \widehat{\otimes}_{\mathbb{Z}_p} \Lambda_{\mathrm{cyc}}}
\left( 
H^1(\mathbb{Q}_{\Sigma}/\mathbb{Q} ,
(\mathbb{T}\widehat{\otimes}_{\mathbb{Z}_p} \Lambda^\sharp_{\mathrm{cyc}}
)^\ast (1)) 
 \Bigl/ \Lambda_{(k_0 ;r),\mathcal{O}_{{K}}} \widehat{\otimes}_{\mathbb{Z}_p} \Lambda_{\mathrm{cyc}}  \cdot \mathcal{Z}(1)  
 \right) 
\\  \subset 
 \mathrm{char}_{\Lambda_{(k_0 ;r),\mathcal{O}_{{K}}} \widehat{\otimes}_{\mathbb{Z}_p} \Lambda_{\mathrm{cyc}}} 
\left( H^2(\mathbb{Q}_{\Sigma}/\mathbb{Q} ,
(\mathbb{T}\widehat{\otimes}_{\mathbb{Z}_p} \Lambda^\sharp_{\mathrm{cyc}}
)^\ast (1))  \right) . 
\end{multline*}
\end{enumerate}
\end{hthm}
A remarkable difference of the proof of the case of non-ordinary $p$-adic families (Theorem 3) compared to the case of ordinary $p$-adic families (Theorem 2) is  
that the existence of the Euler system is much more non-trivial in the non-ordinary case. 
In the ordinary case, apart from the $p$-optimization of the Euler system, the existence of (not necesarily $p$-optimal) Euler System 
followed from the existence of Kato's Euler System over modular curves of level $Np^r$ for each $r$. 
As is shown in \cite{320}, Kato's Euler system at each level $Np^r$ forms a projective system when the power $r$ of $p$ varies and we can obtain an Euler system over a Hida family by taking the projective limit with respect to $r$. 
\par 
In the non-ordinary case, 
the big Galois representation $\mathbb{T}$ is not obtained in such a way. Hence even if the Beilinson Kato Euler system exist for 
every cuspform $f_\kappa$ with $\kappa$ varying in the sets of arithmetic points in the Coleman family in a point-wise manner, it is not clear if the Beilinson-Kato Euler system exists over 
the family. 
The Coleman family obtained in \cite{co97} is constructed in a more analytic way. 
A new idea to prove Theorem 3 (2) is that we use Coleman map constructed in \cite{NO16} (Theorem \ref{theorem:2variablep-adicLnon-ordinary})
and the two-variable $p$-adic $L$-function to glue the Euler systems which exist in a point-wise manner. We refer the reader to later sections of the paper for 
the detail of the proof of Theorem 3. 
\section{Iwasawa Main Conjecture for a cuspform (Review on classical results)}
\label{sec:Intro_motif}
Let $p>2$ be a fixed prime number. Throughout the paper, we fix embeddings
\begin{equation}
\overline{\mathbb{Q}} \hookrightarrow \mathbb{C}, \ \ \ 
\overline{\mathbb{Q}} \hookrightarrow \overline{\mathbb{Q}_p}
\end{equation}
of $\overline{\mathbb{Q}}$. 
For any field $L$, we denote by $G_L$ the absolute Galois group $\mathrm{Gal}(\overline{L}/L)$. 
\\ 
\ 
\par 
Let $f \in S_k (\Gamma_1 (M))$ be a normalized eigen cuspform 
of weight $k\geq 2$ and of level $\Gamma_1 (M)$ where $M$ is a natural number divisible by $p$. 
We usually denote by $f(q) = \sum^\infty_{n=1} a_n (f) q^n$ the Fourier expansion of $f$ at $i\infty$. 
Thus $f$ is $p$-stabilized\footnote{An eigenform $f$ is said to be $p$-stabilized either if $f$ is new at $p$ or if there exists an eigenform $g$ of weight $k$ whose level is prime to $p$
such that $f(q) = g(q) -\beta g(q^p)$ where $\beta$ is one of roots of 
$X^2 -a_p (g) X + \psi_g (p) p^{k-1}$. Here $\psi_f$ is the Neben character of $f$. 
In the latter case, we have $a_{n}(f) =a_{n}(g)$ for all natural numbers $n$ prime to $p$, which implies that the $p$-adic Galois representation $V_g$ associated to $g$ is isomorphic to the $p$-adic Galois representation $V_f$ associated to $f$. Thus the assumption of being $p$-stabilized is 
not restrictive at all.} in this paper. We denote by ${K}={K}_f$ the 
finite extension of $\mathbb{Q}_p$ obtained by adjoining the Fourier 
coefficients $a_n (f)$ to $\mathbb{Q}_p$. 
\par 
By Deligne \cite{del71} and Shimura \cite{shimura68}, to each normalized eigen cuspform $f$ as above, 
we can attach a modular $p$-adic Galois representation space $V_f = {K}^{\oplus 2}$ equipped with a continuous irreducible 
representation $\rho_f : G_{\mathbb{Q}} \longrightarrow 
\mathrm{Aut}_{{K}} (V_f)$ unramified outside all primes dividing $M$. 
The representation $\rho_f$ is characterized by the property 
$$
\mathrm{Tr} (\rho_f (\mathrm{Frob}_\ell )) = a_\ell (f) 
$$
for every prime $\ell$ not dividing $M$. 
\par 
The one--variable cyclotomic Iwasawa main conjecture for all Galois representations arising as above has been formulated and studied by various people including Greenberg, Mazur, Perrin-Riou, Kato and Skinner--Urban. 
\subsection{Selmer group and $p$-adic $L$-function}\label{section:Selmer group and $p$-adic $L$-function}
When $f$ is {\bf ordinary} at $p$ in the sense that $a_p (f)$ is a $p$-adic unit, Wiles proved that the representation $V_f$ is reducible as a $G_{\mathbb{Q}_p}$-module, with a canonical submodule $F^+_p V_f
\subset V_f$ stable by the action of $G_{\mathbb{Q}_p}$. 
For any $G_{\mathbb{Q}}$-stable lattice $T \subset V_f$ 
and for each number field $F$, a Selmer group $\mathrm{Sel}_A (F)$ for $A:=V_f /T$ over $F$ is 
defined as a subgroup of the Galois cohomology group $H^1 (F , A)$ 
according to Greenberg \cite {gre89}: 
\begin{equation}\label{equation:definition_Selmer}
\mathrm{Sel}_A (F) = 
\mathrm{Ker} \left[ 
H^1 (F,A) \longrightarrow \prod_{\lambda \nmid p} H^1 (I_\lambda ,A)]
\times 
\prod_{\mathfrak{p} \vert p}\dfrac{H^1 (I_\mathfrak{p} ,A )}
{\mathrm{Image}(H^1 (I_\mathfrak{p} ,F^+_p A))}\right]
\end{equation}
where $I_\lambda$ (resp. $I_\mathfrak{p}$) denotes the inertia subgroup 
at each place $\lambda$ (resp. $\mathfrak{p}$) not dividing $p$ (resp. dividing $p$). 
\begin{lemma}\label{lemma:finitenessSel_A}
When $F$ is a finite extension of $\mathbb{Q}$, the Pontrjagin dual $\mathrm{Sel}_A (F)^\vee$ of $\mathrm{Sel}_A (F)$ 
is a finitely generated $\mathbb{Z}_p$-module. 
\end{lemma}
This lemma can be proved by reducing the problem to 
Hermite's theorem on the finiteness 
of the number of the extensions of a given number fields whose ramified places and the degrees are fixed. Since the argument of the proof is well-known and found in \cite[\S 4]{gre} for example, we omit the proof. 
\par 
When $F=\mathbb{Q} (\mu_{p^\infty})$, the Galois group 
$G_{\mathrm{cyc}}:= \mathrm{Gal} (\mathbb{Q} (\mu_{p^\infty})/\mathbb{Q})$ 
acts on $\mathrm{Sel}_A (\mathbb{Q} (\mu_{p^\infty}))$ thanks to the functoriality of Galois cohomology. Since the Pontrjagin dual $\mathrm{Sel}_A (\mathbb{Q} (\mu_{p^\infty}))^\vee$ is a compact module with continuous action of $G_{\mathrm{cyc}}$, the module $\mathrm{Sel}_A (\mathbb{Q} (\mu_{p^\infty}))^\vee$ naturally has a structure of a module over the semi-local algebra  
$\Lambda_{\mathrm{cyc}}:= \mathbb{Z}_p [[ G_{\mathrm{cyc}}]]$. 
We note that $\Lambda_{\mathrm{cyc}}$ is a semi-local algebra which is isomorphic to 
$\mathbb{Z}_p [[X]] \times \cdots \times \mathbb{Z}_p [[X]]$ ($p-1$ components). 
\par 
We have also the following lemma. 
\begin{lemma}\label{lemma:finitelygenerated}
$\mathrm{Sel}_A (\mathbb{Q} (\mu_{p^\infty}))^\vee$ is finitely generated over $\Lambda_{\mathrm{cyc}}$. 
\end{lemma}
Since Lemma \ref{lemma:finitelygenerated} is a corollary of Lemma \ref{lemma:finitenessSel_A} by using 
the control theorem of Selmer group (see \cite[Proposition 3.8]{ochiai-JNT01}) and Nakayama's theorem, we omit the proof. 
On the other hand, the following theorem 
relies essentially on the technique of Euler system.
\\
\\  
{\bf Theorem A (Torsion property of the Selmer group of $f$/Kato--Rubin--Rohrlich)} 
\\ 
Let $f$ be an ordinary eigen cuspform of weight 
$k\geq 2$. Assume that $p \geq 5$ \footnote{We exclude $p=3$ only because we did not find a complete reference for the CM case.}.     
Then the finitely generated $\Lambda_{\mathrm{cyc}}$-module $\mathrm{Sel}_A (\mathbb{Q}(\mu_{p^\infty}))^\vee$ is torsion over $\Lambda_{\mathrm{cyc}}$ \footnote{Note that the algebra $\Lambda_{\mathrm{cyc}}$ is not local, but 
a product of $p-1$ copies of local domain $\mathbb{Z}_p [[X]]$. A $\Lambda_{\mathrm{cyc}}$-module $M$ is said to be 
torsion when the base extension of $M$ to each local component $\mathbb{Z}_p [[X]]$ is torsion.}. 
\\ 
\ 
\begin{rem}
In \eqref{equation:definition_Selmer}, we required $f$ to be ordinary 
in order to define the Selmer group. However, there is another 
Selmer group $\mathrm{Sel}^{\mathrm{BK}}_A (\mathbb{Q} (\mu_{p^\infty}))$ defined by using the local condition $H^1_f$ 
of Bloch-Kato (see \cite[Section 3]{20}), which is almost equal to  
the above $\mathrm{Sel}_A (\mathbb{Q} (\mu_{p^\infty}))$ and 
is also valid when $V_f$ is non-ordinary. 
When the form $f$ is non-ordinary,  
Lemma \ref{lemma:finitelygenerated} is also applied also to see that 
$\mathrm{Sel}^{\mathrm{BK}}_A (\mathbb{Q} (\mu_{p^\infty}))^\vee$ is a finitely generated $\Lambda_{\mathrm{cyc}}$-module. 
However, it is known that $\mathrm{Sel}^{\mathrm{BK}}_A (\mathbb{Q} (\mu_{p^\infty}))^\vee$ is not a torsion $\Lambda_{\mathrm{cyc}}$-module.
Let us recall the following exact sequence obtained by the global duality theorem of Poitou-Tate: 
\begin{multline}\label{equation:4termsequence0}
\varprojlim_n H^1 (\mathbb{Q}_{\Sigma}/\mathbb{Q}(\mu_{p^n}),T^\ast (1))   \overset{\overline{\mathrm{loc}}_p}{\longrightarrow }
\varprojlim_n  
\dfrac{H^1 (\mathbb{Q}_p (\mu_{p^n}) ,T^\ast (1))}
{H^1_f (\mathbb{Q}_p (\mu_{p^n}) ,T^\ast (1))}
\\ 
\longrightarrow \mathrm{Sel}^{\mathrm{BK}}_A (\mathbb{Q}(\mu_{p^\infty}))^\vee 
\longrightarrow 
\varprojlim_n H^2 (\mathbb{Q}_{\Sigma}/\mathbb{Q}(\mu_{p^n}),T^\ast (1))  
\end{multline}
where we denote by $T^\ast$ the $\mathbb{Z}_p$-linear dual of $T$ and by $T^\ast (1)$ the Tate twist of $T^\ast$. 
Let ${K}$ be a finite extension of 
$\mathbb{Q}_p$ obtained by adjoining Fourier coefficients of $f$ and $\mathcal{O}_K$ the ring of integers in $K$. 
By using the semi-global Euler--Poincar\'{e} characteristic formula of Galois cohomology, 
the $\Lambda_{\mathrm{cyc}}\otimes_{\mathbb{Z}_p} \mathcal{O}_{{K}}$-module 
$\varprojlim_n H^1 (\mathbb{Q}_{\Sigma}/\mathbb{Q}(\mu_{p^n}),T^\ast (1))$ is 
generically of rank one over $\Lambda_{\mathrm{cyc}} \otimes_{\mathbb{Z}_p} \mathcal{O}_{{K}}$.
On the other hand, by a result of Berger \cite[Th\'{e}or\`{e}m A]{ber05}, we have 
$$
\varprojlim_n  
\dfrac{H^1 (\mathbb{Q}_p (\mu_{p^n}) ,T^\ast (1))}
{H^1_f (\mathbb{Q}_p (\mu_{p^n}) ,T^\ast (1))}
\cong \varprojlim_n H^1 (\mathbb{Q}_p (\mu_{p^n}) ,T^\ast (1)). 
$$
By using the local Euler--Poincar\'{e} characteristic formula of Galois cohomology, 
the $\Lambda_{\mathrm{cyc}}\otimes_{\mathbb{Z}_p} \mathcal{O}_{{K}}$-module 
$\varprojlim_n H^1 (\mathbb{Q}_p (\mu_{p^n}) ,T^\ast (1))$ is 
generically of rank two over $\Lambda_{\mathrm{cyc}} \otimes_{\mathbb{Z}_p} \mathcal{O}_{{K}}$. 
Hence, we conclude that $\mathrm{Sel}^{\mathrm{BK}}_A (\mathbb{Q} (\mu_{p^\infty}))^\vee$ is not a torsion $\Lambda_{\mathrm{cyc}}$-module 
by \eqref{equation:4termsequence0}. We remark that, as we will discuss later 
in the proof of the inequality \eqref{equation:IMC_for_elliptic_cuspform_modulo_mu2} in Theorem C, 
a result of Kato \cite{320} implies that the first map in \eqref{equation:4termsequence0} is injective 
and the last term $\varprojlim_n H^2 (\mathbb{Q}_{\Sigma}/\mathbb{Q}(\mu_{p^n}),T^\ast (1))$ in \eqref{equation:4termsequence0} 
is $\Lambda_{\mathrm{cyc}}$-torsion. By \eqref{equation:4termsequence0}, this implies that 
$\mathrm{Sel}^{\mathrm{BK}}_A (\mathbb{Q} (\mu_{p^\infty}))^\vee$ is 
generically of rank two over $\Lambda_{\mathrm{cyc}} \otimes_{\mathbb{Z}_p} \mathcal{O}_{{K}}$ more precisely. 
\end{rem}
\begin{proof}[Proof of Theorem A] 
We remark that the proof is completely different 
depending on whether the cuspform $f$ has complex multiplication (in the sense  of Ribet \cite{410ggg}) or not. 
\par 
When $f$ is a cuspform with complex multiplication by an imaginary quadratic field $F$, the result 
follows from the two-variable Iwasawa Main Conjecture for $F$ proved by Rubin \cite{419efg}. 
The Selmer group $\mathrm{Sel}_A (\mathbb{Q}(\mu_{p^\infty}))^\vee$ is a one-variable specialization of 
the two-variable Selmer group and this specialization is torsion over the cyclotomic Iwasawa algebra by the non-triviality of the 
$p$-adic $L$-function of $f$ proved by by Rohrlich \cite{417roh} (see 
\cite{haraochiai} for a further detail 
and Remark \ref{remark:Proof_Theorem_C_onevaiable} (1) for related comments on the proof of Theorem C for $f$ with complex multiplication).  
%
%
%
%
%
\par 
When the cuspform $f$ does not have complex multiplication, Kato proves that the Pontrjagin dual of the Selmer group is torsion 
by using a Beilinson--Kato Euler system. By the same reasoning as the above, Euler systems of Beilinson-Kato are 
non-trivial thanks to the non-triviality of the $p$-adic $L$-function of $f$ due to Rohrlich.  
\end{proof}
In order to introduce the $p$-adic $L$-function, we recall the following result which is due to Shimura \cite{441ggg}. 
\\ 
{\bf Theorem (algebraicity of special values/Shimura)} 
\\ 
There exist complex numbers $\Omega^{+}_{f,\infty},\,  \Omega^{-}_{f,\infty}
\in \mathbb{C}^\times$ (called {\bf complex periods}) such that we have 
\begin{equation}\label{equation:algebraicity}
\dfrac{L(f , \phi^{-1},j)}
{(2\pi \sqrt{-1})^{j} \Omega_{f,\infty} ^{\mathrm{sgn}(j,\phi )}} 
\in \mathbb{Q}_f [\phi ]^\times 
\end{equation}
for any integer $j$ satisfying $1 \leq j\leq k-1$ and for any Dirichlet 
character $\phi$, where $\mathrm{sgn}(j,\phi )$ denotes the sign of 
$(-1)^{(j-1 ) \phi (-1)}$, $\mathbb{Q}_f$ denotes the Hecke field of $f$ which is a finite extension of $\mathbb{Q}$ obtained by adjoining 
all Fourier coefficients of $f$ and $\mathbb{Q}_f [\phi ]$ is the field obtained by adjoining the values of $\phi $ to $\mathbb{Q}_f$. 
\\ 
\begin{rem}\label{remark:complexperiod1}
\begin{enumerate}
\item 
When $f$ is an eigen cuspform of level $\Gamma_1 (M)$, the periods $\Omega^{+}_{f,\infty}$ and $\Omega^{-}_{f,\infty}$ 
can be obtained as ``determinants of the geometric comparison isomorphism'' 
\begin{equation}
\mathrm{Fil}^1 H^1_{\mathrm{dR}} (X_1 (M ),\omega^{\otimes k_f -2}) [f]\otimes_{\mathbb{Q}_f} \mathbb{C} \cong H^1_c (Y_1 (M)_{\mathbb{C}} ,
\mathcal{L}_{k_f -2} (\mathbb{Q}_f))^{\pm}[f ] \otimes_{\mathbb{Q}_f} \mathbb{C}
\end{equation}
where $\omega$ is a standard coherent sheaf on $X_1 (M )$ obtained from 
the push-forward of the sheaf of relative differential forms on the universal elliptic curve on $Y_1 (M)$ and $\mathcal{L}_{k_f -2} (\mathbb{Q}_f)$ is the standard local system of rank $k-1$ over 
$Y_1 (M)_{\mathbb{C}} $. The symbol $[f]$ denotes the cutting out of the 
rank-one $\mathbb{Q}_f$-subspace of the cohomology group on which the Hecke operators act via 
Fourier coefficients of $f$.  
The period can be defined by: 
\begin{equation}\label{equation:definition_of_complecx periods}
d_f \otimes 1 = \Omega^{\pm}_{f,\infty} \cdot b^\pm _f \otimes 1 
\end{equation}
where $d_f$ is a canonical basis of $\mathrm{Fil}^1 H^1_{\mathrm{dR}} (X_1 (M ),\omega^{\otimes k_f -2}) [f]$ given by $f$
 and $b^\pm _f$ is a basis of a rank-one $\mathbb{Q}_f$-space 
 $H^1_c (Y_1 (M)_{\mathbb{C}} ,\mathcal{L}_{k_f -2} (\mathbb{Q}_f))^{\pm}[f ]$. Since there are no canonical choices of $b^\pm _f$ in general, 
$\Omega^{\pm}_{f,\infty} =\Omega^{\pm}_{f,\infty} (b^\pm_f)$ are defined 
only modulo multiplication by elements of $(\mathbb{Q}_f )^\times$ and they 
make sense as an element of $\mathbb{C}^\times /\mathbb{Q}_f ^\times$. 
\item 
The algebraicity theorem gives a supporting example of Deligne conjecture 
(see \cite{del79}) on special values of the $L$-function associated to 
the motive of $f$. 
\item 
It is important to choose the periods 
$\Omega^{+}_{f,\infty} , \Omega^{-}_{f,\infty} \in \mathbb{C}^\times$ so that the $p$-adic valuations 
of the values match with 
the generalized Birch and Swinnerton-Dyer conjecture as follows
$$
\left\vert \dfrac{L(f ,\phi^{-1} ,j)}
{(2\pi \sqrt{-1})^{j} \Omega_{f,\infty} ^{\mathrm{sgn}(j,\phi)}} \right\vert_p 
= \dfrac{\left\vert \# \mathrm{Sel}^{\mathrm{BK}}_{(V_f /T )(j) \otimes \phi } \right\vert_p \cdot  
\left\vert \text{local Tamagawa number for $V_f (j) \otimes \phi$} \right\vert_p}
{\left\vert H^0 (\mathbb{Q} , (V_f /T )(j) \otimes \phi ) \right\vert_p \cdot 
\left\vert H^0 (\mathbb{Q} , (V^\ast_f /T^\ast )(1-j) \otimes \phi^{-1} ) \right\vert_p},
$$ 
where $T$ is a Galois stable lattice of $V_f$, $j$ is an integer with $1\leq j\leq k-1$ and $\phi$ is a Dirichlet character of $p$-power conductor such that 
$\mathrm{Sel}_{(V_f /T )(j) \otimes \phi } $ is finite. Such periods are called {\bf $p$-optimal complex periods} with respect to the lattice $T$ \footnote{This notion of $p$-optimal complex periods is conditional. The Iwasawa Main Conjecture is expected to be compatible with generalized Birch and Swinnerton-Dyer conjecture for any choice of $j$ and $\phi$ through the control theorem of Selmer groups. Assuming that, the notion of $p$-optimal does not depend on the choice of $j$ and $\phi$. }. 
Let $\mathbb{Z}_{f,(p)}$ be the localization of the ring of integers $\mathbb{Z}_{f}$ of $\mathbb{Q}_f$ at the prime over $p$ of $\mathbb{Z}_{f}$ induced by the fixed inclusion $\mathbb{Q}_f \hookrightarrow \overline{\mathbb{Q}_p}$. 
The $p$-optimal complex periods $\Omega^{+}_{f,\infty} , \Omega^{-}_{f,\infty}$ are unique modulo multiplication by elements in $\mathbb{Z}_{f,(p)}^\times$ and hence $\Omega^{+}_{f,\infty}$ 
and $\Omega^{-}_{f,\infty}$ 
make sense as elements of $\mathbb{C}^\times /\mathbb{Z}_{f,(p)}^\times$. 
\end{enumerate}
\end{rem}
Recall that ${K}$ is the completion of $\mathbb{Q}_f$ 
in $\overline{\mathbb{Q}_p}$. 
Now we recall the $p$-adic $L$-function due to Amice--V\'{e}lu \cite{9fff} and Visik \cite{470} following  the modular symbol method developed by Manin and Mazur.  
\\
\\ 
{\bf Theorem B (Existence of the one--variable ordinary $p$-adic $L$-function)} 
\\ 
Let $f$ be an ordinary eigen cuspform of weight 
$k\geq 2$. 
Let us choose complex periods $\Omega^{+}_{f,\infty}$ and $ \Omega^{-}_{f,\infty} $. Then, there exists an analytic $p$-adic $L$-function $L_p (\{ 
\Omega^{\pm}_{f,\infty} \}) 
\in \Lambda_{\mathrm{cyc}} \otimes_{\mathbb{Z}_p} {K}$ satisfying the interpolation formula\footnote{The both sides of the equation 
are in $\overline{\mathbb{Q}}$ thanks to algebraicity theorem of Shimura stated above.}:
\begin{equation}\label{equation:interpolation_AV_V}
\chi^j_{\mathrm{cyc}} \phi (L_p (\{ 
\Omega^{\pm}_{f,\infty} \})) = (-1)^{j} (j-1)!  e_p (f, j ,\phi ) 
 \tau (\phi) \dfrac{L(f , \phi^{-1},j)}
{(2\pi \sqrt{-1})^{j}\Omega^{\mathrm{sgn}(j,\phi )}_{f,\infty} } ,
\end{equation}
for any integer $j$ with $1 \leq j\leq k-1$ and for any Dirichlet 
character $\phi$ whose conductor $\mathrm{Cond} (\phi)$ is a power of $p$. 
In the formula, $\tau (\phi )$ denotes the Gauss sum 
$\displaystyle{\sum^{
\mathrm{Cond} (\phi)}_{i=1}} \phi (i) 
\left( 
\mathrm{exp} \dfrac{2\pi \sqrt{-1}}{\mathrm{Cond} (\phi)}\right)^i$ and 
the factor $e_p (f, j, \phi )$ is defined as follows $:$ 
$$ 
e_p(f, j, \phi ) = 
\begin{cases} 
 1- \dfrac{p^{j-1} }{a_{p}(f ) } 
& \text{when $\phi =\text{\boldmath$1$}$,} \\ 
\left( \frac{p^{j-1}}{a_{p}(f )} 
\right)^{\mathrm{ord}_p (\mathrm{Cond} (\phi))} 
& \text{when $\phi  \not= \text{\boldmath$1$}$.}
\end{cases} 
$$ 
\\ 
\\ 
\begin{rem}\label{remark:p-adicLforf}
\begin{enumerate}
\item 
There are 
several constructions of $L_p (\{ \Omega^{\pm}_{f,\infty} \})$ 
other than the method of modular symbol mentioned above.   
First, thanks to Hida, Panchishkin and Dabrowski, we should be able to 
construct $L_p (\{ \Omega^{\pm}_{f,\infty} \})$ by the Rankin-Selberg method using Eisenstein series. Also, $L_p (\{ \Omega^{\pm}_{f,\infty} \})$ 
is constructed from a Beilinson-Kato Euler system via a generalized Coleman map (see \cite[chap. IV]{320}). Finally, thanks to Coates--Wiles, de Shalit and Yager, when the cuspform $f$ has complex multiplication, 
$L_p (\{ \Omega^{\pm}_{f,\infty} \})$ is constructed from an Euler system of elliptic units via a generalized Coleman map. 
\item We expect that, when the complex periods $\Omega^{+}_{f,\infty} $ and $\Omega^{-}_{f,\infty} $ are $p$-optimal, the $p$-adic $L$-function $L_p (\{ \Omega^{\pm}_{f,\infty} \})$ is integral in the sense that $L_p (\{ \Omega^{\pm}_{f,\infty} \}) \in 
\Lambda_{\mathrm{cyc}} \otimes_{\mathbb{Z}_p} \mathcal{O}_{{K}}$ where $ \mathcal{O}_{{K}}$ is the ring of integers in ${K}$. 
\item 
When the weight $k$ of $f$ is greater than $2$, the value $j=k-1$
among the critical values $1\leq j\leq k-1$ 
is inside the convergence region $\mathrm{Re}(s)> \frac{k+1}{2}$ of the Euler product expansion of the $L$-function $L(f , \phi^{-1},s)$. By the definition of the convergence of infinite product, $L_p (\{ 
\Omega^{\pm}_{f,\infty} \})$ is nonzero thanks to the interpolation property  \eqref{equation:interpolation_AV_V} \footnote{When $k=1$, the value $j=k-1$ is on the border of the convergence region. In this case, Jacquet-Shalika \cite{276jsn} shows that $L(f , \phi^{-1},2)$ is always nonzero for $k=3$.}. When $k=2$, the only critical value $L(f , \phi^{-1},1)$
can be sometimes zero. Hence, a priori, it is not clear if $L_p (\{ \Omega^{\pm}_{f,\infty} \})$ is nonzero. However, 
Rohrhlich \cite{417roh} proves that, for any eigen cuspform $f$ of weight $2$, there exists a Dirichlet character $\phi$ 
whose conductor $\mathrm{Cond} (\phi)$ is a power of $p$ such that 
we have $L(f , \phi^{-1},1) \not= 0$. Hence, 
$L_p (\{ \Omega^{\pm}_{f,\infty} \}) $ is always nonzero as an element of 
$\Lambda_{\mathrm{cyc}} \otimes_{\mathbb{Z}_p} {K}$. 
\end{enumerate}
\end{rem}
\subsection{Iwasawa Main Conjecture}
With the preparation above, we recall the statement of the one-variable cyclotomic Iwasawa Main conjecture for elliptic cuspforms:  
\\ 
\\ 
{\bf Conjecture (Cyclotomic Iwasawa Main Conjecture for $f$)} 
\\ 
Let $f$ be an ordinary eigen cuspform of weight $k\geq 2$ and let $T \subset V_f$ be a Galois stable lattice. Let us choose 
the complex periods $\Omega^{+}_{f,\infty}  , 
\Omega^{-}_{f,\infty}$ to be $p$-optimal with respect to the lattice $T$. 
Then, we have the the following equality of principal ideals in the ring $\Lambda_{\mathrm{cyc}} \otimes_{\mathbb{Z}_p} \mathcal{O}_{{K}}$:  
\begin{equation}\label{equation:IMC_for_elliptic_cuspform}
(L_p (\{ \Omega^{\pm}_{f,\infty} \}))=\mathrm{char}_{\Lambda_{\mathrm{cyc}} \otimes_{\mathbb{Z}_p} \mathcal{O}_{{K}}} \, \mathrm{Sel}_A (\mathbb{Q}(\mu_{p^\infty}))^\vee 
\end{equation} 
where $A:=V_f /T$.  
\\ 
\begin{rem}
The right-hand side of \eqref{equation:IMC_for_elliptic_cuspform} 
is nonzero by Theorem A. Also, the left-hand side of \eqref{equation:IMC_for_elliptic_cuspform} 
is nonzero by Rohrlich's result explained in Remark \ref{remark:p-adicLforf} (3). 
\end{rem}
Before introducing the result on the above conjecture, we recall 
the relation between the ideals of the ring  
$\Lambda_{\mathrm{cyc}} \otimes_{\mathbb{Z}_p} {\mathcal{O}_{{K}}}$
and the ideals of the ring 
$\Lambda_{\mathrm{cyc}} \otimes_{\mathbb{Z}_p} {{K}}$. 
Since $\Lambda_{\mathrm{cyc}} \otimes_{\mathbb{Z}_p} {{K}}$ is a localization of $\Lambda_{\mathrm{cyc}} \otimes_{\mathbb{Z}_p} {\mathcal{O}_{{K}}}$ with respect to the multiplicative set $\{ \varpi^n \}_{n \in \mathbb{Z}}$ with $\varpi$ a uniformizer of $\mathcal{O}_{{K}}$, 
we have an injection 
\begin{align*}
\{ \text{prime ideals in $\Lambda_{\mathrm{cyc}} \otimes_{\mathbb{Z}_p} {{K}}$} \}  
& \hookrightarrow \{ \text{prime ideals in $\Lambda_{\mathrm{cyc}} \otimes_{\mathbb{Z}_p} \mathcal{O}_{{K}}$ } \} 
\\ 
I & \mapsto I \cap ( \Lambda_{\mathrm{cyc}} \otimes_{\mathbb{Z}_p} \mathcal{O}_{{K}} )
\end{align*} 
and the image coincide with prime ideals which do not contain 
$\varpi (\Lambda_{\mathrm{cyc}} \otimes_{\mathbb{Z}_p} \mathcal{O}_{{K}})$. 
For a finitely generated torsion 
$\Lambda_{\mathrm{cyc}} \otimes_{\mathbb{Z}_p} 
\mathcal{O}_{{K}}$-module $M$, 
we lose the information on prime factors containing $\varpi (\Lambda_{\mathrm{cyc}} \otimes_{\mathbb{Z}_p} \mathcal{O}_{{K}})$ 
by considering 
$\mathrm{char}_{\Lambda_{\mathrm{cyc}} \otimes_{\mathbb{Z}_p} {{K}}} \, (M\otimes_{\mathcal{O}_{{K}}} {{K}})$ 
instead of $\mathrm{char}_{\Lambda_{\mathrm{cyc}} \otimes_{\mathbb{Z}_p} \mathcal{O}_{{K}}} \, M $. 
The following result is known: 
\\
\\ 
{\bf Theorem C (Cyclotomic Iwasawa Main Conjecture for $f$)} 
\\ 
Let $f$ be an ordinary eigen cuspform of weight $k\geq 2$. 
We fix a Galois stable lattice $T \subset V_f$ and we choose 
complex periods $\Omega^{+}_{f,\infty} $ and $
\Omega^{-}_{f,\infty}$ to be $p$-optimal with respect to the lattice $T$. 
We assume the following technical conditions: 
\begin{enumerate}
\item[(i)]
The image of the mod $\varpi$ representation $\overline{\rho} :\, G_{\mathbb{Q}} \longrightarrow \mathrm{Aut}_{\mathcal{O}_K /(\varpi )} 
(T/\varpi T ) \cong GL_2 (\mathcal{O}_K /(\varpi ))$ is irreducible.  
\item[(ii)]
The semi-simplification of $\overline{\rho}$ restricted to the decomposition group $G_{\mathbb{Q}_p}$ at $p$ is a direct sum of 
two different characters. 
\item[(iii)]  
The prime-to-$p$ part of the conductor of $f$ is square-free natural number. 
\item[(iv)] 
The odd prime number $p$ which we fix is not equal to $3$ (assumed only in the CM case).
\end{enumerate}

Then, we have the following equality of principal ideals in the ring $\Lambda_{\mathrm{cyc}} \otimes_{\mathbb{Z}_p} {{K}}$:    
\begin{equation}\label{equation:IMC_for_elliptic_cuspform_modulo_mu}
(L_p (\{ \Omega^{\pm}_{f,\infty} \}))=\mathrm{char}_{\Lambda_{\mathrm{cyc}} \otimes_{\mathbb{Z}_p} {{K}}} \, \, \mathrm{Sel}_A (\mathbb{Q}(\mu_{p^\infty}))^\vee \otimes_{\mathcal{O}_{{K}}} {{K}} . 
\end{equation} 
\ \\ 
\begin{rem}\label{remark:Proof_Theorem_C_onevaiable}
In the known case of \eqref{equation:IMC_for_elliptic_cuspform_modulo_mu},
the proof is completely different depending on whether the cuspform $f$ has complex multiplication or not. 
\begin{enumerate}
\item 
We explain the case when $f$ is a cuspform with complex multiplication by an imaginary quadratic field $F$. 
When $f$ is of weight $2$ with rational Fourier coefficients corresponding to a CM elliptic curve, 
the equality \eqref{equation:IMC_for_elliptic_cuspform_modulo_mu} is obtained by specializing the two-variable Iwasawa Main conjecture for $F$ proved by Rubin \cite[\S 12]{419efg} at a certain height-one prime of the two-variable Iwasawa algebra for $F$. 
When $f$ is of arbitrary weight, we refer the reader to \cite[Theorem A]{haraochiai} for this ``descent'' from the two-variable Iwasawa Main conjecture to the one-variable cyclotomic 
Iwasawa Main conjectures of CM cuspforms as well as its generalization to Hilbert modular cuspforms with complex multiplication. \par 
The reason why we have only a weaker equality over $\Lambda_{\mathrm{cyc}} \otimes_{\mathbb{Z}_p} {{K}}$ (not over $\Lambda_{\mathrm{cyc}} \otimes_{\mathbb{Z}_p} \mathcal{O}_{{K}}$) is the lack of knowledge on the periods. 
It is not known whether the ratio of a $p$-optimal complex period in the context of complex multiplication and a $p$-optimal complex period in the above context 
is a $p$-adic unit (see \cite[Conj. 2.26]{haraochiai}). 
Thus we can compare the ideal of a $p$-adic $L$-function obtained in Theorem B constructed by the modular symbol method with 
the ideal of a different $p$-adic $L$-function constructed by the CM method obtained in \cite[Theorem A]{haraochiai} only modulo a power of $\varpi$ 
(see \cite[\S 2.4]{haraochiai} for calculation). 
\item 
When $f$ does not have complex multiplication, Kato \cite{320} proved the following inclusion of principal ideals in the ring $\Lambda_{\mathrm{cyc}} \otimes_{\mathbb{Z}_p} {{K}}$ using a Beilinson-Kato Euler system: 
\begin{equation}\label{equation:IMC_for_elliptic_cuspform_modulo_mu2}
(L_p (\{ \Omega^{\pm}_{f,\infty} \}))
\subset 
\mathrm{char}_{\Lambda_{\mathrm{cyc}} \otimes_{\mathbb{Z}_p} {{K}}} \,  (\mathrm{Sel}_A (\mathbb{Q}(\mu_{p^\infty}))^\vee \otimes_{\mathcal{O}_{{K}}} {{K}}) . 
\end{equation} 
(see Remark \ref{remark:Colemanmap_singlemodularform} for the case when the level of $f$ is divisible by $p$) 
\par 
On the other hand, Skinner--Urban \cite{444sui} proves 
the following opposite inclusion of principal ideals in the ring $\Lambda_{\mathrm{cyc}} \otimes_{\mathbb{Z}_p} {{K}}$: 
\begin{equation}\label{equation:IMC_for_elliptic_cuspform_modulo_mu3}
(L_p (\{ \Omega^{\pm}_{f,\infty} \}))
\supset 
\mathrm{char}_{\Lambda_{\mathrm{cyc}} \otimes_{\mathbb{Z}_p} {{K}}} \,  (\mathrm{Sel}_A (\mathbb{Q}(\mu_{p^\infty}))^\vee \otimes_{\mathcal{O}_{{K}}} {{K}}) . 
\end{equation} 
by the method using the Eisenstein ideal for $U(2,2)$. 
The desired equality is obtained by combining \eqref{equation:IMC_for_elliptic_cuspform_modulo_mu2} and 
\eqref{equation:IMC_for_elliptic_cuspform_modulo_mu3}. 
\item 
We explain about the technical conditions (i) to (iv). 
The condition (i) is essential for \eqref{equation:IMC_for_elliptic_cuspform_modulo_mu3}. Kato's result \eqref{equation:IMC_for_elliptic_cuspform_modulo_mu2} 
does not require the condition (i), but it suffices to assume that $f$ is not CM for \eqref{equation:IMC_for_elliptic_cuspform_modulo_mu2}.   
Note also that Kato proves the following inclusion of principal ideals in the ring 
$\Lambda_{\mathrm{cyc}} \otimes_{\mathbb{Z}_p} \mathcal{O}_{{K}}$ stronger than \eqref{equation:IMC_for_elliptic_cuspform_modulo_mu2}: 
\begin{equation*}
(L_p (\{ \Omega^{\pm}_{f,\infty} \}))
\subset 
\mathrm{char}_{\Lambda_{\mathrm{cyc}} \otimes_{\mathbb{Z}_p}  \mathcal{O}_{{K}}
} \,  (\mathrm{Sel}_A (\mathbb{Q}(\mu_{p^\infty}))^\vee )  
\end{equation*}
under the condition that the the image of $\overline{\rho}$ contains $SL_2 (\mathbb{F}_p)$.
The condition (ii) is required only for \eqref{equation:IMC_for_elliptic_cuspform_modulo_mu3} and \eqref{equation:IMC_for_elliptic_cuspform_modulo_mu2} holds without assuming (ii). The condition (iii) is required only for \eqref{equation:IMC_for_elliptic_cuspform_modulo_mu3} and \eqref{equation:IMC_for_elliptic_cuspform_modulo_mu2} holds without assuming (iii). 
We also remark that cuspforms $f$ with complex multiplication are excluded also in the work of Skinner-Urban \eqref{equation:IMC_for_elliptic_cuspform_modulo_mu3}
because of this assumption (iii). 
\end{enumerate}
\end{rem}
We discuss a further generalization of the above result in later sections, 
especially a generalization of \eqref{equation:IMC_for_elliptic_cuspform_modulo_mu2}. For this purpose, it is important to look into the proof of \eqref{equation:IMC_for_elliptic_cuspform_modulo_mu2} in Theorem C (In fact, we prove Theorem A at the same time under the setting of \eqref{equation:IMC_for_elliptic_cuspform_modulo_mu2}). 
\par 
\begin{proof}[Brief sketch of the proof of the inequality \eqref{equation:IMC_for_elliptic_cuspform_modulo_mu2} in Theorem C] 
Let $\Sigma$ be a finite set of primes of $\mathbb{Q}$ which contains 
$\{ p\}$, $\{\infty \}$ and the places dividing the level of the cuspform $f$. 
Let $\mathbb{Q}_\Sigma$ be the maximal Galois extension of $\mathbb{Q}$ which is unramified outside $\Sigma$. 
\par 
We denote by $\mathrm{loc}_p$ the $\Lambda_{\mathrm{cyc}}$-linear homomorphism: 
$$
\varprojlim_n 
H^1 (\mathbb{Q}_{\Sigma}/\mathbb{Q}(\mu_{p^n}),T^\ast (1))  
\longrightarrow 
\varprojlim_n  
{H^1 (\mathbb{Q}_p (\mu_{p^n}) ,T^\ast (1))}
$$ 
and we denote by $\overline{\mathrm{loc}}_p$ the $\Lambda_{\mathrm{cyc}}$-linear homomorphism obtained by composing the following $\Lambda_{\mathrm{cyc}}$-linear homomorphism after $\mathrm{loc}_p$: 
$$
\varprojlim_n  
{H^1 (\mathbb{Q}_p (\mu_{p^n}) ,T^\ast (1))}
 \longrightarrow 
\varprojlim_n  
\dfrac{H^1 (\mathbb{Q}_p (\mu_{p^n}) ,T^\ast (1))}
{\mathrm{Im}(H^1 (\mathbb{Q}_p (\mu_{p^n}) ,F^+_p T^\ast (1)))}. 
$$ 
Then we have the following sequence of 
$\Lambda_{\mathrm{cyc}}$-modules: 
\begin{multline}\label{equation:4termsequence1}
0 \longrightarrow \varprojlim_n H^1 (\mathbb{Q}_{\Sigma}/\mathbb{Q}(\mu_{p^n}),T^\ast (1))   \overset{\overline{\mathrm{loc}}_p}{\longrightarrow }
\varprojlim_n  
\dfrac{H^1 (\mathbb{Q}_p (\mu_{p^n}) ,T^\ast (1))}
{\mathrm{Im}(H^1 (\mathbb{Q}_p (\mu_{p^n}) ,F^+_p T^\ast (1)))}
\\ 
\longrightarrow \mathrm{Sel}_A (\mathbb{Q}(\mu_{p^\infty}))^\vee 
\longrightarrow 
\varprojlim_n H^2 (\mathbb{Q}_{\Sigma}/\mathbb{Q}(\mu_{p^n}),T^\ast (1)) \longrightarrow 0 
\end{multline}
such that 
\begin{enumerate}
\item[{(i)}]
The sequence \eqref{equation:4termsequence1}$\otimes_{\mathcal{O}_{{K}}} {{K}} 
$ is exact. 
\item[(ii)]  
The last two terms are torsion over (every local component of) $\Lambda_{\mathrm{cyc}}$.  
\item[(iii)]
The first two terms are of rank one over (every local component of) 
$\Lambda_{\mathrm{cyc}}$. 
\end{enumerate} 
Here, the assertion (i) is nothing but the Poitou-Tate exact sequence of Galois cohomology 
except that the injectivity of the map $\overline{\mathrm{loc}}_p$ 
is due to Kato's result on the Beilinson-Kato Euler system in \cite{320}. Usually the Pontrjagin dual of 
$H^2 ( \mathbb{Q} (\mu_{p^\infty}) ,A)$ contributes to the kernel of $\overline{\mathrm{loc}}_p$. 
However Kato's result assures that this group is a cotorsion 
$\Lambda_{\mathrm{cyc}}$-module. Since the cohomological dimension of $G_{\mathbb{Q} (\mu_{p^\infty})}$ is two, 
the Pontrjagin dual of $H^2 ( \mathbb{Q} (\mu_{p^\infty}) ,A)$ is $p$-torsion free. This proves that 
$H^2 ( \mathbb{Q} (\mu_{p^\infty}) ,A) =0$. 
\par 
The assertion (ii) is also a consequence of a deep result 
by Kato on the Beilinson-Kato Euler system \cite{320}. 
\par 
For the assertion (iii), the rank of the first term is obtained as an application of the 
semi-global Euler-Poincar\'{e} characteristic formula of Galois cohomology theory as well as the fact that 
$\varprojlim_n H^2 (\mathbb{Q}_{\Sigma}/\mathbb{Q}(\mu_{p^n}),T^\ast (1))$ 
is torsion, the rank of the second term is obtained as an application of the 
local Euler-Poincar\'{e} characteristic formula of Galois cohomology theory.  
\par 
Now, we take the (first layer of the) Beilinson--Kato Euler system 
$$
\varprojlim_n z_n \in \varprojlim_n H^1 (\mathbb{Q}_{\Sigma}/\mathbb{Q}(\mu_{p^n}),T^\ast (1)). 
$$ 
Then the sequence \eqref{equation:4termsequence1} 
induces the following sequence of $\Lambda_{\mathrm{cyc}}$-modules: 
\begin{multline}\label{equation:4termsequence2}
0 \longrightarrow 
 \varprojlim_n H^1 (\mathbb{Q}_{\Sigma}/\mathbb{Q}(\mu_{p^n}),T^\ast (1))  
 \Bigl/ \Lambda_{\mathrm{cyc}}\varprojlim_n z_n 
\\  
\overset{\overline{\mathrm{loc}}_p}{\longrightarrow} 
\varprojlim_n  
\dfrac{H^1 (\mathbb{Q}_p (\mu_{p^n}) ,T^\ast (1))}
{\mathrm{Im}(H^1 (\mathbb{Q}_p (\mu_{p^n}) ,F^+_p T^\ast (1)))}
 \Bigl/ \Lambda_{\mathrm{cyc}}\varprojlim \overline{\mathrm{loc}}_p( z_n ) 
\\ 
\longrightarrow \mathrm{Sel}_A (\mathbb{Q}(\mu_{p^\infty}))^\vee 
\longrightarrow 
\varprojlim_n H^2 (\mathbb{Q}_{\Sigma}/\mathbb{Q}(\mu_{p^n}),T^\ast (1)) \longrightarrow 0 , 
\end{multline}
which becomes exact after the base extension $\otimes_{\mathcal{O}_{{K}}} {{K}} 
$. Now, the following theorem is a special case of
the general machinary of Euler system bounds established by 
Kato \cite{305}, Perrin-Riou \cite{398} and Rubin \cite{420} independently 
for Galois representations of arbitrary rank: 
\begin{thm}[Kato, Perrin-Riou, Rubin] \label{theorem:KatoRubinPerrinRIou}
Let us assume the setting of Theorem C of \S \ref{sec:Intro_motif}. 
Suppose further that the following conditions are satisfied: 
\begin{enumerate}
\item[(i)] 
The first layer of a given Euler system $\varprojlim_n   z_n \in \varprojlim_n H^1 (\mathbb{Q}_{\Sigma}/\mathbb{Q}(\mu_{p^n}),T^\ast (1))$ is not contained 
in the $\Lambda_{\mathrm{cyc}}$-torsion part of 
$\varprojlim_n H^1 (\mathbb{Q}_{\Sigma}/\mathbb{Q}(\mu_{p^n}),T^\ast (1))$. 
\item[(ii)] The Galois representation $G_{\mathbb{Q}} \longrightarrow \mathrm{Aut}_{\mathbb{Z}_p}V_f \cong GL_2 (\mathbb{Q}_p)$ contains a nontrivial unipotent element \footnote{This condition is equivalent to the condition that $f$ has no complex multiplication.}. 
\end{enumerate}
Then, $\varprojlim_n H^2 (\mathbb{Q}_{\Sigma}/\mathbb{Q}(\mu_{p^n}),T^\ast (1)) $ is a torsion $\Lambda_{\mathrm{cyc}}$-module and we have:  
\begin{multline}\label{euqtion:bound_cyclo_euler}
\mathrm{char}_{\Lambda_{\mathrm{cyc}}\otimes_{\mathbb{Z}_p} {{K}}}
\left( 
\varprojlim_n H^1 (\mathbb{Q}_{\Sigma}/\mathbb{Q}(\mu_{p^n}),T^\ast (1))  
 \Bigl/ \Lambda_{\mathrm{cyc}}\varprojlim z_n  
 \right) 
\otimes_{\mathcal{O}_{{K}}} {{K}}
 \\ 
 \subset 
 \mathrm{char}_{\Lambda_{\mathrm{cyc}} \otimes_{\mathbb{Z}_p} {{K}}} 
\left( \varprojlim_n H^2 (\mathbb{Q}_{\Sigma}/\mathbb{Q}(\mu_{p^n}),T^\ast (1)) \right) \otimes_{\mathcal{O}_{{K}}} {{K}}. 
\end{multline}
\end{thm}
On the other hand, we have the following theory of Coleman maps: 
\begin{thm}[\cite{ochiai-AJM03}, \cite{396}] \label{theorem:colemanmap}
Let us assume the setting of Theorem C of \S \ref{sec:Intro_motif}. 
\begin{enumerate}
\item 
There is a $\Lambda_{\mathrm{cyc}}\otimes_{\mathbb{Z}_p} {{K}}$-linear homomorphism: 
$$ 
\mathrm{Col}:\, 
\left( \varprojlim_n  {H^1 (\mathbb{Q}_p (\mu_{p^n}) ,T^\ast (1))}
\right) \otimes_{\mathcal{O}_{{K}}} {{K}} 
\longrightarrow \mathrm{Fil}^0 D_{\mathrm{dR}}(V_f^\ast  (1))
\otimes_{\mathbb{Z}_p} \Lambda_{\mathrm{cyc}}
$$ 
such that we have the following commutative diagram 
for any positive integer $j$ and for any Dirichlet 
character $\phi$ whose conductor $\mathrm{Cond} (\phi)$ is a power of $p$: 
$$
\begin{CD}
\left( 
\varprojlim_n  {H^1 (\mathbb{Q}_p (\mu_{p^n}) ,T^\ast (1))} 
\right) \otimes_{\mathcal{O}_{{K}}} {{K}} 
@>{\mathrm{Col}}>>  \mathrm{Fil}^0 D_{\mathrm{dR}}((V_f)^\ast  (1))
\otimes_{\mathbb{Z}_p} \Lambda_{\mathrm{cyc}} \\ 
@V{\chi^{-j}_{\mathrm{cyc}}\phi^{-1}}VV @VV{\chi^{-j}_{\mathrm{cyc}}\phi^{-1}}V \\ 
 {H^1 (\mathbb{Q}_p ,(V_f \otimes \phi )^\ast (1-j))}
@>>{e_p (f, j,\phi )\times \mathrm{exp}^{\ast}}>  \mathrm{Fil}^0 D_{\mathrm{dR}}((V_f \otimes \phi )^\ast  (1-j))
\end{CD}
$$
where $\mathrm{exp}^{\ast}$ is the dual exponential map of Bloch-Kato and 
$e_p (f, j,\phi )$ is defined as follows$:$ 
$$
e_p (f, j,\phi ) =
\begin{cases}  (-1)^j (j-1)! \left( 1-\dfrac{p^{j-1}}{a_p(f)} \right)
& \text{when $\phi = \text{\boldmath$1$}$}, \medskip \\ 
(-1)^j (j-1)! \left( \dfrac{p^{j-1}}{a_p(f)} \right)^n &  \text{when $\phi \not= \text{\boldmath$1$}$ $($with conductor $p^n >1 )$}. 
\end{cases}
$$  
\item 
Recall that there is a natural homomorphism 
$$
\mathrm{P}: \left( 
\varprojlim_n  {H^1 (\mathbb{Q}_p (\mu_{p^n}) ,T^\ast (1))}
\right) \otimes_{\mathcal{O}_{{K}}} {{K}}
\longrightarrow 
\left( \varprojlim_n  
\dfrac{H^1 (\mathbb{Q}_p (\mu_{p^n}) ,T^\ast (1))}
{\mathrm{Im}(H^1 (\mathbb{Q}_p (\mu_{p^n}) ,F^+_p T^\ast (1)))} 
\right) \otimes_{\mathcal{O}_{{K}}} {{K}}
$$ 
Then, the map 
$$ 
\mathrm{Col}:\, \left( 
\varprojlim_n  {H^1 (\mathbb{Q}_p (\mu_{p^n}) ,T^\ast (1))}
\right) \otimes_{\mathcal{O}_{{K}}} {{K}}
\longrightarrow \mathrm{Fil}^0 D_{\mathrm{dR}}(V_f^\ast  (1))
\otimes_{\mathbb{Z}_p} \Lambda_{\mathrm{cyc}}
$$ 
is factorized as $\overline{\mathrm{Col}}\circ \mathrm{P}$ where 
$\overline{\mathrm{Col}}$ is a Coleman map  
$$ 
\overline{\mathrm{Col}}:\, \left( \varprojlim_n  
\dfrac{H^1 (\mathbb{Q}_p (\mu_{p^n}) ,T^\ast (1))}
{\mathrm{Im}(H^1 (\mathbb{Q}_p (\mu_{p^n}) ,F^+_p T^\ast (1)))} 
\right) \otimes_{\mathcal{O}_{{K}}} {{K}}
\longrightarrow \mathrm{Fil}^0 D_{\mathrm{dR}}(V_f^\ast  (1))
\otimes_{\mathbb{Z}_p} \Lambda_{\mathrm{cyc}} 
$$ 
which interpolates \begin{multline*}
\mathrm{exp}^{\ast}:\ H^1 (\mathbb{Q}_p ,(V_f \otimes \phi )^\ast (1-j)/F^+_p 
(V_f \otimes \phi )^\ast (1-j)) \\ 
\longrightarrow D_{\mathrm{dR}}((V_f \otimes \phi )^\ast (1-j)/F^+_p 
(V_f \otimes \phi )^\ast (1-j)) \cong \mathrm{Fil}^0 D_{\mathrm{dR}}((V_f \otimes \phi )^\ast  (1-j)) 
\end{multline*} 
with the same range of $(j,\phi )$ and the same multipliers $e_p (f,j\phi )$ as $\mathrm{Col}$. 
When $a_p (f) \not=1$, the map $\overline{\mathrm{Col}}$ 
is an isomorphism. When $a_p (f) =1$ \footnote{This happens only when 
the weight of $f$ is equal to $2$.}, 
$\mathrm{Ker} (\overline{\mathrm{Col}})$ and $\mathrm{Coker} (\overline{\mathrm{Col}})$ are both ${K}$-vector spaces of dimension one. 
\end{enumerate}
\end{thm}
Here are several remarks on Theorem \ref{theorem:colemanmap}. 
\begin{rem}\label{remark:Colemanmap_singlemodularform}
\begin{enumerate}
\item 
Perrin-Riou (\cite{396}, \cite{Pe01}) constructs Coleman maps similar to Theorem \ref{theorem:colemanmap} for lattices $T$ of 
more general $p$-adic representations $V$ than $V_f$ 
(In her work, $V$ is of arbitrary rank and $V$ is not necessarily ordinary). 
However, the paper \cite{396} requires an assumption that the representation 
$V$ is crystalline. The result was later generalized in her paper 
\cite{Pe01}. Though the condition gets milder in \cite{Pe01}, 
there still remains an assumption of $V$ being semi-stable. 
\item 
Unfortunately, the representation $V_f$ associated to an ordinary cuspform $f$ is not semi-stable when the conductor of Neben character of $f$ is divisible by $p$ and the results of \cite{396} and \cite{Pe01} can not be applied to lattices of $V_f$ in such cases. The author 
\cite{ochiai-AJM03} constructs Coleman map as in Theorem \ref{theorem:colemanmap} in a different manner as \cite{396} 
(see also Remark \ref{remark:Colemanmap_Hidafamily}) and \cite{Pe01} 
and the result in \cite{ochiai-AJM03} is also valid for ordinary eigen cuspforms $f$ whose associated to Galois representations $V_f$ are not semi-stable. 
\item We remark that the results of Perrin-Riou (\cite{396}, \cite{Pe01}) are the interpolation of exponential maps of Bloch-Kato rather than 
the interpolation of dual exponential maps of Bloch-Kato as we formulated above. However, once the interpolation of exponential maps are constructed 
in a suitable manner, we can recover the above result by taking ``Kummer dual'' of the above sequence. 
\end{enumerate}
\end{rem}
Below, we will give a brief sketch of the proof of Theorem \ref{theorem:colemanmap} according to the idea of \cite{ochiai-AJM03}, which covers Galois representations $V_f$ associated to ordinary eigen cuspforms of any level and simplifies the proof even in the crystalline case. 
\begin{proof}[Brief sketch of the proof of Theorem \ref{theorem:colemanmap}]
The first observation is that the kernel of the dual exponential map 
$$
\mathrm{exp}^{\ast}_{(V_f \otimes \phi )^\ast (1-j)}:\ H^1 (\mathbb{Q}_p ,(V_f \otimes \phi )^\ast (1-j)) \longrightarrow \mathrm{Fil}^0 D_{\mathrm{dR}}((V_f \otimes \phi )^\ast  (1-j))
$$ 
is isomorphic to the image of $H^1 (\mathbb{Q}_p , F^+_p (V_f \otimes \phi )^\ast (1-j)) \longrightarrow 
H^1 (\mathbb{Q}_p ,  (V_f \otimes \phi )^\ast (1-j))$ under certain techn ical condition. Hence we prove that an expected Coleman map factors through 
the module $\left( \varprojlim_n  
\dfrac{H^1 (\mathbb{Q}_p (\mu_{p^n}) ,T^\ast (1))}
{\mathrm{Im}(H^1 (\mathbb{Q}_p (\mu_{p^n}) ,F^+_p T^\ast (1)))} 
\right) \otimes_{\mathcal{O}_{{K}}} {{K}}$ and we reduce the statement (1) to (2). 
\par 
The second observation is that $T^\ast (1-j)/F^+ T^\ast (1-j)$ is isomorphic to $\mathcal{O}_{{K}} (1-j) $ 
as $G_{\mathbb{Q}^{\mathrm{ur}}_p}$-module. Hence 
$H^1 (\mathbb{Q}^{\mathrm{ur}}_p (\mu_{p^n}) ,T^\ast (1)/F^+ T^\ast (1-j)) $ is isomorphic to 
$H^1 (\mathbb{Q}^{\mathrm{ur}}_p (\mu_{p^n}) , \mathcal{O}_{{K}} (1-j) ) $. 
By the local Tate duality theorem, the problem of the interpolation of 
the dual exponential maps 
$$ 
\mathrm{exp}^{\ast} :\ H^1 (\mathbb{Q}^{\mathrm{ur}}_p ,K (1-j) \otimes \phi^{-1} ) 
\longrightarrow D^{\mathrm{ur}}_{\mathrm{dR}}(K (1-j) \otimes \phi^{-1} ) 
$$ 
is reduced to the problem of the interpolation of 
the exponential maps 
$$ 
\mathrm{exp} :\ 
D^{\mathrm{ur}}_{\mathrm{dR}}(K (j) \otimes \phi ) \longrightarrow 
H^1 (\mathbb{Q}^{\mathrm{ur}}_p ,K (j) \otimes \phi ). 
$$ 
By the Kummer theory, the existence of a Coleman map
$$ 
\overline{\mathrm{Col}}^{\mathrm{ur}}:\ 
\left( 
\varprojlim_n  {H^1 (\mathbb{Q}^{\mathrm{ur}}_p (\mu_{p^n}) ,T^\ast (1))}
\right) \otimes_{\mathcal{O}_{{K}}} {{K}}
\longrightarrow \mathrm{Fil}^0 D^{\mathrm{ur}}_{\mathrm{dR}}(V_f^\ast  (1))
\otimes_{\mathbb{Z}_p} \Lambda_{\mathrm{cyc}}
$$
which interpolates 
$\mathrm{exp}^{\ast}:\ H^1 (\mathbb{Q}^{\mathrm{ur}}_p ,(V_f \otimes \phi )^\ast (1-j)) \longrightarrow \mathrm{Fil}^0 D^{\mathrm{ur}}_{\mathrm{dR}}((V_f \otimes \phi )^\ast  (1-j))$ only for $j=1$ follows from the theory of Coleman power series over $\mathbb{Q}^{\mathrm{ur}}_p (\mu_{p^\infty}) / \mathbb{Q}^{\mathrm{ur}}_p$. 
For the interpolation in the range $j>1$, we can compare the (dual) exponential map for $\mathbb{Z}_p$ 
and the (dual) exponential map by the explicit reciprocity law (see \cite[Propposition 2.4.3]{396} on the explicit reciprocity law for general crystalline representations $T$). Finally, by careful arguments of coming down to $\mathbb{Q}_p$ 
by taking the $\mathrm{Gal} (\mathbb{Q}^{\mathrm{ur}}_p / \mathbb{Q}_p)$-invariant, we recover the desired interpolation result.  
For further details of this quick sketch of the proof, we refer the reader to \cite{ochiai-AJM03}. 
\end{proof}

Finally, by combining the above results, we deduce the inequality 
\eqref{equation:IMC_for_elliptic_cuspform_modulo_mu2}. First, 
the exactness of the sequence \eqref{equation:4termsequence2}
$\otimes_{\mathcal{O}_{{K}}} {{K}}$ and the inclusion \eqref{euqtion:bound_cyclo_euler}, implies 
\begin{multline}\label{equation:bound_middleterms}
\mathrm{char}_{\Lambda_{\mathrm{cyc}} \otimes_{\mathbb{Z}_p} {{K}}}
\left( \varprojlim_n  
\dfrac{H^1 (\mathbb{Q}_p (\mu_{p^n}) ,T^\ast (1))}
{\mathrm{Im}(H^1 (\mathbb{Q}_p (\mu_{p^n}) ,F^+_p T^\ast (1)))}
 \Bigl/ \Lambda_{\mathrm{cyc}}\varprojlim \overline{\mathrm{loc}}_p( z_n ) 
\right) 
\otimes_{\mathcal{O}_{{K}}} {{K}}
\\ 
\subset 
\mathrm{char}_{\Lambda_{\mathrm{cyc}} \otimes_{\mathbb{Z}_p} {{K}}} 
\, \, \mathrm{Sel}_A (\mathbb{Q}(\mu_{p^\infty}))^\vee 
\otimes_{\mathcal{O}_{{K}}} {{K}}
\end{multline}
Also, it is known that the Beilinson-Kato Euler system 
is related to the special-value of the Hecke $L$-function 
$L(f , \phi^{-1},s)$. 
\par 
For example, let us denote by $z_{1 ,  \chi^{-j}_{\mathrm{cyc}}\phi^{-1}} \in 
H^1 (\mathbb{Q}_{\Sigma}/\mathbb{Q},T^\ast (1) \otimes \chi^{-j}_{\mathrm{cyc}}\phi^{-1}) $ 
the image of twisted element 
$$
( \varprojlim z_n ) \otimes \chi^{-j}_{\mathrm{cyc}}\phi^{-1}
\in 
\varprojlim_n H^1 (\mathbb{Q}_{\Sigma}/\mathbb{Q}(\mu_{p^n}),T^\ast (1) \otimes \chi^{-j}_{\mathrm{cyc}}\phi^{-1})
$$
by the corestriction map. 
Then the element $z_{1 ,  \chi^{-j}_{\mathrm{cyc}}\phi^{-1}}$ is related to 
the $L$-value in the following manner: 
\begin{equation}\label{equation:specialvalue_at_bottom}
\mathrm{exp^\ast} (\mathrm{loc}_p ( 
z_{1 ,  \chi^{-j}_{\mathrm{cyc}}\phi^{-1}})) 
= \tau (\phi) 
\dfrac{L_{(p)}(f , \phi^{-1},j)}
{(2\pi \sqrt{-1})^{j}\Omega^{\mathrm{sgn}(j,\phi )}_{f,\infty} } \cdot 
\overline{f} \otimes \phi^{-1}
\end{equation}
where $L_{(p)}(f , \phi^{-1},s)$ is the Hecke $L$-function of $f$ twisted by 
$\phi^{-1}$ whose $p$-Euler factor is removed 
and $\overline{f}$ is a dual modular form of $f$ whose $q$-expantion 
is given by 
$\overline{f}= \sum^\infty_{n=0} \overline{a_n (f)}q^n$. We note that 
$\overline{f} \otimes \phi^{-1}$ defines a ${K}$-vector space $\mathrm{Fil}^0 D_{\mathrm{dR}}((V_f \otimes \phi )^\ast  (1-j))$ of dimension one. 
\par 
Now if we apply the interpolation property of Theorem \ref{theorem:colemanmap} to the Beilinson-Kato euler system, we obtain 
\begin{multline}\label{equation:bound_middleterms2}
\mathrm{char}_{\Lambda_{\mathrm{cyc}} \otimes_{\mathbb{Z}_p} {{K}}}
\left( \varprojlim_n  
\dfrac{H^1 (\mathbb{Q}_p (\mu_{p^n}) ,T^\ast (1))}
{\mathrm{Im}(H^1 (\mathbb{Q}_p (\mu_{p^n}) ,F^+_p T^\ast (1)))}
 \Bigl/ \Lambda_{\mathrm{cyc}}\varprojlim \overline{\mathrm{loc}}_p( z_n ) 
\right) 
\otimes_{\mathcal{O}_{{K}}} {{K}}
\\ 
= 
(L_p (\{ \Omega^{\pm}_{f,\infty} \})). 
\end{multline}
The equality \eqref{equation:bound_middleterms2} combined with 
Rohrlich's result on the non-triviality of 
$L_p (\{ \Omega^{\pm}_{f,\infty} \})$ explained in Remark \ref{remark:p-adicLforf} implies that the left-hand side of \eqref{equation:bound_middleterms2} is non-zero. 
The final fact combined with \eqref{equation:bound_middleterms} 
implies that $\mathrm{char}_{\Lambda_{\mathrm{cyc}} \otimes_{\mathbb{Z}_p} {{K}}} 
\, \, \mathrm{Sel}_A (\mathbb{Q}(\mu_{p^\infty}))^\vee 
\otimes_{\mathcal{O}_{{K}}} {{K}}$ is non-zero, which 
means that $\mathrm{Sel}_A (\mathbb{Q}(\mu_{p^\infty}))^\vee 
\otimes_{\mathcal{O}_{{K}}} {{K}}
$ is torsion over $\Lambda_{\mathrm{cyc}} \otimes_{\mathbb{Z}_p} {{K}}$.
Thus Theorem A is proved under the setting of \eqref{equation:IMC_for_elliptic_cuspform_modulo_mu2}. 
Finally, the proof of the desired inequality \eqref{equation:IMC_for_elliptic_cuspform_modulo_mu2} follows simply 
by combining \eqref{equation:bound_middleterms} and \eqref{equation:bound_middleterms2}.  
\end{proof}
We finish this section with an another version of Iwasawa Main Conjecture 
proposed by Kato. 
\\ 
\\ 
{\bf Conjecture (another version of Cyclotomic Iwasawa Main Conjecture for $f$)} 
\\ 
(Recall that) we take $f$ to be an ordinary eigen cuspform of weight $k\geq 2$. 
Let $T \subset V_f$ be a Galois stable lattice and we choose 
complex periods $\Omega^{+}_{f,\infty}  , 
\Omega^{-}_{f,\infty}$ to be $p$-optimal with respect to the lattice $T$. 
Then, we have the the following equality of principal ideals in the ring $\Lambda_{\mathrm{cyc}} \otimes_{\mathbb{Z}_p} \mathcal{O}_{{K}}$:  
\begin{multline}\label{equation:IMC_for_elliptic_cuspform2}
\mathrm{char}_{\Lambda_{\mathrm{cyc}} \otimes_{\mathbb{Z}_p} \mathcal{O}_{{K}}}
\left( 
\varprojlim_n H^1 (\mathbb{Q}_{\Sigma}/\mathbb{Q}(\mu_{p^n}),T^\ast (1))  
 \Bigl/ \Lambda_{\mathrm{cyc}}\varprojlim z_n  
 \right) 
 \\ 
 = \mathrm{char}_{\Lambda_{\mathrm{cyc}} \otimes_{\mathbb{Z}_p} \mathcal{O}_{{K}}} 
\left( \varprojlim_n H^2 (\mathbb{Q}_{\Sigma}/\mathbb{Q}(\mu_{p^n}),T^\ast (1)) \right)  .
\end{multline}
\\ 
\\ 
The formulation of this version of Cyclotomic Iwasawa Main Conjecture 
is due to Kato \cite{300}. We have several remarks on the above conjecture. 
\\ 
\begin{rem}
\begin{enumerate}
\item  Thanks to the sequence \eqref{equation:4termsequence2}$\otimes_{\mathcal{O}_{{K}}} {{K}}$, 
the equality \eqref{equation:IMC_for_elliptic_cuspform2} in the above Iwasawa Main Conjecture is equivalent to the equality \eqref{equation:IMC_for_elliptic_cuspform} in the previously stated 
Iwasawa Main Conjecture up to a factor of $(\varpi^r)$, 
where $\varpi$ is a uniformaizer  of $\mathcal{O}_{{K}}$ and 
$r$ is an integer. With a more careful discussion, it is also possible to 
eliminate the ambiguity of a factor of $(\varpi^r)$. 
\par 
The latter Iwasawa Main conjecture makes sense also when the cuspform 
$f$ is not ordinary. 
\item
There are various choices involved in the definition of the Beilinson-Kato Euler system and this 
poses an ambiguity in the left-hand side of \eqref{equation:IMC_for_elliptic_cuspform2},  
by multiplication of a power of $\varpi$. 
Since the right-hand side of \eqref{equation:IMC_for_elliptic_cuspform2} has nothing to do 
with an Euler system, the latter version of Iwasawa Main Conjecture seems to have an 
ambiguity of multiplication by a power of $\varpi$ as it is presented. 
\par 
As a solution, we propose to characterize $\{z_n \}$ to be the one 
which satisfies $\mathrm{Col} (\varprojlim_n \mathrm{loc}_p (z_n ))= 
L_p (\{ \Omega^{\pm}_{f,\infty} \})$ where $\mathrm{Col}$ is the Coleman map 
given in Theorem \ref{theorem:colemanmap} 
and we choose the complex periods 
$\Omega^{\mathrm{sgn}(j,\phi )}_{f,\infty}$ in the interpolation property 
of the $p$-adic $L$-function $L_p (\{ \Omega^{\pm}_{f,\infty} \})$ to be $p$-optimized with respect to the lattice $T$ in an appropriate sense. 
\item 
Sometimes, we call the latter Iwasawa Main Conjecture as 
``Iwasawa Main Conjecture without $p$-adic $L$-function''. However, from what we discussed in (2), 
it is important to note that, even for this version of the conjecture, 
we need a $p$-adic $L$-function implicitly to normalize $\{ z_n \}$ precisely. We would like to stress this point, especially because 
the ambiguity becomes a more serious problem if we consider Iwasawa Main Conjecture for a family of cuspforms $f$, which is the theme of this 
article. 
\end{enumerate}
\end{rem}
\section{Setting of Iwasawa theory for a family of cuspforms}
\label{section:motivation}
In the last section, we briefly explained the formalism and the known results of the one-variable cyclotomic Iwasawa Main Conjecture for $f$. On the other hand, the theory of $p$-adic families of cuspforms are much developed since mid 1980's. First, when $f$ is an ordinary eigen cuspform of weight $\geq 2$, 
there is a family of ordinary eigen cuspforms which contain $f$. 
This family is called a Hida family. Below, we will recall the results of Hida theory briefly and without proof and fix the setting at the same time. 
\subsection{Review of Hida families}\label{subsection:hidafamily}
Before stating the results, we prepare some notations: 
\begin{defn}\label{definition:arithmeticcharacters}
\begin{enumerate}
\item 
Let $\kappa: \mathbb{Z}^\times_p \longrightarrow (\overline{\mathbb{Q}_p})^\times$ be a continuous character. When there is an open subgroup 
$U \subset \mathbb{Z}^\times_p$ such that the restriction $\kappa \vert _U$
of $\kappa$ to $U \subset \mathbb{Z}^\times_p$ coincides with a character 
$x \mapsto x^{w(\kappa)}$ ($x\in U$) with $w= w(\kappa) \in \mathbb{Z}$, we call $\kappa$ 
an {\bf arithmetic character} of weight $w$. 
\item 
If we have a continuous character $\kappa: \mathbb{Z}^\times_p \longrightarrow (\overline{\mathbb{Q}_p})^\times$, 
it induces a specialization map of algebra 
$\kappa: \mathbb{Z}_p [[\mathbb{Z}^\times_p ]]\longrightarrow \overline{\mathbb{Q}_p}$, which we denote by $\kappa$ by abuse of notation. 
When $\kappa$ is an arithmetic character of weight $w$, we call 
$\kappa: \mathbb{Z}_p [[\mathbb{Z}^\times_p ]]\longrightarrow \overline{\mathbb{Q}_p}$ an {\bf arithmetic specialization} of weight $w$. 
\par 
Let $R$ be an algebra which is finite over $\mathbb{Z}_p[[\mathbb{Z}^\times_p]]$. Then, a continuous specialization map  
$\kappa: R \longrightarrow \overline{\mathbb{Q}_p}$ is 
an {\bf arithmetic specialization on $R$} of weight $w$ if $\kappa \vert_{ \mathbb{Z}_p [[\mathbb{Z}^\times_p ]]}$ is an 
arithmetic specialization of weight $w$ in the sense stated above. 
\end{enumerate}
\end{defn}
The following theorem is due to Hida (see \cite{Hi86a} and 
\cite{Hi86b} for two different proofs of the following theorem):  
\begin{thm}[Hida]\label{theorem:hidafamily}
Let $f = \sum^\infty_{n=1} a_n (f) q^n \in S_{k_0} (\Gamma_1 (M))$
be an ordinary normalized  $p$-stabilized eigen cuspform 
of weight $k_0 \geq 2$ and of level $\Gamma_1 (M)$. We denote by ${K}={K}_f$ the 
finite extension of $\mathbb{Q}_p$ obtained by adjoining Fourier 
coefficients $a_n (f)$ to $\mathbb{Q}_p$. 
\par 
Then there are a local domain $\mathbb{I}$ which is finite over 
$\mathcal{O}_{{K}}[[1+p\mathbb{Z}_p]]$ and a formal 
$q$-expansion $\mathbb{F}= \sum^\infty_{n=1} A_n (\mathbb{F})q^n 
\in \mathbb{I}[[q]]$ such that the following properties hold: 
\begin{enumerate}
\item 
For each arithmetic specialization $\kappa$ on $\mathbb{I}$ 
of weight $w(\kappa ) \geq 0$, 
$f_\kappa := \kappa (\mathbb{F}) 
\in \kappa (\mathbb{I}) [[q]]$ is the $q$-expansion of a classical 
ordinary eigen cuspform of weight $w(\kappa )+2$.   
\item There exists an arithmetic specialization $\kappa_0$ on $\mathbb{I}$ 
of weight $k_0 -2$ such that $f_{\kappa_0} =f$. 
\end{enumerate}
\end{thm}
This result gives a deformation $\mathbb{F}$ of a given ordinary form 
$f$. We call a family of ordinary cuspforms $\mathbb{F}$ given in the 
above theorem a {\bf Hida family}. 
\par 
Hida \cite{Hi86b} and Wiles \cite{498} prove the following 
result. 
\begin{thm}[Hida, Wiles]\label{theorem:construction_bigrepresentation}
Let $\mathbb{F}$ be a Hida family in the sense of 
Theorem \ref{theorem:hidafamily}. 
Then we have a Galois deformation $\mathbb{V}_{\mathbb{F}} \cong \mathrm{Frac}(\mathbb{I})^{\oplus 2}$ over the fraction field $\mathrm{Frac}(\mathbb{I})$ of $\mathbb{I}$ equipped with a continuous irreducible 
representation $\rho_{\mathbb{F}} : G_{\mathbb{Q}} \longrightarrow 
\mathrm{Aut}_{\mathrm{Frac}(\mathbb{I})} (\mathbb{V}_{\mathbb{F}})$ unramified outside primes dividing $M$ such that the equality 
$$
\mathrm{Tr} (\rho_{\mathbb{F}} (\mathrm{Frob}_\ell )) = A_\ell (\mathbb{F}) 
$$
holds for every prime $\ell$ not dividing $M$. 
\end{thm} 
\begin{rem}
The field $\mathrm{Frac}(\mathbb{I})$ is not locally compact 
with respect to the topology induced by the maximal ideal of $\mathbb{I}$. 
Hence, it is not clear (and maybe not known) if $V_{\mathbb{F}}$ always has a free Galois stable lattice over $\mathbb{I}$ \footnote{Note that Galois representations with values in $p$-adic field have always a Galois stable lattice over the ring of integers (see \cite[Chap. I \S 1]{438} for example)}.  
\end{rem}
The following local property at $p$ of the representation $\mathbb{V}_{\mathbb{F}}$ is also important.
\begin{thm}[Wiles]\label{theorem:localfiltration}
Let $\mathbb{V}_{\mathbb{F}} \cong \mathrm{Frac}(\mathbb{I})^{\oplus 2}$ be a representation associated to a Hida family $\mathbb{F}$ 
given in Theorem \ref{theorem:construction_bigrepresentation}. 
Then, the representation $\rho_{\mathbb{F}} \vert_{G_{\mathbb{Q}_p}}$  obtained by restricting $\rho_{\mathbb{F}} $ to the decomposition group at $p$ admits the following local filtration on $\mathbb{V}_{\mathbb{F}}$: 
\begin{equation}\label{equation:localfiltrationatp}
0 \longrightarrow F^+ \mathbb{V}_{\mathbb{F}} 
\longrightarrow \mathbb{V}_{\mathbb{F}} 
\longrightarrow \mathbb{V}_{\mathbb{F}} /F^+ \mathbb{V}_{\mathbb{F}}\longrightarrow 0 
\end{equation}
where the submodule $F^+ \mathbb{V}_{\mathbb{F}}$ is of dimension one over $\mathrm{Frac}(\mathbb{I})$, hence $\mathbb{V}_{\mathbb{F}} /F^+ \mathbb{V}_{\mathbb{F}}$ is 
also of dimension one over $\mathrm{Frac}(\mathbb{I})$. 
The rank-one subspace $F^+ \mathbb{V}_{\mathbb{F}}$ is characterized by the property that 
the action of $G_{\mathbb{Q}_p}$ on $F^+ \mathbb{V}_{\mathbb{F}}$ is unramified. 
\end{thm}
We refer the reader to the article \cite{498} for the proof of this theorem. 
\par 
In order to avoid technical complications of working over non-free lattices, we will consider the following conditions: 
\begin{enumerate}
\item[{\bf (F)}] 
We have a free $\mathbb{I}$-submodule $\mathbb{T}\subset \mathbb{V}_{\mathbb{F}}$ of rank two which is stable under the action of $G_{\mathbb{Q}}$. 
\par 
Furthermore, 
the representation $\rho_{\mathbb{F}} \vert_{G_{\mathbb{Q}_p}}$ obtained by restricting $\rho_{\mathbb{F}} $ to the decomposition group at $p$ admits the following local filtration on $\mathbb{T}$: 
\begin{equation}\label{equation:localfiltrationatp2}
0 \longrightarrow F^+ \mathbb{T} 
\longrightarrow \mathbb{T} 
\longrightarrow \mathbb{T} /F^+ \mathbb{T}\longrightarrow 0 
\end{equation}
such that the sequence is stable under the action of $G_{\mathbb{Q}_p}$, 
graded pieces $F^+ \mathbb{T}$ and $\mathbb{T} /F^+ \mathbb{T}$ are free of rank one over $\mathbb{I}$ and the sequence \eqref{equation:localfiltrationatp2} recovers \eqref{equation:localfiltrationatp} by applying the base extension $\otimes_{\mathbb{I}} \mathrm{Frac}(\mathbb{I})$. 
\end{enumerate}
\begin{rem}
\begin{enumerate}
\item 
The first half of the statement of condition (F) holds when the residual  representation $G_{\mathbb{Q}} \longrightarrow \mathrm{Aut}_{
\mathbb{I}/\mathfrak{M}_{\mathbb{I}}} 
(\mathbb{T}/\mathfrak{M}_{\mathbb{I}} \mathbb{T}) \cong GL_2 (\mathbb{F}_q)$ 
is irreducible where $\mathfrak{M}_{\mathbb{I}}$ is the maximal ideal of $\mathbb{I}$. As another example for the first half of 
the condition $(F)$, it suffices that $\mathbb{I}$ is a regular local ring. 
In fact, when we are given a Galois stable lattice $\mathbb{L} \subset \mathbb{V}_{\mathbb{F}}$ which is not necessarily free over $\mathbb{I}$,  we define $\mathbb{T}$ to be the double $\mathbb{I}$-linear dual $\mathbb{L}^{\ast \ast}$ of $\mathbb{L}$. The lattice $\mathbb{T}$ is reflexive module over $\mathbb{I}$. Since $\mathbb{I}$ is of Krull dimension two by definition and is regular by assumption, any reflexive module over $\mathbb{I}$ is free over $\mathbb{I}$. 
\item As a caution, we remark that Theorem \ref{theorem:localfiltration} and 
the first statement of the condition (F) might not be sufficient to imply the second half of the condition (F). However, if the residual representation $\mathbb{T}/\mathfrak{M}_{\mathbb{I}} \mathbb{T}$ decomposed into two different characters as a $G_{\mathbb{Q}_p}$-module, 
Theorem \ref{theorem:localfiltration} and 
the first statement of the condition (F) imply the second half of the condition (F) 
(see \cite[Rem. 2.13]{fo12} for the proof of this fact \footnote{We take this occasion to correct a typo in \cite{fo12}. The sequence obtained by $\otimes_{\mathcal{R}} \mathcal{R}/\mathfrak{M}_{\mathcal{R}}$ in the third line of the proof of \cite[Rem. 2.13]{fo12} is only right exact and the phrase ``Note that the sequence is left-exact since $\mathcal{T}_{\mathfrak{p}}$ is a torsion-free $\mathcal{R}$-module.'' is not correct. However, the argument following this phrase uses only the right exactness of the sequence and the argument remains correct except this point.}). 
\end{enumerate}
\end{rem}
\subsection{Review of Coleman families} 
About ten years later after the Hida theory of elliptic modular forms appeared, 
Coleman (\cite{co96}, \cite{co97}) showed the existence of a family of modular forms deforming non-ordinary eigen cuspforms. 
The construction of Coleman families is more ``analytic'' than that of Hida families and they are obtained as rigid analytic local sections  over the weight space. Thus, Coleman families exist only locally over the weight space while Hida families exist over the whole weight spaces. Below, we will recall some results from  Coleman theory without proof and we fix the setting at the same time. 
\par
Let $\mathcal{W}$ be the weight space over $\mathbb{Q}_p$ whose ${K}$-valued points $\mathcal{W}({K})$ is identified with the set of continuous characters $\mathbb{Z}^\times_p \longrightarrow {K}^\times$. We have an embedding $\mathbb{Z} \hookrightarrow 
\mathcal{W} (\mathbb{Q}_p)$ by sending $k \in \mathbb{Z}$ to the character 
$ \mathbb{Z}^\times_p \longrightarrow \mathbb{Q}^\times_p,\,  
a \mapsto a^k$. The space $\mathcal{W}$ is a $p-1$ copies of open discs of radius $1$ and the ring of power bounded rigid analytic functions on $\mathcal{W}$ is isomorphic to $\mathbb{Z}_p [[\mathbb{Z}^\times _p]]$. 
\par 
Now as we mentioned earlier, we need to work locally on $\mathcal{W}$. 
To that end, we introduce some definitions and notations. 
\begin{defn}\label{definition:arithmeticcharacters2}
Let $k_0$ be an integer and $r$ a natural number.
\begin{enumerate}
\item 
We denote by 
$B(k_0 ;r)$ a rigid analytic space 
whose rational points are identified with an open disc of radius 
$\frac{1}{p^r}$ centered at $k_0 \in \mathcal{W} (\mathbb{Q}_p)$ 
with the above mentioned embedding of $\mathbb{Z}$ in $\mathcal{W} (\mathbb{Q}_p)$. 
\item 
The ring of power bounded rigid analytic functions on $B(k_0 ;r)$ is 
noncanonically isomorphic to $\mathbb{Z}_p [[1+p \mathbb{Z}_p]]$ and we denote it by $\Lambda_{(k_0 ;r)}$ in this article
. 
When $\mathcal{O}$ is the ring of the integers of a finite extension of 
$\mathbb{Q}_p$, we denote $\Lambda_{(k_0;r)} \otimes_{\mathbb{Z}_p} \mathcal{O}$ by $\Lambda_{(k_0 ;r),\mathcal{O}}$ for short. 
\item 
A character $\eta : \mathbb{Z}^\times_p \longrightarrow 
\mathbb{Q}^\times_p$ identified with a point $B(k_0 ;r) \subset \mathcal{W} (\mathbb{Q}_p)$, is called an {\bf arithmetic character} if it lies in 
$\mathbb{Z} \cap B(k_0;r)$ and the corresponding integer $k$ is called the {\bf weight} of $\eta$. 
\item  
As explained in Definition \ref{definition:arithmeticcharacters}, 
any continuous character $\eta : \, \mathbb{Z}^\times_p \longrightarrow \mathbb{Q}^\times_p$ extends to a specialization map 
$\mathbb{Z}_p [[\mathbb{Z}^\times _p]] \longrightarrow \mathbb{Z}_p$. 
Further, if $\eta$ is in $B(k_0 ;r) \subset \mathcal{W} (\mathbb{Q}_p)$, the map extends to $\eta : 
\Lambda_{(k_0 ;r),\mathcal{O}} \longrightarrow \mathcal{O}$. Here, we note that the ring $\Lambda_{(k_0 ;r)}$ 
contains one of the local components $\mathbb{Z}_p [[1+p \mathbb{Z}_p]]$ 
of $\mathbb{Z}_p[[\mathbb{Z}^\times_p]]$ as the ring of power bounded rigid analytic functions on $\mathcal{W}$ since $B(k_0 ;r) \subset \mathcal{W}$. 
A specialization $\eta : \Lambda_{(k_0 ;r),\mathcal{O}} \longrightarrow \mathcal{O}$ thus obtained is called an {\bf arithmetic specialization} if 
the corresponding character $\mathbb{Z}^\times_p \longrightarrow \mathbb{Q}^\times_p$ is arithmetic. 
\item  
The set of arithmetic specializations is identified with $\mathbb{Z} \cap B(k_0;r)$. Hence, we denote it by an integer $k \in \mathbb{Z} \cap 
B(k_0 ;r)$ rather than $\eta$ and we call it as an {\bf arithmetic point}.
\end{enumerate}
\end{defn} 
\begin{thm}[Coleman]\label{theorem:colemanfamily}
Let $f = \sum^\infty_{n=1} a_n (f) q^n \in S_{k_0} (\Gamma_1 (M))$
be a normalized  $p$-stabilized eigen cuspform 
of weight $k_0 \geq 2$ and of level $\Gamma_1 (M)$. 
Assume that $a_p (f) \not=0$ and put $s \in \mathbb{Q}_{\geq 0}$ to be 
$s:= \mathrm{val}_p (a_p (f))$ where $\mathrm{val}_p:
(\overline{\mathbb{Q}_p})^\times \longrightarrow \mathbb{Q}_{\geq 0}$ 
be a valuation map normalized by $\mathrm{val}_p (p)=1$. 
When $f$ is a $p$-stabilization of a newform $f_0$ whose level is prime to $p$, we assume that the $p$-th Hecke polynomial $X^2 -a_p(f_0) X+\psi_{f_0} (p)p^{k_0 -1}$ does not have double roots\footnote{In fact, this assumption is conjectured to be always true and the conjecture is already 
proved for $k=2$ by Coleman-Edixhoven
\cite{ce98}.}. 
We denote by ${K}={K}_f$ the 
finite extension of $\mathbb{Q}_p$ obtained by adjoining Fourier 
coefficients $a_n (f)$ to $\mathbb{Q}_p$. 
\par 
Then there exists a natural number $r$ and a formal 
$q$-expansion $\mathbb{F}= \sum^\infty_{n=1} A_n (\mathbb{F}) q^n 
\in \Lambda_{(k_0 ;r),\mathcal{O}_{{K}}}[[q]]$ such that the following properties hold: 
\begin{enumerate}
\item 
At each arithmetic point $k \in \mathbb{Z} \cap B(k_0;r)$ larger than 
$s+1$, $f_{k} := \mathbb{F}(k) 
\in \mathcal{O}_{{K}} [[q]]$ is the $q$-expansion of a classical 
ordinary eigen cuspform of weight $k$. 
\item 
For each arithmetic point $k \in \mathbb{Z} \cap B(k_0;r)$, we have 
$\mathrm{val}_p (a_p (f_k))=s$. 
\item At the arithmetic point $k_0 \in \mathbb{Z} \cap B(k_0;r)$, 
we have $f_{k_0} =f$. 
\end{enumerate}
\end{thm}
We call the number $s$ which appeared in the theorem the {\bf slope} of 
a given Coleman family $\mathbb{F}$. 
\begin{rem}
\begin{enumerate}
\item 
In Definition \ref{definition:arithmeticcharacters}, 
the arithmetic characters remain to be arithmetic after twisting by Dirichlet characters 
of $p$-power conductor. On the other hand, the arithmetic characters 
of Definition \ref{definition:arithmeticcharacters2} do not remain to be arithmetic after twisting by Dirichlet characters 
of $p$-power conductor. The construction of Coleman families does not allow 
``infinitesimal deformation'' (by finite order characters of $p$-power conductor) as in the case of Hida families. 
\item 
In \cite{co96}, \cite{co97}, Coleman works over affinoid spaces. 
Hence, we refer the reader to \cite[\S 2.2]{NO16} for a reformulation of the theory of Coleman families over a complete local ring 
as presented in Theorem \ref{theorem:colemanfamily} \footnote{Note that a typical example of affinoid space is a $p$-adic closed disc 
corresponding to the Tate algebra $\mathbb{Z}_p \langle X \rangle := \varprojlim_n  
(\mathbb{Z}/p^n \mathbb{Z}) [X]$ and a $p$-adic open disc is not an affinoid space.}. 
An affinoid algebra is complete with respect to Guass norm but it is not a complete local ring. 
Hence parallel arguments of gluing as the construction of Beilinson-Kato elements and Coleman maps for Hida families carried out by the author 
do not work over an affinoid algebra. This is why the construction of Coleman families over a formal base ring as presented in Theorem \ref{theorem:colemanfamily} plays important roles in the paper \cite{NO16}. 
\end{enumerate} 
\end{rem}
\ 
\\ 
As an application of the construction of Coleman families over a formal base as presented in Theorem \ref{theorem:colemanfamily}, we can apply 
Wiles' method of pseudo-representation to obtain the following consequence: 
\begin{thm}\label{theorem:construction_bigrepresentation_non-ordinary}
Let $\mathbb{F}$ be a Coleman family in the sense of 
Theorem \ref{theorem:colemanfamily}. 
Then we have a Galois deformation $\mathbb{T} \cong 
\Lambda_{(k_0 ;r),\mathcal{O}_{{K}}}^{\oplus 2}$ equipped with a continuous irreducible 
representation $\rho_{\mathbb{F}} : G_{\mathbb{Q}} \longrightarrow 
\mathrm{Aut}_{\Lambda_{(k_0 ;r),\mathcal{O}_{{K}}}} (\mathbb{T})$  unramified outside primes dividing $M$ such that the equality 
$$
\mathrm{Tr} (\rho_{\mathbb{F}} (\mathrm{Frob}_\ell )) = A_\ell (\mathbb{F}) 
$$
holds for every prime $\ell$ not dividing $M$. 
\end{thm} 
We refer the reader to \cite[\S 2.3]{NO16} for the proof of Theorem\ref{theorem:construction_bigrepresentation_non-ordinary} 
\\ 
\\  
In each of the ordinary case and the non-ordinary case presented above, there is a complete Noetherian local ring $\mathcal{R}$ 
of characteristic $0$ with finite residue field of characteristic $p$ 
and a family of Galois representation $\mathbb{T} \cong \mathcal{R}^{\oplus d}$ 
on which the absolute Galois group $G_{\mathbb{Q}}$ acts continuously. 
Then there should be Iwasawa Main Conjecture over $\mathbb{T}$ as an equality of 
ideals in $\mathcal{R}\widehat{\otimes}_{\mathbb{Z}_p} \Lambda_{\mathrm{cyc}}$ relating an algebraic object and an analytic object
We outlined the setting of $p$-adic deformations of modular forms very quickly for our later use later in this article. 
The setting of eigencurves (cf. \cite{cm98}) which unifies the ordinary case and the non-ordinary case is also known. 
Formulating and proving the results which will be stated in later sections of this paper over eigencurves should be also an interesting problem. 
\  
\\ 
\\ 
\section{Iwasawa Main Conjecture for an ordinary family of cuspforms}
\label{section:motivation2}
\subsection{Statements of known results}
As for our motivation stated in the end of the previous section, 
we have already some conditional results in the ordinary setting given in \ref{subsection:hidafamily}. 
In this section, we shall introduce the results and the proof. Then we will make our motivation more precise 
explaining new  difficulties and new phenomena on this project of ``Iwasawa theory for deformation spaces''. 
\par 
We fix the setting of \S \ref{subsection:hidafamily}. We take a Hida family $\mathbb{F}= \sum^\infty_{n=1} A_n (\mathbb{F})q^n \in \mathbb{I}[[q]]$. We choose and fix an integral 
Galois representation $\mathbb{T} \cong \mathbb{I}^{\oplus 2}$ associated 
to $\mathbb{F}$ by Theorem \ref{theorem:localfiltration} 
which satisfies the condition \rm{(F)}. Similarly as in \S \ref{sec:Intro_motif}, for each number field $K$, we define the Greenberg type Selmer group 
\begin{equation}\label{equation:definition_Selmer2}
\mathrm{Sel}_\mathbb{A} (K) = 
\mathrm{Ker} \left[ 
H^1 (K,\mathbb{A}) \longrightarrow \prod_{\lambda \nmid p} H^1 (I_\lambda ,\mathbb{A})]
\times 
\prod_{\mathfrak{p} \vert p}\dfrac{H^1 (I_\mathfrak{p} ,\mathbb{A} )}
{\mathrm{Image}(H^1 (I_\mathfrak{p} ,F^+_p \mathbb{A}))}\right]
\end{equation}
where $\mathbb{A}$ is a discrete Galois representation 
defined to be $\mathbb{A}=\mathbb{T}\otimes_{\mathbb{I}} \mathbb{I}^\vee$ 
and where $I_\lambda$ (resp. $I_\mathfrak{p}$) denotes the inertia subgroup 
at each place $\lambda$ (resp. $\mathfrak{p}$) not dividing $p$ (resp. dividing $p$). We deduce that the Pontrjagin dual $\mathrm{Sel} (\mathbb{Q}(\mu_{p^\infty}))^\vee$ of $\mathrm{Sel} (\mathbb{Q}(\mu_{p^\infty}))$ 
is finitely generated over $\mathbb{I} \widehat{\otimes}_{\mathbb{Z}_p} \Lambda_{\mathrm{cyc}}$ by a similar argument as the argument of Lemma \ref{lemma:finitenessSel_A} and Lemma \ref{lemma:finitelygenerated}. 
\\
\\  
The following theorem follows from Theorem A of \S \ref{sec:Intro_motif} 
combined with the control theorem of Selmer group specializing from two-variable to one-variable (see \cite{ochiai-JNT01} and \cite{fo12}). 
\\ 
\\ 
{\bf Theorem A$'$ (Torsion property of the Selmer group for a Hida family $\mathbb{F}$)} 
\\ 
Let $\mathbb{F}$, $\mathbb{T}$ and $\mathbb{A}$ be as above. Assume that $p \geq 5$ \footnote{We exclude $p=3$ only because we did not find a complete reference for CM case.}.     
Then the finitely generated $\mathbb{I} \widehat{\otimes}_{\mathbb{Z}_p} \Lambda_{\mathrm{cyc}}$-module $\mathrm{Sel}_{\mathbb{A}} (\mathbb{Q}(\mu_{p^\infty}))^\vee$ is torsion over $\mathbb{I} \widehat{\otimes}_{\mathbb{Z}_p} \Lambda_{\mathrm{cyc}}$. 
\\ 
\\ 
In order to generalize Theorem B of \S \ref{sec:Intro_motif} from a cuspform $f$ to a family of cuspforms $f_\kappa$, it is important to  understand ``periods''. In fact, a complex period $\Omega^{\mathrm{sgn}}_{f_\kappa,\infty}$ is not canonical at all and it is defined only modulo 
multiplication by elements of $(\mathbb{Q}_{f_\kappa})^\times$ at each $\kappa$. Hence, it does not make sense to  formulate a family of $p$-adic $L$-functions for cuspforms $f_\kappa$, without ``regularizing a choice of periods 
$\Omega^{\mathrm{sgn}}_{f_\kappa ,\infty}$''. 
\par 
Greenberg-Stevens \cite{gr93} 
and Kitagawa \cite{ki94} constructed an $\mathbb{I}$ module called ``$\mathbb{I}$-adic modular symbols'' to overcome the problem of ``regularization of periods'' and they constructed two-variable $p$-adic $L$-functions for given Hida families. Though the constructions of 
\cite{gr93} and \cite{ki94} are similar, the former interpolates parabolic cohomologies and the construction is only local over $\mathbb{I}$. The latter interpolates compactly supported cohomologies and the construction is global over the whole $\mathbb{I}$. 
\par 
Hence, we follow more closely the construction of \cite{ki94} and we denote the module of $\mathbb{I}$-adic modular symbols by $\mathbb{MS}(\mathbb{I})^\pm$. Though we do not recall the definition of $\mathbb{MS}(\mathbb{I})^\pm$, we remark that 
$\mathbb{MS}(\mathbb{I})^\pm$ is the same as the one denoted by $\mathcal{MS}(\mathbb{I})^\pm[\lambda]$ in \cite{ki94}. As an important property of 
$\mathbb{MS}(\mathbb{I})^\pm$, for each arithmetic specialization $\kappa$ of non-negative weight $w(\kappa) \geq 0$, 
$\mathbb{MS}(\mathbb{I})^\pm \otimes_{\mathbb{I}} \kappa (\mathbb{I})$ is canonically identified with 
$H^1_c (Y_1 (M)_{\mathbb{C}} ,
\mathcal{L}_{k_f -2} (\mathbb{Q}_f))^{\pm}[f_\kappa  ]$. 
\begin{defn}\label{definition:padicperiodMS}
Suppose that $\mathbb{MS}(\mathbb{I})^\pm$ is free of rank one over 
$\mathbb{I}$ 
for each of the signs $\pm$ \footnote{Some sufficient conditions so that $\mathbb{MS}(\mathbb{I})$ becomes free of rank one over $\mathbb{I}$ are listed 
in \cite{ki94}. For example, this holds if $\mathbb{I}$ is a UFD.}. Let us fix a basis $\Xi^\pm$ of $\mathbb{MS}(\mathbb{I})^\pm$ over $\mathbb{I}$ for each of the signs $\pm$. Then, for each arithmetic specialization $\kappa$ of non-negative weight $w(\kappa) \geq 0$, the specialization 
$\mathbb{MS}(\mathbb{I})^\pm \otimes_{\mathbb{I}} \kappa (\mathbb{I})$ 
is naturally identified with a lattice of $H^1_c (Y_1 (N_{f_\kappa})_{\mathbb{C}} ,
\mathcal{L}_{w(\kappa )} (\mathbb{Q}_{f_\kappa}))^{\pm}[f_\kappa ] \otimes_{\mathbb{Q}_f} 
\mathrm{Frac}(\kappa (\mathbb{I}))$. We define $p$-adic periods $C^{\pm}_{f_\kappa ,p}$ to be error terms given by: 
\begin{equation}\label{equation:definition_of_p-adic periods}
\kappa (\Xi^\pm) = C^{\pm}_{f_\kappa ,p} \cdot b^\pm _{f_\kappa} \otimes 1 .
\end{equation}
\end{defn}
\begin{rem}\label{remark:padicperiodMS}
As was cautioned earlier, $C^{\pm}_{f_\kappa ,p} $ and 
$\Omega^{\pm}_{f_\kappa ,\infty} $ 
depend on the choice of a $\mathbb{Q}_{f_\kappa}$-basis $b^\pm_{f_\kappa}$ on $H^1_c (Y_1 (M)_{\mathbb{C}} ,
\mathcal{L}_{k_f -2} (\mathbb{Q}_f))^{\pm}[f_\kappa ]$. 
However, the ``ratio'' 
is independent of the choice of $b^\pm_{f_{\kappa}}$. If we denote the $p$-adic period and the complex period obtained by another 
$\mathbb{Q}_{f_\kappa}$-basis $(b^\pm_{f_{\kappa}})'$ on $H^1_c (Y_1 (M)_{\mathbb{C}} , 
\mathcal{L}_{k_f -2} (\mathbb{Q}_f))^{\pm}[f_\kappa ]$ by 
$(C^{\pm}_{f_\kappa ,\infty} )'$ and 
$(\Omega^{\pm}_{f_\kappa ,\infty} )'$, we have 
$$
\dfrac{C^{\pm}_{f_\kappa ,p}}{(C^{\pm}_{f_\kappa ,p})'} =
\dfrac{\Omega^{\pm}_{f_\kappa ,\infty}}{(\Omega^{\pm}_{f_\kappa ,\infty})'}. 
$$ 
Thus the interpolation property of two-variable ordinary $p$-adic $L$-functions stated in Theorem B$'$ below is well-defined. 
\end{rem}
In order to state the result on two-variable ordinary $p$-adic $L$-function, we prepare the following condition.
\begin{enumerate}
\item[{\bf (M)}] 
The $\mathbb{I}$-module $\mathbb{MS}(\mathbb{I})^\pm$ are free of rank one 
over $\mathbb{I}$. 
\end{enumerate} 
\ 
\\
Here is the main result of \cite{ki94}. 
\\ 
\\ 
{\bf Theorem B$'$ (Existence of two-variable ordinary $p$-adic $L$-function)} 
\\ 
Let $\mathbb{F}$, $\mathbb{T}$ and $\mathbb{A}$ be as above. Assume that $p \geq 5$. We assume the condition \rm{(M)} and fix an $\mathbb{I}$-basis 
$\Xi^\pm$ of $\mathbb{MS}(\mathbb{I})^\pm$ respectively. 
\par 
Then, there exists an analytic $p$-adic $L$-function $L_p (\{ \Xi^\pm \}) 
\in \mathbb{I} \widehat{\otimes}_{\mathbb{Z}_p} \Lambda_{\mathrm{cyc}} 
\otimes_{\mathbb{Z}_p} \mathbb{Q}_p$  
such that 
we have the following equality:
\begin{equation}\label{equation:interpolation_AV_V_def}
\dfrac{( \chi^j_{\mathrm{cyc}} \phi , \kappa )(L_p (\{ \Xi^\pm \} ))}
{C^{\mathrm{sgn}(j,\phi )}_{f_\kappa ,p}} = (-1)^{j} (j-1)!  e_p (f_\kappa , j ,\phi ) 
 \tau (\phi) \dfrac{L(f_\kappa , \phi^{-1},j)}
{(2\pi \sqrt{-1})^{j}\Omega^{\mathrm{sgn}(j,\phi )}_{f_\kappa ,\infty} } ,
\end{equation}
for any arithmetic specialization of $\mathbb{I}$ with $w(\kappa ) \geq 0$,  for any integer $j$ satisfying $1 \leq j\leq w(\kappa )+1$ and for any Dirichlet 
character $\phi$ whose conductor $\mathrm{Cond} (\phi)$ is a power of $p$,
where $\tau (\phi )$ and $e_p (f, j, \phi )$ is the same as Theorem B. 
\\ 
\begin{rem}
\begin{enumerate}
\item 
The $p$-adic $L$-function $L_p (\{ \Xi^\pm \} ))$ is expected to be 
in $\mathbb{I} \widehat{\otimes}_{\mathbb{Z}_p} \Lambda_{\mathrm{cyc}} \subset (\mathbb{I} \widehat{\otimes}_{\mathbb{Z}_p} 
\Lambda_{\mathrm{cyc}} )\otimes_{\mathbb{Z}_p} \mathbb{Q}_p$ when the complex period is $p$-optimal way. 
By lack of appropriate reference for its integrality, 
we introduced $L_p (\{ \Xi^\pm \}) $ as an element of 
$(\mathbb{I} \widehat{\otimes}_{\mathbb{Z}_p} \Lambda_{\mathrm{cyc}} )
\otimes_{\mathbb{Z}_p} \mathbb{Q}_p$. 
\par 
 When the residual Galois representation associated to the Hida family $\mathbb{F}$ is irreducible, we can check that $L_p (\{ \Xi^\pm \}) 
\in \mathbb{I} \widehat{\otimes}_{\mathbb{Z}_p} \Lambda_{\mathrm{cyc}} $.   
In fact, if we choose a complex period $\Omega^{\mathrm{sgn}(j,\phi )}_{f_\kappa ,\infty}$ to be $p$-optimal, 
the special value $\dfrac{L(f_\kappa , \phi^{-1},j)}
{(2\pi \sqrt{-1})^{j}\Omega^{\mathrm{sgn}(j,\phi )}_{f_\kappa ,\infty} }$ is $p$-integral for every $\kappa$ and $p$-adic period 
$C^{\mathrm{sgn}(j,\phi )}_{f_\kappa ,p}$ is a $p$-adic unit for every $\kappa$ in the resudually irreducible case. 
Thus $L_p (\{ \Xi^\pm \}) $ has to be integral since its specializations are $p$-integral.
\item 
The $p$-adic $L$-function $L_p (\{ \Xi^\pm \})$ depends on the choice of $\{ \Xi^\pm \}$. 
However, since $\Xi^+$ and $\Xi^-$ are unique modulo multiplication by 
elements of $\mathbb{I}^\times$, the principal ideal $(L_p (\{ \Xi^\pm \}))$ 
in $\mathbb{I} \widehat{\otimes}_{\mathbb{Z}_p} \Lambda_{\mathrm{cyc}}$ 
is independent of the choice of $\{ \Xi^\pm \}$.  
\end{enumerate}
\end{rem}
Now, we can state the two-variable Iwasawa Main conjecture for Hida families of elliptic cuspforms: 
\\ 
\\ 
{\bf Conjecture (Two-variable Iwasawa Main Conjecture for a Hida family)} 
\\ 
Let $\mathbb{F}$, $\mathbb{T}$ and $\mathbb{A}$ be as above. Assume that $p \geq 5$. We assume the condition \rm{(M)} and fix an $\mathbb{I}$-basis 
$\Xi^\pm$ of $\mathbb{MS}(\mathbb{I})^\pm$ respectively. 
Then, we have the following equality of principal ideals in the ring $\mathbb{I} \widehat{\otimes}_{\mathbb{Z}_p} \Lambda_{\mathrm{cyc}}$:  
\begin{equation}\label{equation:IMC_for_Hidafamily}
(L_p (\{ \Xi^\pm \}))=\mathrm{char}_{
\mathbb{I} \widehat{\otimes}_{\mathbb{Z}_p} \Lambda_{\mathrm{cyc}}
} \, \mathrm{Sel}_{\mathbb{A}} (\mathbb{Q}(\mu_{p^\infty}))^\vee . 
\end{equation} 
\\ 
The following result realizes our motivation which was proposed in the end of the previous section: 
\\
\\ 
{\bf Theorem C$'$  (Two-variable Iwasawa Main Conjecture for a Hida family)} 
\\ 
Under certain assumptions (on the prime number $p$, the tame conductor of the Hida family $\mathbb{F}$, the fullness of the residual representation, 
the condition (M), the regularity of the local ring $\mathbb{I}$ etc), we have the following equality of principal ideals in the ring $\mathbb{I} \widehat{\otimes}_{\mathbb{Z}_p} \Lambda_{\mathrm{cyc}}$:    
\begin{equation}\label{equation:IMC_for_elliptic_cuspform_modulo_mu_def}
(L_p (\{ \Xi^\pm \}))=\mathrm{char}_{
\mathbb{I} \widehat{\otimes}_{\mathbb{Z}_p} \Lambda_{\mathrm{cyc}}
} \, \mathrm{Sel}_{\mathbb{A}} (\mathbb{Q}(\mu_{p^\infty}))^\vee . 
\end{equation} 
\ 
\\ 
\subsection{Proofs of known results} 
By the Euler system machinary (cf. \cite{ochiai-AJM03}, \cite{ochiai-AIF} and \cite{Och06}), we prove an inclusion 
\begin{equation}\label{equation:IMC_for_elliptic_cuspform_modulo_mu_bis}
(L_p (\{ \Xi^\pm \})) \subset \mathrm{char}_{\mathbb{I} \widehat{\otimes}_{\mathbb{Z}_p} \Lambda_{\mathrm{cyc}}} \, \mathrm{Sel}_{\mathbb{A}} (\mathbb{Q}(\mu_{p^\infty}))^\vee .
\end{equation}
Once we have this inclusion in the two-variable situation, we may specialize at an arithmetic specialization $\kappa$ of $\mathbb{I}$. Since we have already a 
one-variable equality at $\kappa$ by Theorem C under appropriate conditions, the inclusion \eqref{equation:IMC_for_elliptic_cuspform_modulo_mu_bis} may be upgraded to a one-variable equality (see \cite[\S 7]{Och06} for further details). 
We stress that this strategy implicitly relies on Skinner--Urban \cite{444sui} since the equality in Theorem C is due to the opposite inclusion obtained 
in \cite{444sui} by means of the Eisenstein ideal for $U(2,2)$. Note that Skinner--Urban \cite{444sui} also proves the opposite inequality of  
\eqref{equation:IMC_for_elliptic_cuspform_modulo_mu_bis} for the situation of two-variable. 
\par 
We will discuss a little the proof of the one containment by the Euler system machinary (\cite{ochiai-AJM03}, \cite{ochiai-AIF} and \cite{Och06}) for Theorem C$'$ since 
this should be a good opportunity to explain the difficulties to prove Theorem C$'$ compared to the proof of Theorem C. 
\par 
From now on, we concentrate on the proof of \eqref{equation:IMC_for_elliptic_cuspform_modulo_mu_bis}. 
The proof is composed of two independent steps (the Euler system bounds and the existence of a Coleman map) 
as in the proof of \eqref{equation:IMC_for_elliptic_cuspform_modulo_mu2}
 explained after Remark \ref{remark:Proof_Theorem_C_onevaiable}. However we need some preparations. 
\par 
First, we have to note that a general machinary of Euler system bounds obtained in Theorem \ref{theorem:KatoRubinPerrinRIou} is only for 
usual Galois representation with coefficients in the discrete valuation 
rings with finite residue field. 
In fact, the characteristic ideal of a torsion module $M$ over 
a discrete valuation ring with finite residue field consists only of 
the counting of the finite cardinality of $M$ and 
the proof of Theorem \ref{theorem:KatoRubinPerrinRIou} 
by Kato\cite{305}, Perrin-Riou\cite{398} and Rubin\cite{420} is based on 
the counting. 
\par 
Since we need to generalize Theorem \ref{theorem:KatoRubinPerrinRIou} to 
Galois representations with coefficients in larger deformation rings, 
we need essentially new ideas. Below, we will give the definition of the Euler system 
for Galois deformation and the result on the bound of Galois cohomology groups using this Euler system. 
From now on, we denote by $\Lambda^\sharp_{\mathrm{cyc}}$ a free $\Lambda_{\mathrm{cyc}}$-module of rank one 
on which the absolute Galois group $G_{\mathbb{Q}}$ acts via the character 
$$
G_{\mathbb{Q}} \twoheadrightarrow \mathrm{Gal} (\mathbb{Q} (\mu_{p^\infty})/\mathbb{Q}) \hookrightarrow 
\Lambda^\times_{\mathrm{cyc}}. 
$$ 
\begin{defn}[Def. 2.1 of \cite{ochiai-AIF}] \label{definition:eulersystem_general}
Let $\mathbb{T} \cong \mathcal{R}^{\oplus 2}$ be a Galois representation 
of the absolute Galois group $G_{\mathbb{Q}}$ 
over a complete Noetherian local ring $\mathcal{R}$ of characteristic $0$ 
with finite residue field of characteristic $p$.
Assume that the action of $G_{\mathbb{Q}}$ on $\mathbb{T}$ is continuous 
and unramified outside a finite set of primes 
$\Sigma$ of $\mathbb{Q}$ 
which contain $\{ p\}$, $\{\infty \}$ and the ramified primes of the representation $\mathbb{T}$. 
\par 
We denote by $\mathcal{S}$ the set of all square-free natural numbers 
which are prime to $\Sigma$. 
An Euler system for ${\mathbb{T}}$ is a collection of cohomology classes 
$$
\{ \mathcal{Z}(r) \in H^1(\mathbb{Q} (\mu_r ) ,
(\mathbb{T}\widehat{\otimes}_{\mathbb{Z}_p} \Lambda^\sharp_{\mathrm{cyc}}
)^\ast (1)) \} _{r\in \mathcal{S} }
$$ 
with the following properties \footnote{
The cohomology group $H^1(\mathbb{Q} (\mu_r ) ,
(\mathbb{T}\widehat{\otimes}_{\mathbb{Z}_p} \Lambda^\sharp_{\mathrm{cyc}}
)^\ast (1))$ is isomorphic to $\varprojlim_{n} H^1(\mathbb{Q} (\mu_{p^n r} ) ,\mathbb{T} ^\ast (1))$ by Shapiro's lemma on Galois cohomology theory.}$:$ 
\begin{enumerate}
\item 
The element $\mathcal{Z}(r)$ is unramified outside $\Sigma \cup \{ r \}$ 
for each $r\in \mathcal{S}$. 
\item 
The image of the norm $\mathrm{Norm}_{\mathbb{Q} (\mu_{rq}) /\mathbb{Q} (\mu_{r})}(\mathcal{Z}(rq))$ of $\mathcal{Z}(rq)$
is equal to $P_q (\mathrm{Frob} _q )\mathcal{Z}(r)$, where 
$P_q (X) \in \mathcal{R}
\widehat{\otimes}_{\mathbb{Z}_p} \Lambda_{\mathrm{cyc}} [X]$ is the  polynomial $\mathrm{det} 
(1-\mathrm{Frob}_q X ; \mathbb{T}\widehat{\otimes}_{\mathbb{Z}_p} \Lambda^\sharp_{\mathrm{cyc}} )$ and $\mathrm{Frob}_q$ is a (conjugacy 
class of) geometric Frobenius element at $q$ in the Galois group 
$\mathrm{Gal} (\mathbb{Q} (\mu_r )/\mathbb{Q} )$. 
\end{enumerate}
Here, $\Lambda^\sharp_{\mathrm{cyc}}$ is a free $\Lambda_{\mathrm{cyc}}$-module of rank one on which the absolute Galois group acts via 
a continuous character $G_{\mathbb{Q}}$ acts via the tautological character: 
$$
G_{\mathbb{Q}} \twoheadrightarrow G_{\mathrm{cyc}} \hookrightarrow 
\Lambda_{\mathrm{cyc}}^\times=\mathbb{Z}_p[[G_{\mathrm{cyc}}  ]]^\times. 
$$
\end{defn}
The following theorem is a deformation theoretic generalization of Theorem 
\ref{theorem:KatoRubinPerrinRIou} which gave an Euler system bounds for cyclotomic twists of usual $p$-adic representations. In order to apply the result also to non-ordinary situation later, we do not restrict ourselves 
to Hida families and state the result with a more general setting.   
\begin{thm}[Thm 2.4 of \cite{ochiai-AIF}] \label{theorem:general_Euler_system_bound}
Let us assume the setting of Definition $\ref{definition:eulersystem_general}$. 
Assume further that the following conditions are satisfied: 
\begin{enumerate}
\item[\rm{(i)}]
The residual Galois representation of $\mathbb{T}$ at the maximal ideal of 
$\mathcal{R}$ is absolutely irreducible as a representation of $G_{\mathbb{Q}}$. 
\item[(ii)] 
The first layer of a given Euler system $\mathcal{Z}(1) \in H^1(\mathbb{Q} ,
(\mathbb{T}\widehat{\otimes}_{\mathbb{Z}_p} \Lambda^\sharp_{\mathrm{cyc}}
)^\ast (1)) $ is not contained 
in the $\mathcal{R} \widehat{\otimes}_{\mathbb{Z}_p} \Lambda^{\omega^i}_{\mathrm{cyc}}$-torsion part of 
$H^1(\mathbb{Q}_{\Sigma}/\mathbb{Q} ,
(\mathbb{T}\widehat{\otimes}_{\mathbb{Z}_p} \Lambda^\sharp_{\mathrm{cyc}}
)^\ast (1)) $ at each local component $\mathcal{R} \widehat{\otimes}_{\mathbb{Z}_p} \Lambda^{\omega^i}_{\mathrm{cyc}}$. 
Here, $\omega$ denotes the Teichm\"{u}ller character and we recall that 
$\Lambda_{\mathrm{cyc}}$ is decomposed into a product of local rings 
as $\displaystyle{\Lambda_{\mathrm{cyc}}= \prod^{p-2}_{i=0} }\Lambda^{\omega^i}_{\mathrm{cyc}}$ by character decomposition. 
\item[\rm{(iii)}]
The $\mathcal{R}$-module $\mathbb{T}$ splits into eigenspaces: $\mathbb{T}=\mathbb{T}^+ \oplus \mathbb{T}^-$ with respect to the complex conjugation in $G_{\mathbb{Q}}$, and $\mathbb{T}^+$ (resp. $\mathbb{T}^-$) is free 
$\mathcal{R}$-module of rank one.
\item[(iv)] There exist $\sigma_1 \in G_{\mathbb{Q}(\mu_{p^{\infty}})}$ and $\sigma_2 \in G_{\mathbb{Q}}$ such that $\rho(\sigma_1) $ is conjugate to 
$\begin{pmatrix} 1 & u \\ 0 & 1 \end{pmatrix} \in GL_2(\mathcal{R})$ for a unit $u \in \mathcal{R}^\times$ and $\sigma_2$ acts on $\mathbb{T}$ as multiplication by $-1$
for the Galois representation \footnote{This condition 
excludes the case where $\mathbb{F}$ has complex multiplication.}. 
$$\rho:\, G_{\mathbb{Q}} \longrightarrow \mathrm{Aut}_{\mathcal{R} }(\mathbb{T} ) \cong GL_2 (\mathcal{R}).
$$ 
\item[(v)] 
Every local component $\mathcal{R} \widehat{\otimes}_{\mathbb{Z}_p} \Lambda^{\omega^i}_{\mathrm{cyc}}$ of the semi-local ring $\mathcal{R} \widehat{\otimes}_{\mathbb{Z}_p} \Lambda_{\mathrm{cyc}}$ is isomorphisc to 
a ring of two-variable formal power series $\mathcal{O}[[X,Y]]$ where $\mathcal{O}$ is the ring of integers of a finite extension of $\mathbb{Q}_p$.  
\end{enumerate}
Then, $ H^2(\mathbb{Q}_{\Sigma}/\mathbb{Q} ,
(\mathbb{T}\widehat{\otimes}_{\mathbb{Z}_p} \Lambda^\sharp_{\mathrm{cyc}}
)^\ast (1)) $ is a torsion $\Lambda_{\mathrm{cyc}}$-module and we have:  
\begin{multline}
\mathrm{char}_{\mathcal{R} \widehat{\otimes}_{\mathbb{Z}_p} \Lambda_{\mathrm{cyc}}}
\left( 
H^1(\mathbb{Q}_{\Sigma}/\mathbb{Q} ,
(\mathbb{T}\widehat{\otimes}_{\mathbb{Z}_p} \Lambda^\sharp_{\mathrm{cyc}}
)^\ast (1)) 
 \Bigl/ \mathcal{R} \widehat{\otimes}_{\mathbb{Z}_p} \Lambda_{\mathrm{cyc}}  \cdot \mathcal{Z}(1)  
 \right) 
 \\ 
 \subset 
 \mathrm{char}_{\mathcal{R} \widehat{\otimes}_{\mathbb{Z}_p} \Lambda_{\mathrm{cyc}}} 
\left( H^2(\mathbb{Q}_{\Sigma}/\mathbb{Q} ,
(\mathbb{T}\widehat{\otimes}_{\mathbb{Z}_p} \Lambda^\sharp_{\mathrm{cyc}}
)^\ast (1))  \right) . 
\end{multline}
\end{thm}
\begin{rem}
\begin{enumerate}
\item 
 The paper \cite{ochiai-AIF} considers Euler systems with coefficients in the ring of power series 
in $n$ variables $\mathcal{O}[[X_1 ,\ldots , X_n ]]$ where $\mathcal{O}$ is the ring of integers of a $p$-adic field. 
To overcome an essential difficulty to generalize Theorem \ref{theorem:KatoRubinPerrinRIou} 
explained before Definition \ref{definition:eulersystem_general}, the proof of the paper \cite{ochiai-AIF} uses an approach 
called ``specialization method''.
\par 
When $n=1$, the proof of Euler system bounds in \cite{ochiai-AIF} is based on a method to recover the characteristic ideal of 
a torsion module over $\mathcal{O}[[X ]]$ from the size of various specializations of the modules (see \cite[Prop. 3.7]{ochiai-AIF}). 
When $n$ is greater than one, \cite{ochiai-AIF} requires a method to recover the characteristic ideal of 
a torsion module over $\mathcal{O}[[X_1 ,\ldots , X_n ]]$ from the characteristic ideal of various specializations of the modules 
to situations of $n-1$ variables (see \cite[Prop. 3.6]{ochiai-AIF}) and the proof proceeds by the induction with respect to 
the number of variables $n$. 
\item At the same time as \cite{ochiai-AIF}, Mazur-Rubin \cite{mr04} also considered Euler systems for Galois deformations 
according to their own motivation. They also considered Euler systems with coefficients in $\mathcal{O}[[X ]]$ which are   
not necessary the cyclotomic deformation of a usual $p$-adic representation. Since their result is limited to the case $n=1$, 
Theorem \ref{theorem:general_Euler_system_bound} (which is the case of $n=2$) does not follow from \cite{mr04}. The method of proof of Euler system bounds 
in \cite{mr04} is based on the same principle as \cite[Prop. 3.7]{ochiai-AIF}. 
\item 
In order that ``specialization method'' works well, the assumption \rm{(v)} was essential. 
A joint work with Shimomoto (\cite{so15} and \cite{so17}) relaxed the assumption by replacing the condition 
\rm{(v)} by a weaker assumption \rm{(v$'$)} as follows: 
\begin{enumerate}
\item[(v$'$)] 
Every local component $\mathcal{R} \widehat{\otimes}_{\mathbb{Z}_p} \Lambda^{\omega^i}_{\mathrm{cyc}}$ of the semi-local ring $\mathcal{R} \widehat{\otimes}_{\mathbb{Z}_p} \Lambda_{\mathrm{cyc}}$ is an integral domain which is integrally closed in its field of fractions. 
Here, $\omega$ denotes the Teichm\"{u}ller character and we recall that 
$\Lambda_{\mathrm{cyc}}$ is decomposed into a product of local rings 
as $\displaystyle{\Lambda_{\mathrm{cyc}}= \prod^{p-2}_{i=0} }\Lambda^{\omega^i}_{\mathrm{cyc}}$ by character decomposition. 
\end{enumerate}
and assuming an extra condition on vanishing of $H^2 (\mathbb{Q}_v , 
\mathbb{T}\widehat{\otimes}_{\mathbb{Z}_p} \Lambda^\sharp_{\mathrm{cyc}}
)$ for every $v \in \Sigma \setminus \{ \infty \}$. 
\end{enumerate}
\end{rem}
As in the proof of (1.9), we go on to the theory of Coleman map. 
Though we stated the above result for Euler system bounds in a general setting, the result for Coleman map 
is restricted to the case of Hida family. 
In fact, the construction of Coleman map is much harder in non-ordinary 
case and the construction in non-ordinary case (Coleman family) is 
a joint work with Filippo Nuccio \cite{NO16} which is explained in the next section. 
\par 
Let $\mathbb{F}$ be a Hida family over $\mathbb{I}$. By \cite[Prop.~1]{mw86}, we have a $G_{\mathbb{Q}}$-stable lattice $\mathbb{T} \subset \mathbb{V}_{\mathbb{F}}$ 
and the following $G_{\mathbb{Q}_p}$-stable exact sequence 
\begin{equation}\label{equation:localfiltrationatp3}
0 \longrightarrow F^+ \mathbb{T} 
\longrightarrow \mathbb{T} 
\longrightarrow \mathbb{T} /F^+ \mathbb{T}\longrightarrow 0 
\end{equation}
such that $F^+ \mathbb{T}$ is free of rank one over $\mathbb{I}$ and 
the action of $G_{\mathbb{Q}_p}$ on $F^+ \mathbb{T}$ is unramified. 

$\mathbb{T}$ a lattice and $\mathbb{A}$ the discrete Galois representation $\mathbb{T}\otimes_{\mathbb{I}} \mathbb{I}^\vee$. 
Let us define 
\begin{equation}
\mathbb{D}^{\mathrm{ord}}= (F^+ \mathbb{T} \widehat{\otimes}_{\mathbb{Z}_p} \widehat{\mathbb{Z}^\mathrm{ur}_p} )^{G_{\mathbb{Q}_p}}  
\end{equation}
Since $F^+ \mathbb{T}$ is unramified, the $\mathbb{I}$-module $\mathbb{D}^{\mathrm{ord}}$ is free of rank one by definition. Also, the specialization $\mathbb{D}^{\mathrm{ord}} \otimes_{\mathbb{I}} \kappa (\mathbb{I})$ is naturally identified with a $\kappa (\mathbb{I})$-lattice of $D_{\mathrm{dR}} (V_{f_\kappa})/\mathrm{Fil}^1 D_{\mathrm{dR}} (V_{f_\kappa})$ for any arithmetic specialization $\kappa $ on $\mathbb{I}$ 
of non-negative weight. 
\par 
We denote by 
$(\mathbb{D}^{\mathrm{ord}} )^\ast (1)$ the $\mathbb{I}$-module 
$\mathrm{Hom}_{\mathbb{I}} (\mathbb{D}^{\mathrm{ord}},\mathbb{I}) \otimes_{\mathbb{Z}_p} D_{\mathrm{crys}}(\mathbb{Z}_p (1))$. 
Here, $D_{\mathrm{crys}}(\mathbb{Z}_p (1))$ is a $\mathbb{Z}_p$-lattice of 
$D_{\mathrm{crys}}(\mathbb{Q}_p (1)) = (\mathbb{Q}_p (1) \otimes_{\mathbb{Q}_p} B_{\mathrm{crys}} )^{G_{\mathbb{Q}_p}}$ generated by $
\zeta_{p^\infty} \otimes t^{-1} $ with $\zeta_{p^\infty}$ a generator of $\mathbb{Z}_p$-module 
$\mathbb{Z}_p (1)$ and $t \in B_{\mathrm{crys}}$ is a standard element of Fontaine on which 
$G_{\mathbb{Q}_p}$ acts by the $p$-adic cyclotomic character. 
The specialization $(\mathbb{D}^{\mathrm{ord}})^\ast (1) \otimes_{\mathbb{I}} \kappa (\mathbb{I})$ is naturally identified with a $\kappa (\mathbb{I})$-lattice of $\mathrm{Fil}^0 D_{\mathrm{dR}} (V^\ast_{f_\kappa} (1))$ for any arithmetic specialization $\kappa $ on $\mathbb{I}$ 
of non-negative weight. 
\begin{thm}[[Thm 3.13 of \cite{ochiai-AJM03}] \label{theorem:colemanmap2}
Let us assume the setting of Theorem $C'$. 
\begin{enumerate}
\item 
There is a $
\mathbb{I}
\widehat{\otimes}_{\mathbb{Z}_p}
\Lambda_{\mathrm{cyc}}  
$-linear homomorphism: 
$$ 
\mathrm{Col}:\, 
H^1(\mathbb{Q}_p ,
(\mathbb{T}\widehat{\otimes}_{\mathbb{Z}_p} \Lambda^\sharp_{\mathrm{cyc}}
)^\ast (1)) 
\longrightarrow (\mathbb{D}^{\mathrm{ord}})^\ast (1) 
\widehat{\otimes}_{\mathbb{Z}_p} \Lambda_{\mathrm{cyc}}
$$ 
such that we have the following commutative diagram 
for any arithmetic specialization $\kappa $ on $\mathbb{I}$ 
of non-negative weight, any positive integer $j$ and for any Dirichlet 
character $\phi$ whose conductor $\mathrm{Cond} (\phi)$ is a power of $p$: 
$$
\begin{CD}
H^1(\mathbb{Q}_p ,
(\mathbb{T}\widehat{\otimes}_{\mathbb{Z}_p} \Lambda^\sharp_{\mathrm{cyc}}
)^\ast (1)) 
@>{\mathrm{Col}}>>  (\mathbb{D}^{\mathrm{ord}})^\ast (1) 
\widehat{\otimes}_{\mathbb{Z}_p} \Lambda_{\mathrm{cyc}} \\ 
@V{(\kappa , \chi^{-j}_{\mathrm{cyc}}\phi^{-1})}VV @VV{(\kappa ,\chi^{-j}_{\mathrm{cyc}}\phi^{-1})}V \\ 
 {H^1 (\mathbb{Q}_p ,(V_f \otimes \phi )^\ast (1-j))}
@>>{e_p (f_\kappa, j,\phi )\times \mathrm{exp}^{\ast}}>  \mathrm{Fil}^0 D_{\mathrm{dR}}((V_{f_\kappa} \otimes \phi )^\ast  (1-j))
\end{CD}
$$
where $\mathrm{exp}^{\ast}$ is the dual exponential map of Bloch-Kato and 
the factor $e_p (f_\kappa, j,\phi )$ is given in Theorem $C$. 
\item 
Recall that there is a natural homomorphism 
$$
\mathrm{P}:H^1(\mathbb{Q}_p ,
(\mathbb{T}\widehat{\otimes}_{\mathbb{Z}_p} \Lambda^\sharp_{\mathrm{cyc}}
)^\ast (1))
\longrightarrow 
\dfrac{H^1(\mathbb{Q}_p ,
(\mathbb{T}\widehat{\otimes}_{\mathbb{Z}_p} \Lambda^\sharp_{\mathrm{cyc}}
)^\ast (1))}
{\mathrm{Im}(H^1(\mathbb{Q}_p ,
F^+(\mathbb{T}\widehat{\otimes}_{\mathbb{Z}_p} \Lambda^\sharp_{\mathrm{cyc}}
)^\ast (1)))} 
$$ 
Then, the map 
$$ 
\mathrm{Col}:\, 
\varprojlim_n  {H^1 (\mathbb{Q}_p (\mu_{p^n}) ,T^\ast (1))}
\longrightarrow (\mathbb{D}^{\mathrm{ord}})^\ast (1) 
\widehat{\otimes}_{\mathbb{Z}_p} \Lambda_{\mathrm{cyc}}
$$ 
is factorized as $\overline{\mathrm{Col}}\circ \mathrm{P}$ where 
$\overline{\mathrm{Col}}$ is given by 
$$ 
\overline{\mathrm{Col}}:\, \dfrac{H^1(\mathbb{Q}_p ,
(\mathbb{T}\widehat{\otimes}_{\mathbb{Z}_p} \Lambda^\sharp_{\mathrm{cyc}}
)^\ast (1))}
{\mathrm{Im}(H^1(\mathbb{Q}_p ,
F^+(\mathbb{T}\widehat{\otimes}_{\mathbb{Z}_p} \Lambda^\sharp_{\mathrm{cyc}}
)^\ast (1)))} \longrightarrow (\mathbb{D}^{\mathrm{ord}})^\ast (1) 
\widehat{\otimes}_{\mathbb{Z}_p} \Lambda_{\mathrm{cyc}}. 
$$ 
Further, $\overline{\mathrm{Col}}$ is an $
\mathbb{I}
\widehat{\otimes}_{\mathbb{Z}_p}
\Lambda_{\mathrm{cyc}}  
$-linear injective homomorphism whose cokernel is a pseudo-null 
$
\mathbb{I}
\widehat{\otimes}_{\mathbb{Z}_p}
\Lambda_{\mathrm{cyc}}  
$-module. 
\end{enumerate}
\end{thm}
\begin{rem}\label{remark:Colemanmap_Hidafamily}
The proof of Theorem \ref{theorem:colemanmap2} relies very much on 
the fact that the Galois representation $\mathbb{V}_{\mathbb{F}}$ 
associated to a Hida family $\mathbb{F}$ is obtained as an extension of 
a rank-one Galois representation by a rank-one Galois representation. 
Though $\mathbb{V}_{\mathbb{F}} $ is of rank two, we reduce the proof of 
Theorem \ref{theorem:colemanmap2} to a Coleman map of rank-one Galois representation which was classically known and was proved by Coleman power series of local units (see \cite{ochiai-AJM03} for the detail of the proof). 
\end{rem}
According to the construction of Hida family, it is not difficult to see that Kato's work \cite{320}
implies that an Euler system for $\mathbb{T}$ in the sense of Definition \ref{definition:eulersystem_general} 
exists and nontrivial as follows (see Remark \ref{remark:exixtence_of_Eulersytem_Hida} (1) below). 
\begin{thm}\label{theorem:eulersystem_Hida}
For natural numbers $c,d$ which does not divide the tame level of the Hida family nor $6p$ and for $\xi \in SL_2 (\mathbb{Z}_p)$, we have 
an Euler system: 
$$
\bigl\{ {}_{c,d}\mathcal{Z}_\xi (r) \in H^1(\mathbb{Q} (\mu_r ) ,
(\mathbb{T}\widehat{\otimes}_{\mathbb{Z}_p} \Lambda^\sharp_{\mathrm{cyc}}
)^\ast (1)) \bigr\} _{r\in \mathcal{S} }
$$ 
such that $\mathrm{Col} (\mathrm{loc}_p ({}_{c,d}\mathcal{Z}_\xi (r)))$ has 
the following relation to the special value:
\begin{equation}\label{equation:specialvalue_at_bottom2}
\psi \left({\mathrm{exp^\ast} (\mathrm{loc}_p ( (\kappa , \chi^{-j}_{\mathrm{cyc}}\phi^{-1})
({}_{c,d}\mathcal{Z}_\xi (r)))) }\right)
= \tau (\phi \psi ) 
\dfrac{L_{(pr)}(f_\kappa , \psi^{-1} \phi^{-1},j)}
{(2\pi \sqrt{-1})^{j}\Omega^{\mathrm{sgn}(j,\phi \psi )}_{f_\kappa ,(c,d,\xi ),\infty} } \cdot 
\overline{f}_\kappa \otimes \psi^{-1} \phi^{-1},
\end{equation}
for any arithmetic specialization $\kappa $ on $\mathbb{I}$ 
of non-negative weight, any integer $j$ satisfying $1\leq j\leq w(\kappa )+1$, any Dirichlet 
character $\phi$ whose conductor $\mathrm{Cond} (\phi)$ is a power of $p$ and any character $\psi$ of $\mathrm{Gal} (\mathbb{Q}(\mu_r)/\mathbb{Q})$.  
Here $L_{(pr)}(f_\kappa , \phi^{-1},s)$ denotes the Hecke $L$-function of $f$ twisted by $\psi ^{-1} \phi^{-1}$ whose Euler factors are removed at primes dividing $pr$. 
The element $\Omega^{\mathrm{sgn}(j,\phi \psi )}_{f_\kappa ,(c,d,\xi ),\infty}$ is 
a complex period which satisfies Algebraicity Theorem of Shimura presented in in \S \ref{sec:Intro_motif}. 
\end{thm}
\begin{rem}\label{remark:exixtence_of_Eulersytem_Hida}
\begin{enumerate}
\item 
Let $N$ be the tame conductor of Hida family. 
The Euler system of Theorem \ref{theorem:eulersystem_Hida} is obtained by 
taking the inverse limit of Beilinson-Kato Euler systems on 
modular curves $Y (Np^n)$ constructed by Kato \cite{320} 
with respect to norm maps $Y (Np^{n+1}) \longrightarrow Y (Np^n)$. Kato  \cite{320} does not discuss the situation of Hida family at all and it only provides an Euler system for each cuspform $f$. However, he proves the Beilinson-Kato Euler system on modular curves $Y (Np^n)$ are norm compatible in \cite{320}. As seen in \cite{Hi86b}, Hida families are constructed by the inverse limit of cohomologies of $Y (Np^n)$ and ``cutting out the ordinary part by Hida's 
$e$-operator''. 
Hence we deduce Theorem \ref{theorem:eulersystem_Hida} by taking inverse limit of Beilinson-Kato elements in \cite{320}. 
\item 
The parameters $c$, $d$ and $\xi$ appear necessarily 
through the construction in \cite{320}. Since these parameters are of no importance to understand the rest of the paper we do not explain how these elements ${}_{c,d}\mathcal{Z}_\xi (r)$ depend on these parameters $c,d$ and $\xi$. See \cite{320} for the detail. 
\end{enumerate}
\end{rem}
\begin{rem}\label{remark:period}
If we consider Iwasawa theory for a family of modular forms, 
there appears an ``essential problem on periods'' as explained below. 
The period 
$
\Omega^{\mathrm{sgn}(j,\phi )}_{f_\kappa ,\infty}$
in \eqref{equation:specialvalue_at_bottom2} is based on Rankin-Selberg method. In order that the ideal $\mathrm{Col}(\mathrm{loc}_p ({}_{c,d}\mathcal{Z}_\xi (1)))$ defines the same principal ideal as 
$L_p (\{ \Xi^\pm \})$ given in Theorem B of this section, 
it is necessary\footnote{It seems that the condition is necessary but not sufficient.} that, for any arithmetic specialization $\kappa$ of non-negative weight on $\mathbb{I}$, the ratio 
$
\dfrac{\Omega^{\mathrm{sgn}(j,\phi )}_{f_\kappa ,\infty} }{\Omega^{\mathrm{sgn}(j,\phi )}_{f_\kappa ,(c,d,\xi ),\infty} } \in \mathbb{Q}_f$ is a $p$-adic unit for a complex period $\Omega^{\mathrm{sgn}(j,\phi )}_{f_\kappa ,\infty} $ which is $p$-optimal in the sense of Remark 
\ref{remark:complexperiod1}. We are not sure if the Euler system $\bigl\{ {}_{c,d}\mathcal{Z}_\xi (r) \bigr\}_{r\in \mathcal{S}}$ 
\eqref{equation:specialvalue_at_bottom} satisfies this condition for certain choice of the parameters $c,d,\xi$. 
\end{rem}
In fact, Kato \cite{320} was already faced with 
the problem of the comparison of the complex periods as discussed in Remark \ref{remark:period}. In \cite{320}, we only consider a cuspform. Hence he manages to prove that a linear combination 
of Euler systems for some parameters $(c,d,\xi )$ (see \cite[\S 16]{320}) provides a $p$-optimized period. 
\par 
If we consider a family of cuspforms, we have to optimize the Euler system 
at infinitely many arithmetic specializations at the same time. Taking a certain finite linear combination like in \cite{320} will not work. Finally, we obtain the following 
optimized Euler system with help of two-variable Coleman map (Theorem \ref{theorem:colemanmap2}). 
\begin{thm}\label{theorem:eulersystem_Hida2}
Let us assume the setting of Theorem $C'$ and let $\mathbb{F}$ and $\mathbb{T}$ be as above. Assume that $p \geq 5$. We assume that the image of the residual representation 
$\overline{\rho} : G_{\mathbb{Q}} \longrightarrow GL_2 (\mathbb{F})$ is non-solvable 
and fix an $\mathbb{I}$-basis $\Xi^\pm$ of $\mathbb{MS}(\mathbb{I})^\pm$ respectively \footnote{We remark that the irreducibility of the residual representation implies the validity of the condition (M).}. 
\par 
Then, we have 
an Euler system $($depending on the choice of $\Xi^\pm )$:
$$
\bigl\{ \mathcal{Z} (r) \in H^1(\mathbb{Q} (\mu_r ) ,
(\mathbb{T}\widehat{\otimes}_{\mathbb{Z}_p} \Lambda^\sharp_{\mathrm{cyc}}
)^\ast (1)) \bigr\} _{r\in \mathcal{S} }
$$ 
such that $\mathrm{Col} (\mathrm{loc}_p (\mathcal{Z} (r)))$ has 
the following relation to the special value: 
\begin{equation}\label{equation:specialvalue_at_bottom22}
\dfrac{\psi \left( \mathrm{exp^\ast} (\mathrm{loc}_p ( (\kappa , \chi^{-j}_{\mathrm{cyc}}\phi^{-1})
(\mathcal{Z} (r)))) \right)}{C^{\mathrm{sgn}(j,\phi \psi )}_{f_\kappa ,p}}
= \tau (\phi \psi ) 
\dfrac{L_{(pr)}(f_\kappa , \psi^{-1} \phi^{-1},j)}
{(2\pi \sqrt{-1})^{j}\Omega^{\mathrm{sgn}(j,\phi \psi )}_{f_\kappa ,\infty} } \cdot 
\overline{f}_\kappa \otimes \psi^{-1}\phi^{-1},
\end{equation}
for any arithmetic specialization $\kappa $ on $\mathbb{I}$ 
of non-negative weight, any integer $j$ satisfying $1\leq j\leq w(\kappa )+1$, any Dirichlet 
character $\phi$ whose conductor $\mathrm{Cond} (\phi)$ is a power of $p$ and 
any character $\psi$ of $\mathrm{Gal} (\mathbb{Q}(\mu_r)/\mathbb{Q})$. Here $\Omega^{\mathrm{sgn}(j,\phi )}_{f_\kappa ,\infty}$ denotes 
a complex period which satisfy Algebraicity  Theorem of Shimura presented in in \S \ref{sec:Intro_motif}
and $C^{\mathrm{sgn}(j,\phi )}_{f_\kappa ,p}$ is a $p$-adic period defined in 
\ref{definition:padicperiodMS} \footnote{See Remark \ref{remark:padicperiodMS} for the fact that 
the interpolation property is well-defined.}. 
\par 
In particular, we have the equality: 
\begin{equation}\label{equation:modified_Euler}
\mathrm{Col}(\mathrm{loc}_p (\mathcal{Z} (1)))= 
L_p (\{ \Xi^\pm \})
\end{equation}
in $\mathbb{I} \widehat{\otimes}_{\mathbb{Z}_p} 
\Lambda_{\mathrm{cyc}}$.
\end{thm}
Though the proof of Theorem \ref{theorem:eulersystem_Hida2} is found in \cite[\S 5.3]{Och06}, we will 
prove it below since the same method will be used later in a similar but slightly different context of the construction of Euler systems of Coleman 
families. 
We prepare a lemma before going into the proof of Theorem \ref{theorem:eulersystem_Hida2}. 
\begin{lemma}\label{lem:glues}
Let us assume the setting of Theorem $C'$ and let $\mathbb{F}$, $\mathbb{T}$ and $\mathbb{A}$ be as above. Assume that $p \geq 5$. We assume that the image of 
the residual representation is non-solvable. Let $r$ be a natural number and put $J, J' \in \mathfrak{A}$.  
Then, we have the exact sequence$:$ 
\begin{multline}\label{equation:glying}
0 \longrightarrow H^1 (((\mathbb{T}\widehat{\otimes}_{\mathbb{Z}_p} \Lambda^\sharp_{\mathrm{cyc}}
)^\ast(1))_{J\cap J'})
\longrightarrow H^1 (((\mathbb{T}\widehat{\otimes}_{\mathbb{Z}_p} \Lambda^\sharp_{\mathrm{cyc}}
)^\ast(1))_J)
\oplus H^1 (((\mathbb{T}\widehat{\otimes}_{\mathbb{Z}_p} \Lambda^\sharp_{\mathrm{cyc}}
)^\ast(1))_{J'}) 
\\ 
\longrightarrow H^1 (((\mathbb{T}\widehat{\otimes}_{\mathbb{Z}_p} \Lambda^\sharp_{\mathrm{cyc}}
)^\ast(1))_{J+J'}) . 
\end{multline}
Here we denote $(\mathbb{T}\widehat{\otimes}_{\mathbb{Z}_p} \Lambda^\sharp_{\mathrm{cyc}})^\ast (1)/J(\mathbb{T}\widehat{\otimes}_{\mathbb{Z}_p} \Lambda^\sharp_{\mathrm{cyc}} )^\ast(1)$ by $((\mathbb{T}\widehat{\otimes}_{\mathbb{Z}_p} \Lambda^\sharp_{\mathrm{cyc}})^\ast (1))_J$ 
and the cohomology group $H^1 (\mathbb{Q} (\mu_r)_{\Sigma_r}/\mathbb{Q} (\mu_r),((\mathbb{T}\widehat{\otimes}_{\mathbb{Z}_p} \Lambda^\sharp_{\mathrm{cyc}})^\ast (1))_J)$ by $H^1 (\mathbb{T}\widehat{\otimes}_{\mathbb{Z}_p} \Lambda^\sharp_{\mathrm{cyc}}(1)_J)$ 
for any ideal $J$ of $\mathbb{I}\widehat{\otimes}_{\mathbb{Z}_p} \Lambda_{\mathrm{cyc}}$. \par 
The first map of \eqref{equation:glying} sends each element $x_{J\cap J'}$ to $(x_{J\cap J'} \text{ mod $J$})\oplus (x_{J\cap J'} \text{ mod $J'$})$. 
The second map \eqref{equation:glying} sends each element $x_{J} \oplus y_{J'}$ to 
$(x_{J} \text{ mod $J+J'$} )- (y_{J'} \text{ mod $J+J'$} )$.  
\end{lemma}
\begin{proof}
Let us consider the short exact sequence of $G_{\mathbb{Q}}$-module: 
\begin{multline}\label{equation:glying2}
0 \longrightarrow   ((\mathbb{T}\widehat{\otimes}_{\mathbb{Z}_p} \Lambda^\sharp_{\mathrm{cyc}}
)^\ast(1))_{J\cap J'}
\longrightarrow 
((\mathbb{T}\widehat{\otimes}_{\mathbb{Z}_p} \Lambda^\sharp_{\mathrm{cyc}}
)^\ast(1))_J
\oplus ((\mathbb{T}\widehat{\otimes}_{\mathbb{Z}_p} \Lambda^\sharp_{\mathrm{cyc}}
)^\ast(1))_{J'} 
\\ 
\longrightarrow ((\mathbb{T}\widehat{\otimes}_{\mathbb{Z}_p} \Lambda^\sharp_{\mathrm{cyc}}
)^\ast(1))_{J+J'} \longrightarrow 0 .
\end{multline} 
By the assumption that  
the image of the residual representation 
$\overline{\rho} : G_{\mathbb{Q}} \longrightarrow GL_2 (\mathbb{F})$ is non-solvable, 
the residual representation $\overline{\rho}$ restricted to $G_{\mathbb{Q} (\mu_r )}$ 
is irreducible for any natural number $r$. 
Thus, $H^0 (G_{\mathbb{Q} (\mu_r )} , ((\mathbb{T}\widehat{\otimes}_{\mathbb{Z}_p} \Lambda^\sharp_{\mathrm{cyc}}
)^\ast(1))_{J+J'}) =0$ for any natural number $r$. Thus by taking 
the long exact sequence of the Galois cohomology groups  
associated to the short exact sequence \eqref{equation:glying2} of $\mathrm{Gal (\mathbb{Q} (\mu_r)_{\Sigma_r}/\mathbb{Q} (\mu_r))}$-modules, we obtain \eqref{equation:glying}. 
This completes the proof. 
\end{proof}
Now let us go into the proof of Theorem \ref{theorem:eulersystem_Hida2}. 
\begin{proof}[Proof of Theorem $\ref{theorem:eulersystem_Hida2}$]
Let us define the following set 
\begin{equation}\label{equation:setS}
\mathfrak{S}=\{ I=\mathrm{Ker} (\kappa  )  
\, \vert \, \text{ $\kappa$ is an arithmetic specialization on $\mathbb{I}$ of 
non-negative weight} 
\}. 
\end{equation}
We denote by $\mathfrak{A}$ a subset of the set of 
height one ideals of $\mathbb{I} \widehat{\otimes}_{\mathbb{Z}_p}\Lambda_{\mathrm{cyc}}$ as follows:
\begin{equation}
\mathfrak{A} = \left\{ \left. J = {\bigcap}_{I \in S} I 
\cdot \mathbb{I} \widehat{\otimes}_{\mathbb{Z}_p}\Lambda_{\mathrm{cyc}}
\ 
\right\vert \ 
S \subset \mathfrak{S}, \sharp S <\infty
 \right\} .
\end{equation}
Note that $J\cap J' \in \mathfrak{A}$ for any $J,J' \in \mathfrak{A}$ and that 
the intersection $\bigcap J$ for infinitely many 
$J \in \mathfrak{A}$ is zero. 
For any natural number $r$, we define $\Sigma_r$ to be 
$\Sigma_r =\Sigma \cup \{ \text{primes $q$ dividing $r$}\}$.  
We denote by $\mathbb{Q}_{\Sigma_r}$ the maximal unramified extension of $\mathbb{Q}$ unramified outside $\Sigma_r$. 
\par 
Let us consider the following morphism$:$ 
\begin{multline}\label{equation:loc_ord_composition}
H^1 (\mathbb{Q}_{\Sigma}/\mathbb{Q} ,((\mathbb{T}\widehat{\otimes}_{\mathbb{Z}_p} \Lambda^\sharp_{\mathrm{cyc}}
)^\ast(1))_J(1)) 
\xrightarrow{\mathrm{loc}} 
{H^1 (\mathbb{Q}_p ,((\mathbb{T}\widehat{\otimes}_{\mathbb{Z}_p} \Lambda^\sharp_{\mathrm{cyc}}
)^\ast(1))_J(1))} 
\\ 
\xrightarrow{\langle d_J \otimes 1 ,\ \rangle \circ \mathrm{Col}} 
(\mathbb{I} /J )  
\widehat{\otimes}_{\mathbb{Z}_p} \Lambda_{\mathrm{cyc}}
\end{multline}
where $d_J$ is a basis of free $\mathbb{I}/J$-module $\mathbb{D}^{\mathrm{ord}}/J \mathbb{D}^{\mathrm{ord}}$ of rank one which coincide 
with $\overline{f}_\kappa$ for any arithmetic specialization $\kappa$ on 
$\mathbb{I}$ and $\langle \ ,\ \rangle $ is the paring:
$$
(\mathbb{D}^{\mathrm{ord}}/J \mathbb{D}^{\mathrm{ord}}) 
{\otimes}_{\mathbb{Z}_p} \Lambda_{\mathrm{cyc}}
\times \big((\mathbb{D}^{\mathrm{ord}})^\ast/J 
(\mathbb{D}^{\mathrm{ord}})^\ast \big)
{\otimes}_{\mathbb{Z}_p} \Lambda_{\mathrm{cyc}}
\longrightarrow (\mathbb{I}/J ){\otimes}_{\mathbb{Z}_p} \Lambda_{\mathrm{cyc}}
$$  
Now we take $I=\mathrm{Ker}(\kappa )$ and $I'=\mathrm{Ker}(\kappa ')$ 
from $\mathfrak{S}$ given in \eqref{equation:setS}. Set 
$J= I \cdot \mathbb{I} \widehat{\otimes}_{\mathbb{Z}_p} \Lambda_{\mathrm{cyc}}$ and $J'= I' \cdot \mathbb{I}\widehat{\otimes}_{\mathbb{Z}_p} \Lambda_{\mathrm{cyc}}$. 

As explained before Theorem \ref{theorem:eulersystem_Hida2}, 
the result of Kato for the cyclotomic deformation of a cuspform 
assures the existence of the Euler system: 
\begin{align*}
& \bigl\{ \mathcal{Z}_J (r) \in H^1(\mathbb{Q} (\mu_r ) ,
((\mathbb{T}\widehat{\otimes}_{\mathbb{Z}_p} \Lambda^\sharp_{\mathrm{cyc}}
)^\ast(1))_J ) \bigr\} _{r\in \mathcal{S} } 
\\ 
& \Bigl( \mathrm{resp}.\,  \bigl\{ \mathcal{Z}_{J'} (r) \in H^1(\mathbb{Q} (\mu_r ) ,
((\mathbb{T}\widehat{\otimes}_{\mathbb{Z}_p} \Lambda^\sharp_{\mathrm{cyc}}
)^\ast(1))_{J'}) \bigr\} _{r\in \mathcal{S} }
\Bigr)
\end{align*} 
such that we have 
\begin{align*}
& 
\mathrm{Col}(\mathrm{loc}_p (\mathcal{Z}_J  (1)))= 
L_p (\{ \Xi^\pm \}) \text{ mod $J$ in $(\mathbb{I}/J) {\otimes}_{\mathbb{Z}_p} 
\Lambda_{\mathrm{cyc}}$}
\\ 
& \Bigl( \mathrm{resp}.\, 
\mathrm{Col}(\mathrm{loc}_p (\mathcal{Z}_{J'} (1)))= 
L_p (\{ \Xi^\pm \}) \text{ mod $J'$ in $(\mathbb{I}/J') {\otimes}_{\mathbb{Z}_p} \Lambda_{\mathrm{cyc}}$}\Bigr)
\end{align*} 

Note that, since $L_p (\{ \Xi^\pm \})$ is an element of $\mathbb{I}\widehat{\otimes}_{\mathbb{Z}_p} 
\Lambda_{\mathrm{cyc}}$, $L_p (\{ \Xi^\pm \}) \text{ mod $J$}$ and $L_p (\{ \Xi^\pm \}) \text{ mod $J'$}$  
can be glued together to have an element of $(\mathbb{I}/(J \cap J')) {\otimes}_{\mathbb{Z}_p} 
\Lambda_{\mathrm{cyc}}$. 
By the injectivity of the map $\langle d_J \otimes 1 ,\ \rangle \circ \mathrm{Col}$
for $(\mathbb{T}\widehat{\otimes}_{\mathbb{Z}_p} \Lambda^\sharp_{\mathrm{cyc}}
)^\ast(1))_J$ (\rm{resp. }$(\mathbb{T}\widehat{\otimes}_{\mathbb{Z}_p} \Lambda^\sharp_{\mathrm{cyc}}
)^\ast(1))_{J'}$) and by Lemma \ref{lem:glues}, 
we have the first layer $\mathcal{Z}_{J \cap J'} (1) \in H^1(\mathbb{Q} ,
((\mathbb{T}\widehat{\otimes}_{\mathbb{Z}_p} \Lambda^\sharp_{\mathrm{cyc}}
)^\ast(1))_{J\cap J'} )$ which glues $ \mathcal{Z}_{J'} (1) $ and 
$ \mathcal{Z}_{J'} (1) $. Similarly for each $r \in  \mathcal{S}$, we can use the Coleman map 
$\mathrm{Col}^{(r)}$ over $\mathbb{Q}_p \otimes_{\mathbb{Q}} \mathbb{Q} (\mu_r)$ and the fact that 
we have 
\begin{align*}
& 
\mathrm{Col}^{(r)}(\mathrm{loc}_p (\mathcal{Z}_J  (r)))= 
L^{(r)}_p (\{ \Xi^\pm \}) \text{ mod $J$ in $(\mathbb{I}/J) {\otimes}_{\mathbb{Z}_p} 
\Lambda_{\mathrm{cyc}}$}
\\ 
& \Bigl( \mathrm{resp}.\, 
\mathrm{Col}^{(r)}(\mathrm{loc}_p (\mathcal{Z}_{J'} (r)))= 
L^{(r)}_p (\{ \Xi^\pm \}) \text{ mod $J'$ in $(\mathbb{I}/J') {\otimes}_{\mathbb{Z}_p} \Lambda_{\mathrm{cyc}}$}\Bigr)
\end{align*} 
where $L^{(r)}_p (\{ \Xi^\pm \})\in \mathbb{I}\widehat{\otimes}_{\mathbb{Z}_p} 
\Lambda_{\mathrm{cyc}}$ is the $p$-adic $L$-function obtained by removing Euler factors 
at primes dividing $r$ from $L_p (\{ \Xi^\pm \})$. 
Thus we obtain a unique Euler system 
$$
\bigl\{ \mathcal{Z}_{J \cap J'} (r) \in H^1(\mathbb{Q} (\mu_r ) ,
((\mathbb{T}\widehat{\otimes}_{\mathbb{Z}_p} \Lambda^\sharp_{\mathrm{cyc}}
)^\ast(1))_{J\cap J'} ) \bigr\} _{r\in \mathcal{S} } 
$$
which glues Euler systems $\bigl\{ \mathcal{Z}_{J'} (r) \bigr\} _{r\in \mathcal{S} }$ and 
$\bigl\{ \mathcal{Z}_{J'} (r)\bigr\}_{r\in \mathcal{S} } $. 
\par 
By repeating the induction on 
the numbers of the ideals of the form $I \cdot \mathbb{I}\widehat{\otimes}_{\mathbb{Z}_p} \Lambda_{\mathrm{cyc}}$ with $I \in \mathfrak{S}$ 
containing a given ideal $J \in \mathfrak{A}$, we construct 
for each $J \in \mathfrak{A}$ an Euler system 
$
\bigl\{ \mathcal{Z}_J (r) \in H^1(\mathbb{Q} (\mu_r ) ,
((\mathbb{T}\widehat{\otimes}_{\mathbb{Z}_p} \Lambda^\sharp_{\mathrm{cyc}}
)^\ast(1))_J ) \bigr\} _{r\in \mathcal{S} } 
$   
such that it provides an Euler system 
for each cuspform corresponding to the representation mod $I \cdot \mathbb{I}\widehat{\otimes}_{\mathbb{Z}_p} \Lambda_{\mathrm{cyc}}$ with $I \in \mathfrak{S}$ such 
that $J \subset I \cdot \mathbb{I}\widehat{\otimes}_{\mathbb{Z}_p} \Lambda_{\mathrm{cyc}}$. The Euler systems 
$\bigl\{ \mathcal{Z}_J (r)  \bigr\} _{r\in \mathcal{S} }$ form an inverse system for $J \subset \mathfrak{A}$. Since we have 
$$
\varprojlim_{J \in \mathfrak{A}} H^1(\mathbb{Q} (\mu_r ) ,
((\mathbb{T}\widehat{\otimes}_{\mathbb{Z}_p} \Lambda^\sharp_{\mathrm{cyc}}
)^\ast(1))_J )
\cong 
H^1(\mathbb{Q} (\mu_r ) ,
((\mathbb{T}\widehat{\otimes}_{\mathbb{Z}_p} \Lambda^\sharp_{\mathrm{cyc}}
)^\ast(1)) ),
$$
we obtain the desired Euler system 
$
\bigl\{ \mathcal{Z}(r) \in H^1(\mathbb{Q} (\mu_r ) ,
((\mathbb{T}\widehat{\otimes}_{\mathbb{Z}_p} \Lambda^\sharp_{\mathrm{cyc}}
)^\ast(1)) ) \bigr\} _{r\in \mathcal{S} } 
$   
by putting $\mathcal{Z}(r)= \varprojlim_{J \in \mathfrak{A}} \mathcal{Z}_J(r)$. 
Here, we remark that the topology defined by $\{ J \in \mathfrak{A} \}$
is equivalent to the topology defined by the maximal ideals of 
$\mathbb{I} \widehat{\otimes}_{\mathbb{Z}_p}\Lambda_{\mathrm{cyc}}$ by 
the fact that $\mathbb{I} \widehat{\otimes}_{\mathbb{Z}_p}\Lambda_{\mathrm{cyc}}$ 
is complete semi-local and by well-known Chevalley's theorem (cf. \cite[\S II, Lemma 7]{ch43}). 
\end{proof}

After we constructed a $p$-optimized Euler system, the argument goes similarly as the one-variable cyclotomic Iwasawa theory explained in 
\S \ref{sec:Intro_motif}. We go back to the proof of \eqref{equation:IMC_for_elliptic_cuspform_modulo_mu_bis} in Theorem $C'$. 
\begin{proof}[Proof of \eqref{equation:IMC_for_elliptic_cuspform_modulo_mu_bis} in Theorem $C'$]
By the Poitou-Tate exact sequence of Galois cohomology, Euler-Poincar\'{e} characteristic formula of Galois cohomology theory and the existence of nontrivial Euler system explained above, we have the following sequence of 
$\mathbb{I} \widehat{\otimes}_{\mathbb{Z}_p}\Lambda_{\mathrm{cyc}}$-modules: 
\begin{multline}\label{equation:4termsequence3}
0 \longrightarrow H^1 (\mathbb{Q}_{\Sigma}/\mathbb{Q}, (\mathbb{T}\widehat{\otimes}_{\mathbb{Z}_p} \Lambda^\sharp_{\mathrm{cyc}}
)^\ast(1) )   \longrightarrow 
\dfrac{H^1 (\mathbb{Q}_p  ,(\mathbb{T}\widehat{\otimes}_{\mathbb{Z}_p} \Lambda^\sharp_{\mathrm{cyc}}
)^\ast(1))}
{\mathrm{Im}(H^1 (\mathbb{Q}_p  ,F^+_p (\mathbb{T}\widehat{\otimes}_{\mathbb{Z}_p} \Lambda^\sharp_{\mathrm{cyc}}
)^\ast(1) ))}
\\ 
\longrightarrow \mathrm{Sel}_{\mathbb{A}} (\mathbb{Q}(\mu_{p^\infty}))^\vee 
\longrightarrow 
H^2 (\mathbb{Q}_{\Sigma}/\mathbb{Q},
(\mathbb{T}\widehat{\otimes}_{\mathbb{Z}_p} \Lambda^\sharp_{\mathrm{cyc}}
)^\ast(1)) \longrightarrow 0 
\end{multline}
such that 
\begin{enumerate}
\item[{(i)}]
The sequence \eqref{equation:4termsequence3} is exact modulo error by 
pseudo-null $\mathbb{I}\widehat{\otimes}_{\mathbb{Z}_p} \Lambda_{\mathrm{cyc}}$-modules. 
\item[(ii)]  
The last two terms are torsion 
over every local component of 
$\mathbb{I}\widehat{\otimes}_{\mathbb{Z}_p} \Lambda_{\mathrm{cyc}}$.  
\item[(iii)]
The first two terms are of rank one over every local component of 
$\mathbb{I}\widehat{\otimes}_{\mathbb{Z}_p} \Lambda_{\mathrm{cyc}}$. 
\end{enumerate} 
The reason of these statements (i), (ii) and (iii) is similar to the reason of similar statements (i), (ii) and (iii) 
explained in the proof of Theorem $C$.  
\par 
Now, we take the first layer of the modified Beilinson--Kato Euler system$:$  
$$
\mathcal{Z}(1) \in H^1 (\mathbb{Q}_{\Sigma}/\mathbb{Q},
(\mathbb{T}\widehat{\otimes}_{\mathbb{Z}_p} \Lambda^\sharp_{\mathrm{cyc}}
)^\ast(1)). 
$$ 
obtained in Theorem \ref{theorem:eulersystem_Hida2} and we denote by $\overline{\mathrm{loc}}_p$ the 
$\mathbb{I}\widehat{\otimes}_{\mathbb{Z}_p} \Lambda_{\mathrm{cyc}}$-linear homomorphism$:$ 
$$
H^1 (\mathbb{Q}_{\Sigma}/\mathbb{Q},
(\mathbb{T}\widehat{\otimes}_{\mathbb{Z}_p} \Lambda^\sharp_{\mathrm{cyc}}
)^\ast(1))  
\longrightarrow 
\dfrac{H^1 (\mathbb{Q}_p  ,(\mathbb{T}\widehat{\otimes}_{\mathbb{Z}_p} \Lambda^\sharp_{\mathrm{cyc}})^\ast(1))}
{\mathrm{Im}(H^1 (\mathbb{Q}_p  ,F^+_p (\mathbb{T}\widehat{\otimes}_{\mathbb{Z}_p} \Lambda^\sharp_{\mathrm{cyc}}
)^\ast(1))}. 
$$ 
Then the sequence \eqref{equation:4termsequence3} 
induces the following sequence of torsion $\mathbb{I}\widehat{\otimes}_{\mathbb{Z}_p} \Lambda_{\mathrm{cyc}}$-modules: 
\begin{multline}\label{equation:4termsequence4}
0 \longrightarrow 
H^1 (\mathbb{Q}_{\Sigma}/\mathbb{Q}, (\mathbb{T}\widehat{\otimes}_{\mathbb{Z}_p} \Lambda^\sharp_{\mathrm{cyc}}
)^\ast(1) )  
 \Bigl/ \mathbb{I}\widehat{\otimes}_{\mathbb{Z}_p} \Lambda_{\mathrm{cyc}} \cdot \mathcal{Z} (1)
\\  
\longrightarrow 
\dfrac{H^1 (\mathbb{Q}_p  ,(\mathbb{T}\widehat{\otimes}_{\mathbb{Z}_p} \Lambda^\sharp_{\mathrm{cyc}}
)^\ast(1))}
{\mathrm{Im}(H^1 (\mathbb{Q}_p  ,F^+_p (\mathbb{T}\widehat{\otimes}_{\mathbb{Z}_p} \Lambda^\sharp_{\mathrm{cyc}}
)^\ast(1) ))}
 \Bigl/ \mathbb{I}\widehat{\otimes}_{\mathbb{Z}_p} \Lambda_{\mathrm{cyc}} \cdot \overline{\mathrm{loc}}_p( 
\mathcal{Z}(1) ) 
\\ 
\longrightarrow \mathrm{Sel}_{\mathbb{A}} (\mathbb{Q}(\mu_{p^\infty}))^\vee 
\longrightarrow 
H^2 (\mathbb{Q}_{\Sigma}/\mathbb{Q},
(\mathbb{T}\widehat{\otimes}_{\mathbb{Z}_p} \Lambda^\sharp_{\mathrm{cyc}}
)^\ast(1)) \longrightarrow 0 
\end{multline}
which is exact modulo error by 
pseudo-null $\mathbb{I}\widehat{\otimes}_{\mathbb{Z}_p} \Lambda^\sharp_{\mathrm{cyc}}$-modules. 
Finally, by applying Theorem \ref{theorem:general_Euler_system_bound} 
for $\mathcal{R}=\mathbb{I}$, we obtain  
\begin{multline}\label{equation:euler_system_bound_general_for_R=I}
\mathrm{char}_{\mathbb{I} \widehat{\otimes}_{\mathbb{Z}_p} \Lambda_{\mathrm{cyc}}}
\left( 
H^1(\mathbb{Q}_{\Sigma}/\mathbb{Q} ,
(\mathbb{T}\widehat{\otimes}_{\mathbb{Z}_p} \Lambda^\sharp_{\mathrm{cyc}}
)^\ast (1)) 
 \Bigl/ \mathbb{I} \widehat{\otimes}_{\mathbb{Z}_p} \Lambda_{\mathrm{cyc}}  \cdot \mathcal{Z}(1)  
 \right) 
 \\ 
 \subset 
 \mathrm{char}_{\mathbb{I} \widehat{\otimes}_{\mathbb{Z}_p} \Lambda_{\mathrm{cyc}}} 
\left( H^2(\mathbb{Q}_{\Sigma}/\mathbb{Q} ,
(\mathbb{T}\widehat{\otimes}_{\mathbb{Z}_p} \Lambda^\sharp_{\mathrm{cyc}}
)^\ast (1))  \right) . 
\end{multline}
We recall that we have the following inclusion thanks to \eqref{equation:modified_Euler}$:$
\begin{equation}\label{equation:conclusion_of_Coleman_map_modified_ES}
\mathrm{char}_{\mathbb{I} \widehat{\otimes}_{\mathbb{Z}_p} \Lambda_{\mathrm{cyc}}}
\left( 
\dfrac{H^1 (\mathbb{Q}_p  ,(\mathbb{T}\widehat{\otimes}_{\mathbb{Z}_p} \Lambda^\sharp_{\mathrm{cyc}}
)^\ast(1))}
{\mathrm{Im}(H^1 (\mathbb{Q}_p  ,F^+_p (\mathbb{T}\widehat{\otimes}_{\mathbb{Z}_p} \Lambda^\sharp_{\mathrm{cyc}}
)^\ast(1) ))}
 \Bigl/ \mathbb{I}\widehat{\otimes}_{\mathbb{Z}_p} \Lambda^\sharp_{\mathrm{cyc}} \cdot \overline{\mathrm{loc}}_p( 
\mathcal{Z}(1) ) 
\right) = (L_p (\{ \Xi^\pm \})). 
\end{equation}
By combining \eqref{equation:4termsequence4}, \eqref{equation:euler_system_bound_general_for_R=I} and 
\eqref{equation:conclusion_of_Coleman_map_modified_ES}, 
we obtain \eqref{equation:IMC_for_elliptic_cuspform_modulo_mu_bis}. 
\end{proof}

\section{Coleman map for a Coleman family}
\label{sec:Intro_results}

In the non-ordinary situation, the evaluations of the $p$-adic $L$-functions might have 
denominators which are divisible by higher powers of $p$ depending on the power of $p$ which divides the conductor of evaluating characters. 
Hence, $p$-adic $L$-functions in non-ordinary case might need to be formulated as elements in a space of distributions which allows unbounded denominators. 
Though distributions with unbounded denominators might not be well-behaved as in the case of bounded measures, 
the space of distributions $\HH_{h,\mathrm{cyc}}$ introduced by Amice--V\'{e}lu is quite well-behaved and has nice properties to 
characterize and recover a function in $\HH_{h,\mathrm{cyc}}$ from the values of its evaluations (see Proposition \ref{proposition:uniquenesspropertyHh}). 
Thus it is quite important to work with $\HH_{h,\mathrm{cyc}}$. 
\begin{defn}\label{def:ell} 
\begin{enumerate}
\item 
For $i\in\mathbb{Z}_{\geq 0}$, we set $\ell(0)=0$ 
and we define $\ell(i)$ to be 
$\ell(i)=\Big\lfloor\frac{\ln (i)}{\ln (p)}\Big\rfloor+1$ if $i\geq 1$. 
\item 
For $h\geq 0$, we define
\[
\HH_{h}=\Big\{\sum^{+ \infty}_{i = 0}a_i X^i \in \Q [[X]] 
\ \Big\vert \ 
\inf
\big\{\mathrm{ord}(a_i) +h\ell (i) \ \big\vert 
\ {i \in\mathbb{Z}_{\geq 0}}
\big\}  >-\infty\Big\}
\]
and call the elements of $\HH_h$ the power series of logarithmic order $h$. 
\end{enumerate}
\end{defn}
We have a Banach norm on $\HH_{h}$ and we have an integral structure $\HH^+_{h} \subset \HH_{h}$ such that $\HH_{h} \cong \HH^+_{h} 
\otimes_{\mathbb{Z}_p} \mathbb{Q}_p$ by means of this norm (see \cite[\S 4.1]{NO16}). 
We do not give a precise definition of the Banach-module structure 
on $\HH_{h}$. But we only remark that $\HH^+_{0}$ is equal to 
$\mathbb{Z}_p [ [X]]$. 
\par 
For every $j \in \mathbb{Z}_{\geq 0}$, we set
\[
\omega_n^{[j]}=\omega_n^{[j]}(X)=(1+X)^{p^n}-(1+p)^{jp^n}
\]
For a fixed $n$ and for different $j\neq j'$ 
the polynomials $\omega_n^{[j]}(X)$ and $\omega_n^{[j']}(X)$ have no common factor in $\Z[[X]]$. 
We will denote $\omega^{[0]}_n (X)$ by $\omega_n (X)$. 
Finally, we recall that $\Lambda_{\mathrm{cyc}}$ is a semi-local ring 
isomorphic to $\mathbb{Z}_p [\mathrm{Gal}(\mathbb{Q}(\mu_p)/\mathbb{Q})] 
\otimes \mathbb{Z}_p [\mathrm{Gal}(\mathbb{Q}(\mu_{p^\infty})/\mathbb{Q}(\mu_p))]$ which is non-canonically isomorphic to 
$\mathbb{Z}_p [(\mathbb{Z}/(p))^\times ] \otimes \mathbb{Z}_p [[X]]$ 
induced by the Iwasawa-Serre isomorphism $\mathbb{Z}_p [\mathrm{Gal}(\mathbb{Q}(\mu_{p^\infty})/\mathbb{Q}(\mu_p))] \cong \mathbb{Z}_p [[X]] 
= \mathcal{H}^+_0$ 
sending a chosen topological generator $\gamma$ of $\mathrm{Gal}(\mathbb{Q}(\mu_{p^\infty})/\mathbb{Q}(\mu_p)$ to $1+X$. We define $\HH_{h,\mathrm{cyc}}$ (resp. $\hh_{h,\mathrm{cyc}}$) to be 
\begin{align*}
& 
\HH_{h,\mathrm{cyc}}= \mathbb{Z}_p [\mathrm{Gal}(\mathbb{Q}(\mu_{p^\infty})/\mathbb{Q}(\mu_p))] \otimes_{\mathcal{H}^+_0} \HH_h \\  
& \left(\mathrm{resp}. \ \hh_{h,\mathrm{cyc}} = \mathbb{Z}_p [\mathrm{Gal}(\mathbb{Q}(\mu_{p^\infty})/\mathbb{Q}(\mu_p))] \otimes_{\mathcal{H}^+_0} \hh_h 
\right) .
\end{align*}

We have the following important results of Amice--V\'{e}lu, which we can find also in \S 1.2 and \S 1.3 of \cite{396}: 
\begin{prop}\label{proposition:uniquenesspropertyHh}
Let $F \in \HH_{h,\mathrm{cyc}} $ and let us choose $l,l' \in \mathbb{Z}$ such that $h =l' -l$. 
Suppose that $F$ is contained in $Q \HH_{h,\mathrm{cyc}} $ for every element of $\mathcal{Q}^{[l ,l']}$. Then we have $F=0$. 
\end{prop}

\begin{prop}\label{proposition:glueingpropertyHh}
Let us choose $l,l' \in \mathbb{Z}$ with $l\leq l'$ and put $h =l' -l$. 
Let $\{ G_{n,j} \in \K [X ] \}_{n \in \mathbb{Z}_{\geq 1}, l \leq j\leq l'} $ 
be a sequence of polynomials satisfying the following conditions: 
\begin{enumerate}
\item[\rm{(i)}] 
For each $j$ satisfying $l \leq j\leq l'$, $\Vert p^{nh} G_{n,j}\Vert $ is bounded when $n$ varies. 
\item[\rm{(ii)}] 
For each $j$ satisfying $l \leq j\leq l'$ and for each $n \in \mathbb{Z}_{\geq 1}$, 
$G_{n+1,j} - G_{n,j} \equiv 0$ modulo  $\omega_n (X)\K [X]$. 
\item[\rm{(iii)}] 
For each $j$ satisfying $l \leq j\leq l' $, 
$$ 
\left\Vert 
p^{n(h-j)} \sum^j_{k=l} (-1)^{j-k} \genfrac{(}{)}{0pt}{}{j}{k}
 G_{n ,l+k}((1+X)-(1+p)^k) \right\Vert 
 $$ 
is bounded when $n$ varies. 
\end{enumerate}
Then there exists a unique element $F \in \HH_{h}$ such that 
$F \equiv G_{n,j} $ modulo $\omega_n^{[j]}\HH_{h}$ 
for every $n \in \mathbb{Z}_{\geq 1}$ and for every 
$j$ satisfying $l \leq j\leq l'$. 
\end{prop}


Let us define a $\Lambda_{(k_0 ;r),\mathcal{O}_{{K}}} 
\otimes_{\mathbb{Z}_p} \mathbb{Q}_p$-module $\mathbb{D}$ by 
\begin{equation}
\mathbb{D}:=
(
\mathbb{T} \widehat{\otimes}_{\mathbb{Z}_p} B_{\mathrm{crys}} )^{G_{\mathbb{Q}_p}, \varphi =A_p (\mathbb{F})}. 
\end{equation}
\begin{lemma}\label{lemma:bigD_nonord}
The $\Lambda_{(k_0 ;r),\mathcal{O}_{{K}}}\otimes_{\mathbb{Z}_p} \mathbb{Q}_p$-module $\mathbb{D}$ is free of rank one over $\Lambda_{(k_0 ;r),\mathcal{O}_{{K}}}\otimes_{\mathbb{Z}_p} \mathbb{Q}_p$ and the 
specialization of $\mathbb{D}$ at each arithmetic point 
$k \in \mathbb{Z} \cap B(k_0 ;r)$ is canonically identified with $D_k := D_\mathrm{cris}^{\varphi = a_p (f_k )} (V_{f_k})$. 
\end{lemma}
The lemma is essentially due to Kisin (see \cite[Corollary 5.6]{Kis03}). However the paper \cite{Kis03} does not really prove exactly the same result as what 
we present in Lemma \ref{lemma:bigD_nonord}. For the proof of Lemma \ref{lemma:bigD_nonord}, we refer the reader to \cite[Lemma 3.4]{NO16}. 
\par 
After this preparation, we can now recall the main result of the paper \cite{NO16}, namely the existence of the Coleman maps for Coleman families. 
\begin{thm}[Nuccio-Ochiai]\label{theorem:interpolationdualexp}
Let $s\in \mathbb{Q}_{\geq 0}$ be the slope of the Coleman family $\mathbb{F}$ and $h$ an integer satisfying $h\geq s$.  
Assume that the residual representation $\overline{\rho} : G_{\mathbb{Q}} \longrightarrow \operatorname{GL}_2 (\overline{\mathbb{F}}_p)$ associated 
to the family is irreducible when restricted to $G_{\Q (\mu_p )}$. 
\par 
Then, we have a unique $\Lambda_{(k_0 ;r),\mathcal{O}_{{K}}}\widehat{\otimes}_{\mathbb{Z}_p} \Lambda_{\mathrm{cyc}}$-linear  
big exponental map 
\[
\mathrm{Col} : 
H^1 \big(\mathbb{Q}_{p} ,(\mathbb{T}\widehat{\otimes}_{\mathbb{Z}_p} \Lambda^\sharp_{\mathrm{cyc}}
)^\ast(1) \big) 
\longrightarrow 
\mathbb{D}^\ast (1) \widehat{\otimes}_{\mathbb{Z}_p} 
\HH_{h,\mathrm{cyc}}
\]
such that, for each arithmetic point $k \in \mathbb{Z} \cap B(k_0 ;r)$ 
and for each arithmetic character $\chi^j_{\mathrm{cyc}} \phi : G_{\mathrm{cyc}} {\longrightarrow } \overline{\mathbb{Q}}^\times_p$ 
with $j$ a positive integer, we have the following commutative diagram: 
\[
\begin{CD}
 H^1 \big(\mathbb{Q}_{p} ,
(\mathbb{T}\widehat{\otimes}_{\mathbb{Z}_p} \Lambda^\sharp_{\mathrm{cyc}}
)^\ast(1) \big)  @>\mathrm{Col}>> 
\mathbb{D}^\ast (1) \widehat{\otimes}_{\mathbb{Z}_p}
\HH_{h,\mathrm{cyc}}
 \\ 
@V{(k, \chi^{-j}_{\mathrm{cyc}} \phi^{-1})}VV @VV{(k, \chi^{-j}_{\mathrm{cyc}} \phi^{-1} )}V \\ 
H^1 (\mathbb{Q}_{p} ,(V_{f_k} \otimes \phi )^\ast \otimes \chi^{1-j}_{\mathrm{cyc}}  )
@>>{e_p (f_k , j,\phi )\times \mathrm{exp}^{\ast}}>  (D_k )^\ast (1) \otimes D_{\mathrm{dR}} (  \chi^{-j}_{\mathrm{cyc}} \phi^{-1} ) 
\end{CD}
\]
where $e_p (f_k , j,\phi )$ is as given in Theorem C. 
\end{thm}
\begin{rem}
We give some remarks on the assumption of Theorem \ref{theorem:interpolationdualexp} on the irreducibility of the residual 
representation $\overline{\rho} : G_{\mathbb{Q}} \longrightarrow \operatorname{GL}_2 (\overline{\mathbb{F}}_p)$ 
restricted to $G_{\Q (\mu_p )}$. 
\begin{enumerate}
\item 
Thanks to the description of the mod $p$ modular Galois representation $\mathbb{G}_{\mathbb{Q}} \longrightarrow 
GL_2 (\overline{\mathbb{F}_p})$ of Serre Conjecture, $\overline{\rho}$ restricted to $G_{\Q (\mu_p )}$ is 
irreducible if and only if $\overline{\rho}$ restricted to $G_{\Q }$ is irreducible. 
\item 
This assumption implies that $H^1 \big(\mathbb{Q}_{p} ,
(\mathbb{T}\widehat{\otimes}_{\mathbb{Z}_p} \Lambda^\sharp_{\mathrm{cyc}}
)^\ast(1) \big) $ is free of rank two over each local component of $\Lambda_{(k_0 ;r),\mathcal{O}_{{K}}}
\widehat{\otimes}_{\mathbb{Z}_p} \Lambda_{\mathrm{cyc}}$. By this assumption, we also see that 
$H^1 \big(\mathbb{Q}_{p} ,
(\mathbb{T}\widehat{\otimes}_{\mathbb{Z}_p} \Lambda^\sharp_{\mathrm{cyc}}
)^\ast(1) \big) $ is exactly controlled under the specialization maps to all arithmetic points. 
These properties help us on the proof of the above theorem in \cite{NO16}. 
\end{enumerate}
\end{rem}
To be precise, the main Theorem of 
\cite{NO16} gives a $\Lambda_{(k_0 ;r),\mathcal{O}_{{K}}}
\widehat{\otimes}_{\mathbb{Z}_p} \Lambda_{\mathrm{cyc}}$-linear map 
\begin{equation*}\label{equation:Bigexponential}
\mathrm{EXP}: \mathbb{D}  \widehat{\otimes}_{\mathbb{Z}_p} \Lambda_{\mathrm{cyc}}
\longrightarrow
H^1 \big(\mathbb{Q}_{p} ,\mathbb{T}\widehat{\otimes}_{\mathbb{Z}_p} \Lambda^\sharp_{\mathrm{cyc}}
\big) 
\widehat{\otimes}_{\Lambda_{\mathrm{cyc}}} 
\HH_{h,\mathrm{cyc}}
\end{equation*}
which interpolates exponential maps (but not dual exponential maps as stated in Theorem \ref{theorem:interpolationdualexp}) 
for each arithmetic point $k \in \mathbb{Z} \cap B(k_0 ;r)$ 
and for each arithmetic character $\chi^j_{\mathrm{cyc}} \phi : G_{\mathrm{cyc}} {\longrightarrow } \overline{\mathbb{Q}}^\times_p$ 
with $j$ a positive integer. However, once we obtain 
an interpolation of exponential maps, it is straight forward to obtain 
an interpolation of dual exponential maps by taking its Kummer dual. 
In fact, we can extend the scalars to the ring $\HH_{\infty,\mathrm{cyc}}
= \displaystyle{{\bigcup}^\infty_{h=0} \HH_{h,\mathrm{cyc}}}$ and we obtain a $\Lambda_{(k_0 ;r),\mathcal{O}_{{K}}}
\widehat{\otimes}_{\mathbb{Z}_p} \HH_{\infty,\mathrm{cyc}}$-linear map  
\begin{equation}\label{equation:Bigexponential2}
\mathrm{EXP}\otimes_{} \HH_{\infty,\mathrm{cyc}}: \mathbb{D}  
\widehat{\otimes}_{\mathbb{Z}_p} 
\HH_{\infty,\mathrm{cyc}}
\longrightarrow
H^1 \big(\mathbb{Q}_{p} ,\mathbb{T}\widehat{\otimes}_{\mathbb{Z}_p} \Lambda^\sharp_{\mathrm{cyc}}
\big) 
\widehat{\otimes}_{\Lambda_{\mathrm{cyc}}} 
\HH_{\infty,\mathrm{cyc}}
\end{equation}
By Tate local duality of Galois cohomology theory, we have a canonical isomorphism 
$$
\mathrm{Hom}_{\Lambda_{\mathrm{cyc}}}
\bigl( H^1 \big(\mathbb{Q}_{p} ,\mathbb{T}\widehat{\otimes}_{\mathbb{Z}_p} \Lambda^\sharp_{\mathrm{cyc}}
\big) , \Lambda_{\mathrm{cyc}} \bigr)  
\cong H^1 \big(\mathbb{Q}_{p} ,(\mathbb{T}\widehat{\otimes}_{\mathbb{Z}_p} \Lambda^\sharp_{\mathrm{cyc}}
)^\ast(1) \big) 
$$
Hence, by taking the $\HH_{\infty,\mathrm{cyc}}$-linear dual of the map 
\eqref{equation:Bigexponential2}, we obtain a $\Lambda_{(k_0 ;r),\mathcal{O}_{{K}}}
\widehat{\otimes}_{\mathbb{Z}_p} \HH_{\infty,\mathrm{cyc}}$-linear map 
\begin{equation}\label{equation:Bigexponential3}
H^1 \big(\mathbb{Q}_{p} ,(\mathbb{T}\widehat{\otimes}_{\mathbb{Z}_p} \Lambda^\sharp_{\mathrm{cyc}}
)^\ast(1) \big) \widehat{\otimes}_{\Lambda_{\mathrm{cyc}}} 
\HH_{\infty,\mathrm{cyc}}
\longrightarrow 
\mathbb{D}^\ast (1) \widehat{\otimes}_{\mathbb{Z}_p} 
\HH_{\infty,\mathrm{cyc}}
\end{equation}
which have exactly the same interpolation property as Theorem 
\ref{theorem:interpolationdualexp}. 
The map $\mathrm{Col}$ is defined to be a $\Lambda_{(k_0 ;r),\mathcal{O}_{{K}}}
\widehat{\otimes}_{\mathbb{Z}_p} \Lambda_{\mathrm{cyc}}$-linear map 
obtained by restricting the map \eqref{equation:Bigexponential3} 
to $H^1 \big(\mathbb{Q}_{p} ,(\mathbb{T}\widehat{\otimes}_{\mathbb{Z}_p} \Lambda^\sharp_{\mathrm{cyc}}
)^\ast(1) \big) 
\subset 
H^1 \big(\mathbb{Q}_{p} ,(\mathbb{T}\widehat{\otimes}_{\mathbb{Z}_p} \Lambda^\sharp_{\mathrm{cyc}}
)^\ast(1) \big) \widehat{\otimes}_{\Lambda_{\mathrm{cyc}}} 
\HH_{\infty,\mathrm{cyc}}$.
By the growth of the denominators of the map $\mathrm{Col}$ thus defined, 
it is clear that the image of $\mathrm{Col}$ falls in 
$\mathbb{D}^\ast (1) \widehat{\otimes}_{\mathbb{Z}_p} 
\HH_{h,\mathrm{cyc}}
\subset 
\mathbb{D}^\ast (1) \widehat{\otimes}_{\mathbb{Z}_p} 
\HH_{\infty,\mathrm{cyc}}$. 
\par 
\begin{rem}
\begin{enumerate}
\item 
We note some important ingredients of the proof of Theorem \ref{theorem:interpolationdualexp}. 
\begin{enumerate}
\item 
Theorem \ref{theorem:colemanmap2} was proved by reducing to the construction of a Coleman map for a rank-one Galois deformation thanks to Theorem \ref{theorem:localfiltration}, which says that the Galois deformation associated to 
a Hida deformation is locally an extension of rank-one deformations 
(see Remark \ref{remark:Colemanmap_singlemodularform} and Remark \ref{remark:Colemanmap_Hidafamily}). 
Since the analogue of Theorem \ref{theorem:localfiltration} does not hold for Coleman families, the proof of Theorem \ref{theorem:interpolationdualexp} is rather close to original techniques of Perrin-Riou 
in \cite{396} and \cite{Pe01}, which rely on  
gluing argument of exponential maps, based on explicit calculation of 
the growth of denominators. 
\item 
As for the technical difference of \cite{NO16} from \cite{396} and \cite{Pe01}, \cite{NO16} makes use of  
the integral structure $\hh_h$ of the Banach module
$\HH_h$, which did not appear in \cite{396} and \cite{Pe01}.
In \cite{NO16}, we need to calculate some of constants  
which appeared in \cite{396} and make them explicit. 
Since we work in families, we need to assure that all these constants are uniformally bounded with respect to the integral structure $\hh_h$ 
in a Coleman family. 
%
%
%
%
%
%
%
%
%
%
%
%
%
%
\end{enumerate}
\item 
While we were finalizing the full detail of \cite{NO16} some years ago, 
David Hansen \cite{hansen} and Shangwen Wang \cite{wan} have announced some results which might partially overlap with our results on Coleman map 
for non-ordinary $p$-adic families of cuspforms. However, we could not follow their arguments in their preprints \cite{hansen} and \cite{wan} 
and we could not see if the construction of these papers makes sense. 
Also, their statements look different from ours and their techniques relying on $(\varphi ,\Gamma)$-modules 
seem completely different from ours. In a certain sense, we are rather going in a complementary direction. 
\end{enumerate}
\end{rem}

\section{Application to Two-variable Iwasawa Main Conjecture}
\label{section:last}
In this section, we apply Theorem \ref{theorem:interpolationdualexp} and other techniques explained in 
earlier sections to obtain a partial result on the two-variable Iwasawa Main Conjecture for Coleman families. 
\par 
In the ordinary case, 
the construction of Euler systems as in Theorem \ref{theorem:eulersystem_Hida} was quite automatic by taking inverse limits with respect to 
the $p$-power of the level of the modular curves $Y (Np^n)$ as explained in Remark \ref{remark:exixtence_of_Eulersytem_Hida} (1). 
In the non-ordinary case, we do not have such a construction of an Euler system over the family and 
the construction of an Euler system over a Coleman family is possible only under strong conditions by which 
the gluing of Euler systems works with help of Coleman map and $p$-adic $L$-function (see Theorem \ref{theorem:eulersystem_Coleman2}). 
\par 
Let us fix the setting of Coleman families as in \S \ref{section:motivation}. 
We have a certain open disc of radius $p^r$ centered at some weight 
$k_0$ and a Coleman family $\mathbb{F}= \sum^\infty_{n=1} A_n (\mathbb{F}) q^n \in \Lambda_{(k_0 ;r),\mathcal{O}_{{K}}}[[q]]$ over 
$\Lambda_{(k_0 ;r),\mathcal{O}_{{K}}}$. Let $s \in \mathbb{Q}_{\geq }0$ be the slope of the family. By Theorem \ref{theorem:construction_bigrepresentation_non-ordinary}, we associate 
a Galois representation $\mathbb{T} \cong (\Lambda_{(k_0 ;r),\mathcal{O}_{{K}}})^{\oplus 2}$ on which $G_{\mathbb{Q}}$ acts continuously. 
By Kato's result \cite{320} discussed in \S \ref{sec:Intro_motif}, 
we already know that Beilinson-Kato Euler systems exist in a point-wise manner 
at arithmetic points $k$ on $\Lambda_{(k_0 ;r),\mathcal{O}_{{K}}}$. 
We need to construct an Euler system over the whole 
$\Lambda_{(k_0 ;r),\mathcal{O}_{{K}}}$ which amounts to an analogue of Threorem \ref{theorem:eulersystem_Hida2} of ordinary 
case. In the non-ordinary case, the situation is different from the ordinary situation. 
We will show that we can glue Beilinson-Kato Euler systems over arithmetic points to obtain an Euler system over 
$\Lambda_{(k_0 ;r),\mathcal{O}_{{K}}}$ under some assumptions. The construction of an Euler system is carried out 
with help of the Coleman map for a Coleman families that was given by Theorem \ref{theorem:interpolationdualexp} and the existence of 
two-variable $p$-adic $L$-functions in $\Lambda_{(k_0 ;r),\mathcal{O}_{{K}}}  \widehat{\otimes}_{\mathbb{Z}_p} 
\HH_{h,\mathrm{cyc}}$ which will be given by Theorem \ref{theorem:2variablep-adicLnon-ordinary} where $h$ is an integer satisfying $h\geq s$. 
\\ 
\subsection{Two-variable $p$-adic $L$-function on a Coleman family}
The construction of the two-variable $p$-adic $L$-function in the non-ordinary case proceeds in the same manner as the ordinary case (Theorem B in \S \ref{section:motivation2})
except that the interpolated values of the $p$-adic $L$-function has denominators and we have to extend the coefficient to $\HH_{h,\mathrm{cyc}} \supset \Lambda_{\mathrm{cyc}}$. 
Following the construction of \cite{ki94} in the ordinary case 
(see Theorem B in \S \ref{section:motivation2}), we need to introduce 
the module of $\Lambda_{(k_0 ;r),\mathcal{O}_{{K}}}$-adic modular symbols 
``$\mathbb{MS}(\Lambda_{(k_0 ;r),\mathcal{O}_{{K}}})^\pm$'' 
which serve as an analogue of 
the module of $\mathbb{I}$-adic modular symbols $\mathbb{MS}(\mathbb{I})^\pm$ in \S \ref{section:motivation2}. 
\par 
We have the following theorem for the existence of two-variable non-ordinary $p$-adic $L$-function which is 
an non-ordinary analogue of \cite{ki94} as follows. 
\begin{thm}\label{theorem:2variablep-adicLnon-ordinary}
Let $\mathbb{F}$ be a Coleman family over $\Lambda_{(k_0;r)} $ with 
the slope $s\in \mathbb{Q}_{\geq 0}$. Let $h$ be an integer satisfying $h\geq s$. Let $\mathbb{T} \cong (\Lambda_{(k_0;r)})^{\oplus 2}$ be the Galois representation associated to $\mathbb{F}$. Assume that the residual representation $\overline{\rho} : G_{\mathbb{Q}} \longrightarrow \operatorname{GL}_2 (\overline{\mathbb{F}}_p)$ associated 
to the family is irreducible. Then 
\begin{enumerate}
\item 
There exists a 
$\Lambda_{(k_0 ;r),\mathcal{O}_{{K}}}$-module of $\Lambda_{(k_0 ;r),\mathcal{O}_{{K}}}$-adic modular symbols 
$\mathbb{MS}(\Lambda_{(k_0 ;r),\mathcal{O}_{{K}}})^\pm$ which is free of 
rank one over $\Lambda_{(k_0 ;r),\mathcal{O}_{{K}}}$ and interpolates 
the $f_k$-part of classical modular symbols with coefficients in $\mathcal{O}_K$ for arithmetic points $k \in \mathbb{Z}_{>s+1} \cap B(k_0;r)$. 
\item 
Let us fix a $\Lambda_{(k_0;r)} $-basis 
$\Xi^\pm$ of $\mathbb{MS}(\Lambda_{(k_0;r)} )^\pm$ respectively. 
Then, there exists an analytic $p$-adic $L$-function $L_p (\{ \Xi^\pm \}) 
\in \Lambda_{(k_0;r)}  \widehat{\otimes}_{\mathbb{Z}_p} \hh_{h,\mathrm{cyc}}$  
such that 
we have the following interpolation formula:
\begin{equation}\label{equation:interpolation_AV_V_def2}
\dfrac{( k ,  \chi^j_{\mathrm{cyc}} \phi )(L_p (\{ \Xi^\pm \} ))}
{C^{\mathrm{sgn}(j,\phi )}_{f_k ,p}} = (-1)^{j} (j-1)!  e_p (f_k , j ,\phi ) 
 \tau (\phi)  
\dfrac{L(f_k , \phi^{-1},j)}
{(2\pi \sqrt{-1})^{j}\Omega^{\mathrm{sgn}(j,\phi )}_{f_k ,\infty} } ,
\end{equation}
for any arithmetic point $k$ of $\Lambda_{(k_0;r)}$ larger than the slope of the family, for any integer $j$ satisfying $1 \leq j\leq k-1$ and for any Dirichlet 
character $\phi$ whose conductor $\mathrm{Cond} (\phi)$ is a power of $p$,
where $\tau (\phi )$ and $e_p (f, j, \phi )$ is the same as Theorem B of \S \ref{sec:Intro_motif}. 
Here, for each $k \in \mathbb{Z}_{>s+1} \cap B(k_0;r)$, the specialization of 
$\mathbb{MS}(\Lambda_{(k_0 ;r),\mathcal{O}_{{K}}})^\pm $ 
at $k$ is naturally identified with a lattice of 
$H^1_c (Y_1 (M )_{\mathbb{C}} ,
\mathcal{L}_{k-2} (\mathbb{Q}_{f_k} ))^{\pm}[f_k ] 
\otimes_{\mathbb{Q}_{f_k}} K$ and $C^{\pm}_{f_k ,p} \in K $
is a $p$-adic period which is defined to be an error term given by \footnote{Recall that the $p$-adic completion of the Hecke field 
$\mathbb{Q}_f$ is contained in the fraction field $K$ of $\mathcal{O}_K$. }: 
\begin{equation}\label{equation:definition_of_p-adic periodsnon-ord}
\Xi^\pm (k) = C^{\pm}_{f_k ,p} \cdot b^\pm _{f_k} \otimes 1 .
\end{equation}
Furthermore, if we choose a complex period $\Omega^{\mathrm{sgn}(j,\phi )}_{f_k ,\infty}$ to be $p$-optimal, 
the $p$-adic period 
$C^{\mathrm{sgn}(j,\phi )}_{f_k ,p}$ is a $p$-adic unit for any arithmetic point $k$ of $\Lambda_{(k_0;r)}$ larger than the slope of the family. 
\end{enumerate}
\end{thm}
\begin{rem}
\begin{enumerate}
\item 
The above theorem is a non-ordinary generalization of Theorem $B'$ which is due to Kitagawa \cite{ki94} and Greenberg-Stevens \cite{gr93}. 
\item 
Glenn Stevens announced such a result early 90's by developing method of families
of distributions-valued modular symbols over the weight space. 
Though it was never published, Jo\"{e}l Bella\"{i}che \cite{be12} publishes a result 
of two variable $p$-adic $L$-functions over eigencurves along Stevens method (see \cite[Theorem 3]{be12}). 
\item The essential idea of the construction of \cite[Theorem 3]{be12} by the method of modular symbols 
is similar to the method of this article (and also to that of Kitagawa and Greenberg-Stevens in the case of Hida family). 
However, there appear no $p$-adic periods $C^{\mathrm{sgn}(j,\phi )}_{f_k ,p}$ in \cite[Theorem 3]{be12}
and the interpolation property in the two-variable $p$-adic $L$-function of \cite[Theorem 3]{be12} at each cuspform $f_k$ 
which is a member in the $p$-adic family has 
an ambiguity of multiplication by a non-zero $p$-adic number $c_k$ which is not necessarily a $p$-adic unit.  
\item 
Panchishkin \cite{pa03} also published a result on a similar to Theorem \ref{theorem:2variablep-adicLnon-ordinary} which. 
His construction is based on another method called the Rankin-Selberg method. 
\end{enumerate}
\end{rem}
The following Remark is parallel to Remark \ref{remark:padicperiodMS} 
concerning the ordinary case and it is important that it assures that the interpolation property of Theorem \ref{theorem:2variablep-adicLnon-ordinary} is well-defined. 
\begin{rem}\label{remark:padicperiodMSnon-ord}
As was cautioned also in the ordinary case, 
$C^{\pm}_{f_k ,p} $ and 
$\Omega^{\pm}_{f_k ,\infty} $ 
depend on the choice of a $\mathbb{Q}_{f_k}$-basis $b^\pm_{f_k}$ on $H^1_c (Y_1 (M)_{\mathbb{C}} ,
\mathcal{L}_{k -2} (\mathbb{Q}_{f_k}))^{\pm}[f_k ]$ where $M$ is the level of $f_k$. 
However, the ``ratio'' 
is independent of the choice of $b^\pm_{f_{\kappa}}$. If we denote the 
$p$-adic period and the complex period obtained by another 
$\mathbb{Q}_{f_k}$-basis $(b^\pm_{f_k})'$ on $H^1_c (Y_1 (M)_{\mathbb{C}} , 
\mathcal{L}_{k -2} (\mathbb{Q}_{f_k}))^{\pm}[f_k ]$ by 
$(C^{\pm}_{f_k ,p} )'$ and 
$(\Omega^{\pm}_{f_k ,\infty} )'$, we have 
$$
\dfrac{C^{\pm}_{f_k ,p}}{(C^{\pm}_{f_k ,p})'} =
\dfrac{\Omega^{\pm}_{f_k ,\infty}}{(\Omega^{\pm}_{f_k ,\infty})'}. 
$$ 
\end{rem}
\begin{proof}[Proof of Theorem \ref{theorem:2variablep-adicLnon-ordinary}]
Recall that, when $k$ varies in $\mathbb{Z}_{>s+1} \cap B(k_0;r)$, the compactly supported cohomology group 
$H^1_c (Y_1 (M )_{\mathbb{C}} ,\mathcal{L}_{k-2} (\mathcal{O}_{K}))^{\pm}[f_k ]$ is naturally identified with the modules of modular symbols with coefficient in $\mathcal{O}_K$ (To understand the identification, 
see \cite[\S 3.2]{ki94}).
\par 
Since we assume that the residual representation $G_{\mathbb{Q}} \longrightarrow \mathrm{Aut}_{\Lambda_{(k_0 ;r),\mathcal{O}_{{K}}}/\mathfrak{M}}(\mathbb{T}/\mathfrak{M} 
\mathbb{T})$ associated to the Coleman family $\mathbb{F}$ is irreducible,  
we have a natural isomorphism  
\begin{equation}
H^1_\mathrm{par} (Y_1 (M )_{\mathbb{C}} ,\mathcal{L}_{k-2} (\mathcal{O}_{K}))^{\pm}[f_k ]
\longrightarrow 
H^1_c (Y_1 (M )_{\mathbb{C}} ,\mathcal{L}_{k-2} (\mathcal{O}_{K}))^{\pm}[f_k ]
\end{equation} 
from the parabolic cohomology group to the compactly supported cohomology group 
at any arithmetic point $k \in \mathbb{Z} \cap B(k_0;r)$ larger than $s+1$. On the other hand, by the comparison of the Betti cohomology and the \'{e}tale cohomology, we have an isomorphism 
\begin{equation}\label{equation:H^1pararithmetic}
H^1_\mathrm{par} (Y_1 (M )_{\mathbb{C}} ,\mathcal{L}_{k-2} (\mathcal{O}_{K}))^{\pm}[f_k ] \cong 
H^1_\mathrm{par, \text{\'{e}t}} (Y_1 (M )_{\overline{\mathbb{Q}}} ,\mathcal{L}_{k-2} (\mathcal{O}_{K}))^{\pm}[f_k ] .
\end{equation}
\par Recall that $H^1_\mathrm{par, \text{\'{e}t}} (Y_1 (M )_{\overline{\mathbb{Q}}} ,\mathcal{L}_{k-2} (\mathcal{O}_{K}))[f_k ]$ which appears in the right-hand side of \eqref{equation:H^1pararithmetic} is 
naturally equipped with continuous Galois action of $G_{\mathbb{Q}}$ by 
the functoriality of \'{e}tale cohomology theory and it is an 
integral lattice in the Galois representation associated to $f_k$ as 
constructed in \cite{del71}. Hence $\mathbb{T}^\pm$ interpolates 
$H^1_\mathrm{par, \text{\'{e}t}} (Y_1 (M )_{\overline{\mathbb{Q}}} ,\mathcal{L}_{k-2} (\mathcal{O}_{K}))[f_k ]$ when $k$ varies in the arithmetic points $k \in \mathbb{Z} \cap B(k_0;r)$ larger than $s+1$. 
Thanks to the above identification, we define the 
$\Lambda_{(k_0 ;r),\mathcal{O}_{{K}}}$-module of $\Lambda_{(k_0 ;r),\mathcal{O}_{{K}}}$-adic modular symbols 
$\mathbb{MS}(\Lambda_{(k_0 ;r),\mathcal{O}_{{K}}})^\pm$ 
to be 
$$
\mathbb{MS}(\Lambda_{(k_0 ;r),\mathcal{O}_{{K}}})^\pm = 
\mathbb{T}^\pm 
$$ 
for each of the signs $\pm$. 
\par 
Once the space of modular symbols are interpolated and we have 
the modules of $\Lambda_{(k_0 ;r),\mathcal{O}_{{K}}}$-adic modular symbols $\mathbb{MS}(\Lambda_{(k_0 ;r),\mathcal{O}_{{K}}})^\pm $, the construction of 
the two-variable $p$-adic $L$-function $L_p (\{ \Xi^\pm \} )$ as stated in Theorem \ref{theorem:2variablep-adicLnon-ordinary} is parallel to the construction of 
the two-variable $p$-adic $L$-function done by \cite{ki94} out of the existence of modules of $\Lambda$-adic modular symbols 
in the ordinary case, except that we need to extend the coefficient to 
$\HH_{h,\mathrm{cyc}}$ because of denominators. 
\par 
For each $k \in \mathbb{Z} \cap B(k_0;r)$, we denote by $I_k$ the principal prime ideal of $\Lambda_{(k_0 ;r),\mathcal{O}_{{K}}} $,   
which corresponds to the closed point $k$ of $\cap B(k_0;r)$. 
Let us define the following set 
\begin{equation}\label{equation:setS1}
\mathfrak{S}=\{ I_k  
\, \vert \, k \in \mathbb{Z}_{> s+1 } \cap B(k_0;r) 
\}. 
\end{equation}
We denote by $\mathfrak{A}$ a subset of the set of 
height one ideals of $\Lambda_{(k_0 ;r),\mathcal{O}_{{K}}}$ as follows:
\begin{equation}\label{equation:A_nonord1}
\mathfrak{A} = \left\{ \left. J = {\bigcap}_{I \in S} I 
\ 
\right\vert \ 
S \subset \mathfrak{S}, \sharp S <\infty
 \right\} .
\end{equation}
Note that $J\cap J' \in \mathfrak{A}$ for any $J,J' \in \mathfrak{A}$ and that 
the intersection of infinitely many 
distinct elements of $\mathfrak{A}$ is zero. 
\par 
For each $J \in \mathfrak{A}$, we denote by $\mathbb{MS}(\Lambda_{(k_0 ;r),\mathcal{O}_{{K}}})^\pm_J$ the space of modular symbols modulo 
$J$. We have $\mathbb{MS}(\Lambda_{(k_0 ;r),\mathcal{O}_{{K}}})^\pm \cong \varprojlim_J \mathbb{MS}(\Lambda_{(k_0 ;r),\mathcal{O}_{{K}}})^\pm_J $. 
For each $J$, by evaluating modular symbols at $\{ 0 \} -\{ \infty \}$, we have a construction of measures by the Mellin transformation 
$$ 
\mathbb{MS}(\Lambda_{(k_0 ;r),\mathcal{O}_{{K}}})^\pm_J \longrightarrow (\Lambda_{(k_0 ;r),\mathcal{O}_{{K}}}/J ) \otimes_{\mathbb{Z}_p} \hh_h 
$$ 
which makes a projective system. The choice of a $\Lambda_{(k_0 ;r),\mathcal{O}_{{K}}}$-basis $\{ \Xi^\pm \} $ of $\mathbb{MS}(\Lambda_{(k_0 ;r),\mathcal{O}_{{K}}})^\pm$ defines a $\Lambda_{(k_0 ;r),\mathcal{O}_{{K}}}/J$-basis $\{ \Xi^\pm_J \} $ of $\mathbb{MS}(\Lambda_{(k_0 ;r),\mathcal{O}_{{K}}})^\pm_J$, 
which provides a unique element 
\begin{equation}
L_p (\{ \Xi^\pm \} )_{J} \in 
(\Lambda_{(k_0 ;r),\mathcal{O}_{{K}}}/J ) \otimes_{\mathbb{Z}_p} \hh_h  
\end{equation}
such that we have the following interpolation formula:
\begin{equation*}\dfrac{( k ,  \chi^j_{\mathrm{cyc}} \phi )
(L_p (\{ \Xi^\pm \} )_{J})}
{C^{\mathrm{sgn}(j,\phi )}_{f_k ,p}} = (-1)^{j} (j-1)!  e_p (f_k , j ,\phi ) 
 \tau (\phi)   
\dfrac{L(f_k , \phi^{-1},j)}
{(2\pi \sqrt{-1})^{j}\Omega^{\mathrm{sgn}(j,\phi )}_{f_k ,\infty} } ,
\end{equation*}
for any arithmetic point $k \in  \mathbb{Z}_{> s+1 } \cap B(k_0;r)$ such that $J \subset I_k$, 
for any integer $j$ satisfying $l \leq j\leq l'$ and for any Dirichlet 
character $\phi$ whose conductor $\mathrm{Cond} (\phi)$ divides $p^n$. 
The set of elements $\{L_p (\{ \Xi^\pm \} )_{J}\}$ is a projective system with respect to $J \in \mathfrak{A}$. 
Thanks to Proposition \ref{proposition:glueingpropertyHh}, 
the projective limit $L_p (\{ \Xi^\pm \} ) 
=\varprojlim_{J} L_p (\{ \Xi^\pm \} )_{J}$ defines a unique element of $\Lambda_{(k_0;r)} \widehat{\otimes}_{\mathbb{Z}_p} \hh_{h,\mathrm{cyc}}$
satisfying the desired interpolation property \eqref{equation:interpolation_AV_V_def2} 
for any arithmetic point $k \in\mathfrak{S}$ , for any integer $j$ satisfying $1\leq j\leq k-1$ and for any Dirichlet 
character $\phi$ whose conductor $\mathrm{Cond} (\phi)$ is a power of 
$p$. This completes the proof. 
\end{proof}
\subsection{Construction of Beilinson-Kato Euler system on a Coleman family} 
As an application of Theorem \ref{theorem:2variablep-adicLnon-ordinary}, we 
have the following theorem on the existence of an Euler system associated to a family of Galois representation $\mathbb{T}$ for a given Colaman family $\mathbb{F}$.  
\begin{thm}\label{theorem:eulersystem_Coleman2}
Let us assume the setting of Theorem \ref{theorem:2variablep-adicLnon-ordinary}. Let us fix a $\Lambda_{(k_0 ;r),\mathcal{O}_{{K}}}$-basis $\Xi^\pm$ of $\mathbb{MS}(\Lambda_{(k_0 ;r),\mathcal{O}_{{K}}})^\pm$ respectively. 
Assume that the image of 
the residual representation $\overline{\rho} : G_{\mathbb{Q}} \longrightarrow \operatorname{GL}_2 (\overline{\mathbb{F}}_p)$ associated 
to the family is non-solvable and 
that $\overline{\rho}$ is irreducible when restricted to $G_{\Q }$. 
Then, the following statements hold: 
\begin{enumerate}
\item 
We have 
an Euler system $($depending on the choice of $\Xi^\pm )$:
$$
\bigl\{ \mathcal{Z} (r) \in H^1(\mathbb{Q} (\mu_r ) ,
(\mathbb{T}\widehat{\otimes}_{\mathbb{Z}_p} \Lambda^\sharp_{\mathrm{cyc}}
)^\ast (1)) \bigr\} _{r\in \mathcal{S} }
$$ 
which satisfies the axiom of Euler system given Definition \ref{definition:eulersystem_general} for $\mathcal{R}=
\Lambda_{(k_0 ;r),\mathcal{O}_{{K}}}$. 
\item 
The elements $\mathcal{Z} (r) 
\in H^1(\mathbb{Q} (\mu_r ) ,
(\mathbb{T}\widehat{\otimes}_{\mathbb{Z}_p} \Lambda^\sharp_{\mathrm{cyc}}
)^\ast (1))$ have the following relation to the special value: 
\begin{equation}\label{equation:specialvalue_at_bottom3}
\dfrac{\psi \left( \mathrm{exp^\ast} (\mathrm{loc}_p ( (\kappa , \chi^{-j}_{\mathrm{cyc}}\phi^{-1})
(\mathcal{Z} (r)))) \right)}{C^{\mathrm{sgn}(j,\phi \psi )}_{f_\kappa ,p}}
= \tau (\phi \psi ) 
\dfrac{L_{(pr)}(f_\kappa , \psi^{-1} \phi^{-1},j)}
{(2\pi \sqrt{-1})^{j}\Omega^{\mathrm{sgn}(j,\phi \psi )}_{f_\kappa ,\infty} } \cdot 
\overline{f}_\kappa \otimes \psi^{-1}\phi^{-1},
\end{equation}
for any $k \in \mathbb{Z}_{\geq s+1} \cap B(k_0;r)$, any integer $j$ satisfying $1\leq j\leq k-1$ and any Dirichlet 
character $\phi$ whose conductor $\mathrm{Cond} (\phi)$ is a power of $p$ and any character $\psi$ of $\mathrm{Gal} (\mathbb{Q}(\mu_r)/\mathbb{Q})$. 
Here $\Omega^{\mathrm{sgn}(j,\phi )}_{f_k ,\infty}$ is 
a complex period which satisfies the algebraicity which was shown in Theorem of Shimura presented of \S 
\ref{section:Selmer group and $p$-adic $L$-function} and 
$C^{\mathrm{sgn}(j,\phi )}_{f_k ,p}$ is a $p$-adic period defined in 
\S \ref{definition:padicperiodMS} \footnote{See Remark \ref{remark:padicperiodMS} for an explanation that the interpolation property is well-defined.}. 
  \par 
In particular, we have the equality
\begin{equation}\label{equation:modified_Euler2}
\mathrm{Col}(\mathrm{loc}_p (\mathcal{Z} (1)))= 
L_p (\{ \Xi^\pm \})
\end{equation}
in $\Lambda_{(k_0 ;r),\mathcal{O}_{{K}}} \widehat{\otimes}_{\mathbb{Z}_p} 
\HH_{h,\mathrm{cyc}}$.
\end{enumerate}
\end{thm}
We prove Theorem \ref{theorem:eulersystem_Coleman2} by using Theorem \ref{theorem:interpolationdualexp}. 
\begin{proof}[Proof of Theorem \ref{theorem:eulersystem_Coleman2}]
The proof proceeds in the same manner as the proof of Theorem 
\ref{theorem:eulersystem_Hida2}, which is an analogue of 
Theorem \ref{theorem:eulersystem_Coleman2} in the ordinary case \footnote{In the ordinary case, Theorem 
\ref{theorem:eulersystem_Hida2} provided us with a $p$-optimization of an Euler system which already existed by Theorem \ref{theorem:eulersystem_Hida}. 
Note that there exists no analogue of Theorem \ref{theorem:eulersystem_Hida} in the non-ordinary case and the method of the proof of Theorem 
\ref{theorem:eulersystem_Hida2} plays a more essential role in the non-ordinary case.}. 
\par 
Similarly as in the proof of Theorem \ref{theorem:2variablep-adicLnon-ordinary}, we define the following set 
\begin{equation}\label{equation:setS3}
\mathfrak{S}=\{ I_k   
\, \vert \, k \in \mathbb{Z}_{> s} \cap B(k_0;r) 
\}. 
\end{equation}
We denote by $\mathfrak{A}$ the following subset of 
height one ideals of $\Lambda_{(k_0 ;r),\mathcal{O}_{{K}}} \widehat{\otimes}_{\mathbb{Z}_p}\Lambda_{\mathrm{cyc}}$: 
\begin{equation}\label{equation:A_nonord}
\mathfrak{A} = \left\{ \left. J = {\bigcap}_{I \in S} I 
\cdot \Lambda_{(k_0 ;r),\mathcal{O}_{{K}}} \widehat{\otimes}_{\mathbb{Z}_p}\Lambda_{\mathrm{cyc}}
\ 
\right\vert \ 
S \subset \mathfrak{S}, \sharp S <\infty
 \right\} .
\end{equation}
Note again that $J\cap J' \in \mathfrak{A}$ for any $J,J' \in \mathfrak{A}$ and that 
the intersection $\bigcap J$ for infinitely many 
$J \in \mathfrak{A}$ is zero. We remark that $\Lambda_{(k_0 ;r),\mathcal{O}_{{K}}} \widehat{\otimes}_{\mathbb{Z}_p}\Lambda_{\mathrm{cyc}}$ 
is a semi-local ring whose local components are all isomorphic to the ring of power series in two variables $\mathcal{O}_{{K}} [[X_1 ,X_2]]$. 
In particular, local components are regular local rings. 
\par 
Recall that we defined $\Sigma_r$ to be 
$\Sigma_r =\Sigma \cup \{ \text{primes $q$ dividing $r$}\}$ 
for any natural number $r$.  
We denote by $\mathbb{Q}_{\Sigma_r}$ the maximal unramified extension of $\mathbb{Q}$ unramified outside $\Sigma_r$. 
We have the following lemma which is the non-ordinary analogue of 
Lemma \ref{lem:glues}. 
\begin{lemma}\label{lem:glues_nonord}
Let us assume the setting of Theorem \ref{theorem:eulersystem_Coleman2}. Let $r$ be a natural number and $J, J' \in \mathfrak{A}$. 
Then, we have the exact sequence$:$ 
\begin{multline}\label{equation:glying_nonord}
0 \longrightarrow H^1 (((\mathbb{T}\widehat{\otimes}_{\mathbb{Z}_p} \Lambda^\sharp_{\mathrm{cyc}}
)^\ast(1))_{J\cap J'})
\longrightarrow H^1 (((\mathbb{T}\widehat{\otimes}_{\mathbb{Z}_p} \Lambda^\sharp_{\mathrm{cyc}}
)^\ast(1))_J)
\oplus H^1 (((\mathbb{T}\widehat{\otimes}_{\mathbb{Z}_p} \Lambda^\sharp_{\mathrm{cyc}}
)^\ast(1))_{J'}) 
\\ 
\longrightarrow H^1 (((\mathbb{T}\widehat{\otimes}_{\mathbb{Z}_p} \Lambda^\sharp_{\mathrm{cyc}}
)^\ast(1))_{J+J'}) . 
\end{multline}
Here we denote $(\mathbb{T}\widehat{\otimes}_{\mathbb{Z}_p} \Lambda^\sharp_{\mathrm{cyc}})^\ast (1)/J(\mathbb{T}\widehat{\otimes}_{\mathbb{Z}_p} \Lambda^\sharp_{\mathrm{cyc}} )^\ast(1)$ by $((\mathbb{T}\widehat{\otimes}_{\mathbb{Z}_p} \Lambda^\sharp_{\mathrm{cyc}})^\ast (1))_J$ 
and the cohomology group $H^1 (\mathbb{Q} (\mu_r)_{\Sigma_r}/\mathbb{Q} (\mu_r),((\mathbb{T}\widehat{\otimes}_{\mathbb{Z}_p} \Lambda^\sharp_{\mathrm{cyc}})^\ast (1))_J)$ by $H^1 (\mathbb{T}\widehat{\otimes}_{\mathbb{Z}_p} \Lambda^\sharp_{\mathrm{cyc}}(1)_J)$ 
for any ideal $J$ of $\Lambda_{(k_0 ;r),\mathcal{O}_{{K}}} \widehat{\otimes}_{\mathbb{Z}_p} \Lambda_{\mathrm{cyc}}$. 
\par 
The first map of \eqref{equation:glying_nonord} sends each element $x_{J\cap J'}$ to $(x_{J\cap J'} \text{ mod $J$})\oplus (x_{J\cap J'} \text{ mod $J'$})$. 
The second map \eqref{equation:glying_nonord} sends each element $x_{J} \oplus y_{J'}$ to 
$(x_{J} \text{ mod $J+J'$} )- (y_{J'} \text{ mod $J+J'$} )$.  
\end{lemma}
Let us denote by $\mathrm{Col}_J$ the Coleman map for 
$((\mathbb{T}\widehat{\otimes}_{\mathbb{Z}_p} \Lambda^\sharp_{\mathrm{cyc}}
)^\ast(1))_J$, which is the cyclotomic deformation of a usual $p$-adic Galois representation $\mathbb{T}/J\mathbb{T}$. 
The map $\mathrm{Col}_J$ is obtained by taking the $\mathcal{H}_{\infty ,\mathrm{cyc}}$-linear dual of Perrin-Riou's big exponential map obtained in 
\cite{396}. 
By observing the denominators appearing in the interpolation, we easily see 
the image of $\mathrm{Col}$ falls in 
the submodule $\big( \mathbb{D}^\ast (1)/J \mathbb{D}^\ast (1)\big)  
{\otimes}_{\mathbb{Z}_p} \HH_{h,\mathrm{cyc}}
\subset 
\big( \mathbb{D}^\ast (1)/J \mathbb{D}^\ast (1)\big)  
{\otimes}_{\mathbb{Z}_p} \HH_{\infty,\mathrm{cyc}}$. 
We consider the following composed morphism$:$ 
\begin{multline}
H^1 (\mathbb{Q}_{\Sigma}/\mathbb{Q} ,((\mathbb{T}\widehat{\otimes}_{\mathbb{Z}_p} \Lambda^\sharp_{\mathrm{cyc}})^\ast(1))_J) 
\xrightarrow{\mathrm{loc}} 
{H^1 (\mathbb{Q}_p ,((\mathbb{T}\widehat{\otimes}_{\mathbb{Z}_p} \Lambda^\sharp_{\mathrm{cyc}}
)^\ast(1))_J)} 
\\ 
\xrightarrow{\mathrm{Col}_J} 
\big( \mathbb{D}^\ast (1)/J \mathbb{D}^\ast (1)\big)
{\otimes}_{\mathbb{Z}_p} \HH_{h,\mathrm{cyc}}
\xrightarrow{\langle d_J \otimes 1 ,\ \rangle } 
(\Lambda_{(k_0 ;r),\mathcal{O}_{{K}}} /J )  
{\otimes}_{\mathbb{Z}_p} \HH_{h,\mathrm{cyc}}
\end{multline}
where $d_J$ is a basis of the free $\Lambda_{(k_0 ;r),\mathcal{O}_{{K}}}/J$-module $\mathbb{D}/J\mathbb{D}$ of rank one which coincides 
with $\overline{f}_k$ for any $k \in \mathbb{Z}_{\geq s+1} \cap B(k_0;r)$ where $\mathbb{D}$ is the free 
$\Lambda_{(k_0 ;r),\mathcal{O}_{{K}}}$-module of rank one constructed 
in Lemma \ref{lemma:bigD_nonord}. 
Moreover, the symbol $\langle \ ,\ \rangle $ represents the paring:
$$
\big( \mathbb{D}/J \mathbb{D} \big) 
{\otimes}_{\mathbb{Z}_p} \HH_{\infty,\mathrm{cyc}}
\times \big( \mathbb{D}^\ast (1)/J \mathbb{D}^\ast (1)\big) 
{\otimes}_{\mathbb{Z}_p} \HH_{\infty,\mathrm{cyc}}
 \longrightarrow (\Lambda_{(k_0 ;r),\mathcal{O}_{{K}}} /J )  
{\otimes}_{\mathbb{Z}_p} \HH_{\infty,\mathrm{cyc}}
$$  
and we note that the image of the composition map 
$\langle d_J \otimes 1 ,\ \rangle \circ \mathrm{Col}$ falls in 
the submodule $(\Lambda_{(k_0 ;r),\mathcal{O}_{{K}}} /J )  
{\otimes}_{\mathbb{Z}_p} \HH_{h,\mathrm{cyc}}
\subset 
(\Lambda_{(k_0 ;r),\mathcal{O}_{{K}}} /J )  
{\otimes}_{\mathbb{Z}_p} \HH_{\infty,\mathrm{cyc}}$. 
The following lemma will be used in the following arguments. 
\begin{lemma}\label{lemma:injectivityofcomposite}
Let us assume the setting of Theorem \ref{theorem:eulersystem_Coleman2}. 
Then the map 
$$
\langle d_J \otimes 1 ,\ \rangle \circ \mathrm{Col}_J \circ \mathrm{loc}: 
H^1 (\mathbb{Q}_{\Sigma}/\mathbb{Q},((\mathbb{T}\widehat{\otimes}_{\mathbb{Z}_p} \Lambda^\sharp_{\mathrm{cyc}})^\ast(1))_J) 
\longrightarrow (\Lambda_{(k_0 ;r),\mathcal{O}_{{K}}} /J )  
{\otimes}_{\mathbb{Z}_p} \HH_{h,\mathrm{cyc}}
$$ 
is injective for every $J \in \mathfrak{A}$. Similarly, for any natural number $r$, the composed morphism 
\begin{multline*}
H^1 (\mathbb{Q}_{\Sigma_r}/\mathbb{Q}(\mu_r) ,((\mathbb{T}\widehat{\otimes}_{\mathbb{Z}_p} \Lambda^\sharp_{\mathrm{cyc}})^\ast(1))_J)
\xrightarrow{\mathrm{loc}} 
{H^1 (\mathbb{Q}_p \otimes_{\mathbb{Q}} \mathbb{Q} (\mu_r ) ,((\mathbb{T}\widehat{\otimes}_{\mathbb{Z}_p} \Lambda^\sharp_{\mathrm{cyc}}
)^\ast(1))_J)} 
\\ 
\xrightarrow{\mathrm{Col}^{(r)}_J} 
\big( \mathbb{D}^\ast (1)/J \mathbb{D}^\ast (1)\big)
{\otimes}_{\mathbb{Z}_p} \HH_{h,\mathrm{cyc}} \otimes_{\mathbb{Q}} \mathbb{Q} (\mu_r )
\xrightarrow{\langle d_J \otimes 1 ,\ \rangle } 
(\Lambda_{(k_0 ;r),\mathcal{O}_{{K}}} /J )  
{\otimes}_{\mathbb{Z}_p} \HH_{h,\mathrm{cyc}} \otimes_{\mathbb{Q}} \mathbb{Q} (\mu_r )
\end{multline*}
is injective for every $J \in \mathfrak{A}$ where $\mathrm{Col}^{(r)}_J$ is a Coleman map over $\mathbb{Q}_p \otimes_{\mathbb{Q}}\mathbb{Q} (\mu_r)$ 
which is similar as $\mathrm{Col}_J$ above. 
\end{lemma}
Though we do not go into the proof of the lemma, we remark that the lemma follows from the fact that 
$H^1 (\mathbb{Q}_{\Sigma_r }/\mathbb{Q}(\mu_r) ,((\mathbb{T}\widehat{\otimes}_{\mathbb{Z}_p} \Lambda^\sharp_{\mathrm{cyc}})^\ast(1))_J) $
is not torsion over any local component of $\Lambda_{\mathrm{cyc}}$ thanks to the assumption of Theorem \ref{theorem:eulersystem_Coleman2} 
saying that the image of the residual representation is non-solvable. 
\par 
Now we take $I, I'$ in the set $\mathfrak{S}$ given in \eqref{equation:setS3}. Let us set 
$J= I \cdot \Lambda_{(k_0 ;r),\mathcal{O}_{{K}}} \widehat{\otimes}_{\mathbb{Z}_p} \Lambda_{\mathrm{cyc}}$ and $J'= I' \cdot \Lambda_{(k_0 ;r),\mathcal{O}_{{K}}}\widehat{\otimes}_{\mathbb{Z}_p} \Lambda_{\mathrm{cyc}}$. 
\par 
Recall that we denote by $\mathcal{S}$ the set of all square-free natural numbers 
which are prime to $\Sigma$. 
As explained before Theorem \ref{theorem:eulersystem_Hida2}, 
the result of Kato for the cyclotomic deformation of a cuspform assures the existence of the Euler system: 
\begin{align*}
& \bigl\{ \mathcal{Z}_J (r) \in H^1(\mathbb{Q} (\mu_r ) ,
((\mathbb{T}\widehat{\otimes}_{\mathbb{Z}_p} \Lambda^\sharp_{\mathrm{cyc}}
)^\ast(1))_J ) \bigr\} _{r\in \mathcal{S} } 
\\ 
& \Bigl( \mathrm{resp}.\,  \bigl\{ \mathcal{Z}_{J'} (r) \in H^1(\mathbb{Q} (\mu_r ) ,
((\mathbb{T}\widehat{\otimes}_{\mathbb{Z}_p} \Lambda^\sharp_{\mathrm{cyc}}
)^\ast(1))_{J'}) \bigr\} _{r\in \mathcal{S} }
\Bigr)
\end{align*} 
such that we have 
\begin{align*}
& 
\mathrm{Col}^{(r)}(\mathrm{loc}_p (\mathcal{Z}_J  (r)))= 
L^{(r)}_p (\{ \Xi^\pm \}) \text{ mod $J$ in $(\Lambda_{(k_0 ;r),\mathcal{O}_{{K}}}/J) {\otimes}_{\mathbb{Z}_p} 
\hh_{h,\mathrm{cyc}}$}
\\ 
& \Bigl( \mathrm{resp}.\, 
\mathrm{Col}^{(r)}(\mathrm{loc}_p (\mathcal{Z}_{J'} (r)))= 
L^{(r)}_p (\{ \Xi^\pm \}) \text{ mod $J'$ in $(\Lambda_{(k_0 ;r),\mathcal{O}_{{K}}}/J') {\otimes}_{\mathbb{Z}_p} \hh_{h,\mathrm{cyc}}$}\Bigr)
\end{align*} 
where $L^{(r)}_p (\{ \Xi^\pm \})\in \Lambda_{(k_0 ;r),\mathcal{O}_{{K}}} 
{\otimes}_{\mathbb{Z}_p} \hh_{h,\mathrm{cyc}}$ is the $p$-adic $L$-function obtained by removing Euler factors 
at primes dividing $r$ from $L_p (\{ \Xi^\pm \})$. Note that, since $L^{(r)}_p (\{ \Xi^\pm \})$ is an element of $\Lambda_{(k_0 ;r),\mathcal{O}_{{K}}} {\otimes}_{\mathbb{Z}_p} 
\hh_{h,\mathrm{cyc}}$, $L^{(r)}_p (\{ \Xi^\pm \}) \text{ mod $J$}$ and $L^{(r)}_p (\{ \Xi^\pm \}) \text{ mod $J'$}$  
can be glued together to have a unique element $L^{(r)}_p (\{ \Xi^\pm \}) \text{ mod $J\cap J'$}$ of 
$(\mathbb{I}/(J \cap J')) {\otimes}_{\mathbb{Z}_p} \hh_{h,\mathrm{cyc}}$. 
\par 
By applying 
Lemma \ref{lemma:injectivityofcomposite} to the ideal $J \cap J'  \in \mathfrak{A}$, 
we have a unique element $\mathcal{Z}_{J \cap J'} (r) \in H^1(\mathbb{Q}(\mu_r) ,
((\mathbb{T}\widehat{\otimes}_{\mathbb{Z}_p} \Lambda^\sharp_{\mathrm{cyc}}
)^\ast(1))_{J\cap J'} )$ such that $(\langle d_J \otimes 1 ,\ \rangle \circ \mathrm{Col}^{(r)}_J \circ \mathrm{loc}
) (\mathcal{Z}_{J \cap J'} (r))$ is equal to $L^{(r)}_p (\{ \Xi^\pm \}) \text{ mod $J\cap J'$}$. 
%
%
%
%
%
%
%
%
Thus, we obtain a unique Euler system 
$$
\bigl\{ \mathcal{Z}_{J \cap J'} (r) \in H^1(\mathbb{Q} (\mu_r ) ,
((\mathbb{T}\widehat{\otimes}_{\mathbb{Z}_p} \Lambda^\sharp_{\mathrm{cyc}}
)^\ast(1))_{J\cap J'} ) \bigr\} _{r\in \mathcal{S} } 
$$
which glues Euler systems $\bigl\{ \mathcal{Z}_{J'} (r) \bigr\} _{r\in \mathcal{S} }$ and 
$\bigl\{ \mathcal{Z}_{J'} (r)\bigr\}_{r\in \mathcal{S} } $. 
By Lemma \ref{lem:glues_nonord} and by repeating the induction on 
the numbers of the ideals of the form $I \cdot 
\Lambda_{(k_0 ;r),\mathcal{O}_{{K}}}\widehat{\otimes}_{\mathbb{Z}_p} \Lambda_{\mathrm{cyc}}$ with $I \in \mathfrak{S}$ 
containing a given ideal $J \in \mathfrak{A}$, we construct for every $J \in \mathfrak{A}$ a unique Euler system 
$
\bigl\{ \mathcal{Z}_J (r) \in H^1(\mathbb{Q} (\mu_r ) ,
((\mathbb{T}\widehat{\otimes}_{\mathbb{Z}_p} \Lambda^\sharp_{\mathrm{cyc}}
)^\ast(1))_J ) \bigr\} _{r\in \mathcal{S} } 
$ 
which provides an Euler system 
for each cuspform corresponding to the representation mod $I \cdot \Lambda_{(k_0 ;r),\mathcal{O}_{{K}}}\widehat{\otimes}_{\mathbb{Z}_p} \Lambda_{\mathrm{cyc}}$ with $I \in \mathfrak{S}$ such 
that $J \subset I \cdot \Lambda_{(k_0 ;r),\mathcal{O}_{{K}}}\widehat{\otimes}_{\mathbb{Z}_p} \Lambda_{\mathrm{cyc}}$. 
\par 
The Euler systems 
$\bigl\{ \mathcal{Z}_J (r)  \bigr\} _{r\in \mathcal{S} }$ form an inverse system for $J \subset \mathfrak{A}$. 
By Chevalley's theorem stated in the proof of Theorem $\ref{theorem:eulersystem_Hida2}$, we have 
$$
H^1(\mathbb{Q} (\mu_r ) ,
((\mathbb{T}\widehat{\otimes}_{\mathbb{Z}_p} \Lambda^\sharp_{\mathrm{cyc}}
)^\ast(1)) )
\cong 
\varprojlim_{J \in \mathfrak{A}} H^1(\mathbb{Q} (\mu_r ) ,
((\mathbb{T}\widehat{\otimes}_{\mathbb{Z}_p} \Lambda^\sharp_{\mathrm{cyc}}
)^\ast(1))_J ). 
$$
Hence, we obtain the desired Euler system 
$
\bigl\{ \mathcal{Z}(r) \in H^1(\mathbb{Q} (\mu_r ) ,
((\mathbb{T}\widehat{\otimes}_{\mathbb{Z}_p} \Lambda^\sharp_{\mathrm{cyc}}
)^\ast(1)) ) \bigr\} _{r\in \mathcal{S} } 
$   
by putting $\mathcal{Z}(r)= \varprojlim_{J \in \mathfrak{A}} \mathcal{Z}_J(r)$ and this completes the proof. 
\end{proof}

\subsection{Two-variable Iwasawa Main Conjecture on a Coleman family}
Thanks to Theorem \ref{theorem:eulersystem_Coleman2}, we can now state 
two-variable Iwasawa Main Conjecture.  
\\ 
\begin{conj}[Two-variable Iwasawa Main Conjecture for a Coleman family] \label{conjecture:IMCNonor_alakato}
Let $\mathbb{F}$ be a Coleman family over $\Lambda_{(k_0;r)} $ with 
the slope $s\in \mathbb{Q}_{\geq 0}$. Let $h$ be an integer satisfying $h\geq s$. Let $\mathbb{T} \cong (\Lambda_{(k_0;r)})^{\oplus 2}$ be the Galois representation associated to $\mathbb{F}$. Let us fix an $\Lambda_{(k_0 ;r),\mathcal{O}_{{K}}}$-basis 
$\Xi^\pm$ of $\mathbb{MS}(\Lambda_{(k_0 ;r),\mathcal{O}_{{K}}})^\pm$ respectively. 
Assume that the residual representation $\overline{\rho} : G_{\mathbb{Q}} \longrightarrow \operatorname{GL}_2 (\overline{\mathbb{F}}_p)$ associated 
to the family is irreducible when restricted to $G_{\Q }$. 
\par 
Then, the following statements hold: 
\begin{enumerate}
\item 
The finitely generated $\Lambda_{(k_0 ;r),\mathcal{O}_{{K}}} \widehat{\otimes}_{\mathbb{Z}_p} \Lambda_{\mathrm{cyc}}$-module 
$H^2(\mathbb{Q}_{\Sigma}/\mathbb{Q} ,
(\mathbb{T}\widehat{\otimes}_{\mathbb{Z}_p} \Lambda^\sharp_{\mathrm{cyc}}
)^\ast (1)) $
is torsion. 
\item 
We have the the following equality of principal ideals in $\Lambda_{(k_0 ;r),\mathcal{O}_{{K}}} \widehat{\otimes}_{\mathbb{Z}_p} \Lambda_{\mathrm{cyc}}$:  
\begin{multline}\label{equation:IMC_alaKato_nonord1}
\mathrm{char}_{\Lambda_{(k_0 ;r),\mathcal{O}_{{K}}} \widehat{\otimes}_{\mathbb{Z}_p} \Lambda_{\mathrm{cyc}}}
\left( 
H^1(\mathbb{Q}_{\Sigma}/\mathbb{Q} ,
(\mathbb{T}\widehat{\otimes}_{\mathbb{Z}_p} \Lambda^\sharp_{\mathrm{cyc}}
)^\ast (1)) 
 \Bigl/ \Lambda_{(k_0 ;r),\mathcal{O}_{{K}}} \widehat{\otimes}_{\mathbb{Z}_p} \Lambda_{\mathrm{cyc}}  \cdot \mathcal{Z}(1)  
 \right) 
 \\ 
= \mathrm{char}_{\Lambda_{(k_0 ;r),\mathcal{O}_{{K}}} \widehat{\otimes}_{\mathbb{Z}_p} \Lambda_{\mathrm{cyc}}} 
\left( H^2(\mathbb{Q}_{\Sigma}/\mathbb{Q} ,
(\mathbb{T}\widehat{\otimes}_{\mathbb{Z}_p} \Lambda^\sharp_{\mathrm{cyc}}
)^\ast (1))  \right) 
\end{multline}
where $\mathcal{Z} (1)$ is the first layer of the Euler system obtained 
by Theorem \ref{theorem:eulersystem_Coleman2} and $\mathbb{Q}_{\Sigma}$ 
is the maximal extension of $\mathbb{Q}$ unramified outside the 
finite set of places $\Sigma$ of $\mathbb{Q}$ which consists of ramified 
places of the Galois representation $\mathbb{T}$ and $\{ \infty\}$. 
\end{enumerate}
\end{conj}

Our main result is as follows: 

\begin{thm}\label{theorem:IMC_nonord }
Let us assume the setting of Conjecture \ref{conjecture:IMCNonor_alakato}.
We also assume the following condition: 
\begin{enumerate}
\item[{\bf (SL)}]The image of Galois representation $G_{\mathbb{Q}} 
\longrightarrow \mathrm{Aut}_{ \Lambda_{(k_0 ;r),\mathcal{O}_{{K}}} } (\mathbb{T}) \cong GL_2 ( \Lambda_{(k_0 ;r),\mathcal{O}_{{K}}} )$ contains 
$SL_2 ( \Lambda_{(k_0 ;r),\mathcal{O}_{{K}}} )$. 
\end{enumerate} 
\par 
Then, $ H^2(\mathbb{Q}_{\Sigma}/\mathbb{Q} ,
(\mathbb{T}\widehat{\otimes}_{\mathbb{Z}_p} \Lambda^\sharp_{\mathrm{cyc}}
)^\ast (1)) $ is a torsion $\Lambda_{(k_0 ;r),\mathcal{O}_{{K}}} \widehat{\otimes}_{\mathbb{Z}_p} \Lambda_{\mathrm{cyc}}$-module and we have:  
\begin{multline}\label{equation:IMC_alaKato_nonord2}
\mathrm{char}_{\Lambda_{(k_0 ;r),\mathcal{O}_{{K}}} \widehat{\otimes}_{\mathbb{Z}_p} \Lambda_{\mathrm{cyc}}}
\left( 
H^1(\mathbb{Q}_{\Sigma}/\mathbb{Q} ,
(\mathbb{T}\widehat{\otimes}_{\mathbb{Z}_p} \Lambda^\sharp_{\mathrm{cyc}}
)^\ast (1)) 
 \Bigl/ \Lambda_{(k_0 ;r),\mathcal{O}_{{K}}} \widehat{\otimes}_{\mathbb{Z}_p} \Lambda_{\mathrm{cyc}}  \cdot \mathcal{Z}(1)  
 \right) 
 \\ 
 \subset 
 \mathrm{char}_{\Lambda_{(k_0 ;r),\mathcal{O}_{{K}}} \widehat{\otimes}_{\mathbb{Z}_p} \Lambda_{\mathrm{cyc}}} 
\left( H^2(\mathbb{Q}_{\Sigma}/\mathbb{Q} ,
(\mathbb{T}\widehat{\otimes}_{\mathbb{Z}_p} \Lambda^\sharp_{\mathrm{cyc}}
)^\ast (1))  \right) . 
\end{multline}
\end{thm}

\begin{rem}
We note that the condition \rm{(SL)} above is studied by Conti--Iovita--Tilouine \cite{cit16}. 
For example, they prove that the Lie algebra of Galois representation is full under certain regularity conditions 
(see Thm 6.2 and Cor. 7.2 of \cite{cit16}).  
\end{rem}

\begin{proof}
We apply a general machinary Theorem \ref{theorem:general_Euler_system_bound} to the Euler system obtained in Theorem \ref{theorem:eulersystem_Coleman2}. 
By our setting and the hypothesis \rm{(SL)} of the theorem, all the conditions 
(i) to (v) of Theorem \ref{theorem:general_Euler_system_bound} is checked to be true. Thus the desired statement follows immediately from Theorem \ref{theorem:general_Euler_system_bound}. 
\end{proof}

\ 
\\ 
{\bf Acknowledgements} 
\\ 
The author is grateful to Kazuma Shimomoto who read carefully 
the manuscript and pointed out a lot of mistakes, typos and 
gave useful opinions to improve the readability of the article. 
He is thankful to Filippo Nuccio for giving some comments 
on \S \ref{sec:Intro_results} and \S \ref{section:last}. 
Finally, he is grateful to the anonymous referees who read carefully 
the article and gave the author some useful comments to improve the readability of 
the article. 

\end{document}